\documentclass{elsarticle}

\usepackage{amsmath,amssymb,amsfonts,graphicx,hyperref}
\usepackage{amsthm}
\usepackage{listings}
\usepackage[margin=1in]{geometry}

\newtheorem{proposition}{Proposition}
\newtheorem{corollary}{Corollary}
\newtheorem{lemma}{Lemma}
\newtheorem{theorem}{Theorem}
\newtheorem{assumption}{Assumption}
\newtheorem{remark}{Remark}

\usepackage{amstext}
\usepackage{stmaryrd}
\usepackage{algpseudocode}

\usepackage[normalem]{ulem}
\usepackage{graphicx}

\usepackage{array}
\usepackage{verbatim}
\usepackage{booktabs}
\usepackage{calc}
\usepackage{textcomp}
\usepackage{multirow}

\usepackage[all]{xy}

\usepackage{cancel}
\usepackage{mathtools}
\usepackage{placeins}
\usepackage{xspace}

\usepackage{tikz,tikzscale}
\usetikzlibrary{calc}
\usetikzlibrary{shadings}

\usepackage[mode=multiuser,status=draft]{fixme} 
\fxsetup{envlayout=color}
\fxsetup{innerlayout=inline}
\fxsetup{targetlayout=colorcb}
\FXProvidesTargetLayout{color}

\fxusetheme{color}
\FXRegisterAuthor{aa}{envmr}{AA}
\FXRegisterAuthor{mw}{envmw}{MW}

\usepackage{subcaption}
\usepackage[format=plain,indention=.5cm]{caption}
\usepackage{pgfplots}
\pgfplotsset{compat=1.18}
\usepgfplotslibrary{fillbetween}

\usepackage{colortbl} 

\usepackage{booktabs} 

\usepackage{bm,upgreek}

\usepackage{graphicx} 
\usepackage[linesnumbered,vlined,commentsnumbered,ruled]{algorithm2e}

\definecolor{darkgreen}{rgb}{0.0, 0.5, 0.0}

\definecolor{deepgreen}{RGB}{0,100,0}

\newcommand{\plotwidth}{0.49\columnwidth}
\newcommand{\plotheight}{0.4\columnwidth}

\pgfplotscreateplotcyclelist{paper markers}{
    {teal,    solid, mark=*},
    {orange,   solid, mark=triangle*},
    {blue,     solid, mark=square*},
    {red,      solid, mark=pentagon*},
    {violet,     solid, mark=otimes*},
    {brown,    solid, mark=halfcircle*},
    {black,   solid, mark=diamond*},
}

\pgfplotsset{cycle list name=paper markers}

\pgfplotsset{
    paperplot/.style={
            width=\plotwidth,
            height=\plotheight,
            cycle list name=paper markers,
            every axis/.append style={font=\small},
            grid=major,
            thick
        }
}

\pgfplotsset{
    longplot/.style={
            width=0.8\columnwidth,
            height=\plotheight,
            cycle list name=paper markers,
            every axis/.append style={font=\small},
            grid=major,
            thick
        }
}

\title{A Unified CutFEM Formulation for Finite-Strain Elasticity: Energy Minimisation and Corner Singularities}

\author[1]{Micha\l{} Tomasz Wichrowski\corref{cor1}}
\author[1]{Ella Godiva Noomen}

\address[1]{Ruprecht-Karls-Universit\"at Heidelberg, Faculty of Mathematics and Computer Science, Germany}
\cortext[cor]{Corresponding author. \texttt{mwichro@mimuw.edu.pl}}

\begin{document}

\providecommand{\Grad}{\nabla}

\newcommand{\modelB}{a deviatoric--volumetric split neo-Hookean model}
\newcommand{\modelBshort}{split neo-Hookean}
\newcommand{\loadf}{\zeta}

\begin{abstract}
    We present a fully variational, model-independent formulation of the Cut Finite Element Method
    (CutFEM) for finite-strain elasticity.
    The discrete problem is derived from a single augmented energy functional consisting of the bulk hyperelastic energy,
    the Nitsche terms that impose the boundary conditions weakly, and the ghost-penalty stabilisation. At each nonlinear
    iterate, the residual is the exact first variation of this functional with the adaptive Nitsche weight frozen, while
    the correction uses a symmetric, coercive approximation of its Hessian. Automatic differentiation (AD) generates the
    first Piola--Kirchhoff stress tensor and the elasticity tensor directly from the scalar energy density, avoiding manual
    re-derivation when exchanging hyperelastic models. To our knowledge, this is the first unfitted finite-strain scheme
    combining an energy-only, model-independent construction with AD and an accuracy analysis at unfitted boundaries.

    Analysis of the linearised problem solved at each quasi-Newton step establishes cut-independent coercivity, continuity,
    and an $O(h^{-2})$ condition number bound, yielding a quasi-optimal convergence theorem for regular solutions through
    the Brezzi--Rappaz--Raviart framework. Numerically, the method attains optimal $h$-convergence for quadratic and
    cubic elements on a smooth test case. Furthermore, we quantify the method's accuracy limit at mixed
    Dirichlet--Neumann junctions using the Kolosov--Muskhelishvili characteristic equation. The exact solution's corner
    singularity caps the convergence rate identically for fitted and unfitted methods.
    Local mesh refinement removes this bound: we verify the recovery of optimal rates numerically for first-order elements,
    and prove that the unfitted discretisation inherits the rate the underlying refinement attains, with
    cut-independent constants.
\end{abstract}


\maketitle

\section{Introduction}
\label{sec:introduction}

Finite-strain elasticity provides the framework for a wide spectrum of materials and applications in which the
assumption of infinitesimally small displacements is no longer tenable \cite{Holzapfel_2000, ogden1997non}. Once large
deformations enter, the linear relation between stress and strain of Hooke's law is lost and the governing equations
become nonlinear, both geometrically and, in general, materially \cite{Bathe2006}. Since analytical solutions for such
problems are rarely available, one relies on numerical methods, primarily the Finite Element Method (FEM), which
discretises the weak (or, for hyperelastic materials, the variational) formulation of the problem
\cite{Zienkiewicz_Taylor_Nithiarasu_Zhu_2005, Ciarlet_1978}. Two practical difficulties, however, persistently
accompany finite element simulations at finite strain: the treatment of complex geometries, and the consistent
derivation of the linearised operators required by Newton-type solvers. This paper addresses both within a single
variational framework.

In classical, body-fitted FEM the mesh matches the geometry of the domain, so that the overall complexity of a
simulation is concentrated in the meshing procedure \cite{FUMAGALLI2019694}, which becomes expensive for intricate or
evolving geometries \cite{BOFFI2023101}. Non-matching (unfitted, immersed) methods avoid this bottleneck by embedding
the physical domain into a fixed background mesh and describing its boundary implicitly, typically by a level-set
function \cite{Burman_Claus_Hansbo_Larson_Massing_2014}. The family of such methods includes the Partition of Unity FEM
\cite{MELENK1996289}, the Extended/Generalised FEM \cite{gfemphdthesis, Fries_Belytschko_2010}, the Immersed FEM
\cite{ZHANG20042051}, the Shifted Boundary Method \cite{MAIN2018972}, and the Cut Finite Element Method (CutFEM),
introduced in \cite{HANSBO20025537} and reviewed comprehensively in \cite{Burman_Claus_Hansbo_Larson_Massing_2014,
    Burman_Hansbo_Larson_Zahedi_2025}. In this work we adopt CutFEM in its classical form: Dirichlet boundary conditions
are imposed weakly by Nitsche's method \cite{Nitsche1971}, closely related to stabilised Lagrange multiplier
formulations \cite{Babuska1973, STENBERG1995139, BURMAN20102680}, and the small-cut ill-conditioning is cured by
ghost-penalty stabilisation \cite{BURMAN20101217, Burman_Claus_Hansbo_Larson_Massing_2014}.

Although CutFEM has matured rapidly, with applications ranging from linear elasticity \cite{Hansbo_2017,
    yang2023squaresfiniteelementmethod} to fluid--structure interaction \cite{AGER2019253, BURMAN2014497, WINTER2018220}
and contact problems \cite{burman2017derivingrobustunfittedfinite, Claus2021, Claus_Kerfriden_2017}, its application to
nonlinear elasticity remains comparatively scarce. The principal contributions are \cite{Rueberg2016, Poluektov2019,
    BADIA2021114093, POLUEKTOV2022114234}, each presenting an unfitted discretisation of a finite-strain problem:
\cite{POLUEKTOV2022114234} treats fracture and contact, \cite{Poluektov2019} coupled mechanics--diffusion--reaction,
\cite{Rueberg2016} boundary-value and interface problems, and \cite{BADIA2021114093} elasto-plasticity with an
aggregation-based alternative to the ghost penalty. All of these impose Dirichlet conditions weakly in the spirit of
Nitsche, and \cite{Poluektov2019, POLUEKTOV2022114234} derive residual and tangent from a total variational
formulation. None of them, however, reports the use of automatic differentiation for the constitutive derivatives.

The second difficulty mentioned above is the derivation of the residual and tangent themselves. At finite strain these
involve the first Piola--Kirchhoff stress tensor and the fourth-order elasticity tensor, i.e. the first and second
derivatives of the strain energy density. Manual derivation is error-prone and its implementation cumbersome
\cite{Vigliotti2020}, while finite differences suffer from truncation errors; both jeopardise the consistent
linearisation on which the quadratic convergence of Newton's method depends \cite{SIMO1985101}. Automatic
differentiation (AD) offers an exact and flexible alternative: it computes derivatives of arbitrarily complex functions
to machine precision, provided they decompose into elementary operations with known derivatives
\cite{baydin2018automaticdifferentiationmachinelearning, Hillgaertner2021ADStrainEnergy}. Its use for residual and
tangent generation in nonlinear solid mechanics has been demonstrated on fitted meshes \cite{Vigliotti2020,
    Hillgaertner2021ADStrainEnergy, LI2025106354, RAMABATHIRAN20142867}, and pairs naturally with energy-based formulations
\cite{NEUNTEUFEL2021113857, SCHRODER199777}: once the material law is specified as an energy, residual and tangent
follow mechanically by differentiation, and the code becomes model-independent.

In this paper we combine these ingredients into a single, fully variational CutFEM framework for nonlinear elasticity.
Here, \emph{unified} means that the bulk constitutive response, weak boundary enforcement, and cut-cell stabilisation
are derived and differentiated within the same augmented-energy and AD workflow, independently of the material model.
The key structural observation is that not only the bulk material response, but also the Nitsche boundary terms and the
ghost-penalty stabilisation can be incorporated into one augmented energy functional. At each nonlinear iterate its
exact first variation gives the residual with the adaptive Nitsche weight frozen, while a symmetric, coercive
approximation of the Hessian gives the quasi-Newton tangent. This energy formulation has immediate practical
consequences: the correction is a descent direction for the current iterate-frozen energy, and the constitutive
derivatives (the Piola--Kirchhoff stress tensor and the fourth-order elasticity tensor) are generated by AD (here via
AceGen, following \cite{wichrowski2025matrixfreemethodsfinitestrainelasticity}). Consequently, exchanging the
hyperelastic model requires changing only the scalar energy density. To the best of our knowledge, no previous work
combines (i) an unfitted finite element discretisation, (ii) nonlinear elasticity, (iii) automatic differentiation, and
(iv) a model-independent, energy-only formulation in a single framework.

On the theoretical side, we analyse the linearised problem solved at each quasi-Newton step: we prove cut-independent
coercivity, continuity, and an $O(h^{-2})$ condition number bound under a uniform ellipticity assumption along the
quasi-Newton path, and show that each damped quasi-Newton correction is a descent direction for the current
iterate-frozen energy. Verifying the three Brezzi--Rappaz--Raviart assumptions \cite{brezzi1980finite} for the
linearised CutFEM forms (stability from the cut-independent coercivity, consistency from the interpolation estimate,
and the Lipschitz tangent on the regular branch, in the sense of \cite{brezzi1980finite}) yields a quasi-optimal
convergence theorem for regular (corner-free) solutions, subject to an explicit operator-norm bound on the quasi-Newton
Jacobian defect. We then analyse the accuracy at unfitted Dirichlet--Neumann junctions, where the regularity that
underpins this theorem fails. When such a junction lies in the \emph{interior} of a cut cell, the corner singularity
$r^\lambda$ of the exact solution \cite{grisvard1985elliptic, rossle2000corner} cannot be reproduced by a single
polynomial on the junction cell, and we prove unconditional lower bounds of $h^{\lambda}$ in energy and $h^{1+\lambda}$
in $L^2$, and derive the duality cap $h^{\min(p+1,\,2\lambda)}$ on the $L^2$ rate at a mixed junction, independently of
the polynomial degree, the Nitsche parameter, the stabilisation, the quadrature, and, by the same regularity argument,
the alignment of the junction with the background mesh. This cap is a property of the exact solution, shared with
body-fitted discretisations on quasi-uniform meshes; what is specific to the unfitted setting is that none of its
devices (exact junction quadrature, penalty tuning, vertex alignment) removes it. We quantify the dependence of
$\lambda$ on the material parameters via the Kolosov--Muskhelishvili characteristic equation, for both the
clamped--free (mixed) and clamped--clamped corner: on a straight boundary segment the first-order cap is
material-independent, while at a right-angle corner a smaller Poisson ratio mitigates but never removes the degradation
for $p \geq 2$. Finally, we show that the cap is removable by local mesh refinement and that the resulting rate is
inherited by CutFEM cut-independently: whatever rate a matching mesh attains on the same background refinement, the
unfitted discretisation attains too (Theorem~\ref{thm:cutinherit},\ref{app:gradedstab}). Consequently, the
corner-specific analysis reduces to the classical body-fitted question.

We validate the framework on two hyperelastic models, the neo-Hookean model and \modelB{}, run through the identical
code path with only the scalar energy density changed, directly demonstrating the model-independence of the energy-only
formulation. On a smooth, corner-free disc test case both models attain optimal $h$-convergence for $p = 2, 3$. On
the pole geometry we then exhibit the predicted corner behaviour: under mixed boundary conditions the
Dirichlet--Neumann junctions cap the rate at $2\lambda < 2$ for every polynomial degree, while under pure Dirichlet
conditions the milder clamped--clamped corners cap it at $1+\lambda \in (2, 2.63)$, optimal only for $p = 1$. Errors
are measured against a matching body-fitted discretisation on the same meshes, and the singular-exponent predictions
are confirmed across a range of values of the Poisson ratio in linear elasticity.
All computations use the \texttt{deal.II} library \cite{dealiipaper97}, whose non-matching module supplies the
level-set geometry and the cut-cell quadrature of \cite{Saye2015-xr}, with the AceGen-generated constitutive routines
linked into the same assembly loop for both material models.

The remainder of the paper is a round trip from formulation to discretisation, through numerics, and back to analysis
to understand and remedy the convergence cap that the numerics reveal at corners. Section~\ref{sec:formulation} casts
the finite-strain problem as energy minimisation and carries it to the CutFEM setting through level-set geometry,
energy-consistent Nitsche terms, ghost-penalty stabilisation, and AD-generated constitutive tensors.
Section~\ref{sec:theory} establishes stability, conditioning, and convergence for the resulting linearised problems.
Section~\ref{testcase} then verifies these predictions numerically: smooth domains behave in accordance with the
theoretical predictions, but corners introduce an unexpected accuracy cap. The remaining sections return to theoretical
analysis to explain and remove it. Section~\ref{sec:mixedjunction} identifies the cap as an approximation-theoretic
obstruction at unfitted Dirichlet--Neumann junctions and quantifies this effect through the Kolosov--Muskhelishvili
characteristic equation, while Section~\ref{sec:remedy} proposes local mesh grading as a remedy and confirms
numerically that it restores optimal, cut-independent convergence rates; the supporting stability analysis is proved in
\ref{app:gradedstab}.

\section{Energy-based CutFEM formulation}
\label{sec:formulation}
\label{sec:discretisationcutfem}
Let $\Omega\subset\mathbb R^d$ be the reference configuration, with boundary
$\partial\Omega=\Gamma_D\cup\Gamma_N$, let
$\bm\varphi(\bm X)=\bm X+\bm u(\bm X)$ be the deformation, and let $\bm n$ denote
the outward unit normal to $\partial\Omega$. The deformation gradient and its
determinant are
\begin{equation}
    \bm F=\Grad\bm\varphi=\bm I+\Grad\bm u,
    \qquad J:=\det\bm F,
\end{equation}
where $\Grad$ denotes the gradient with respect to $\bm X$. For a hyperelastic
material with strain-energy density $\Psi$, body force $\bm f_0$, and prescribed
traction $\bm t_0$ on $\Gamma_N$, the total potential energy is
\begin{equation}
    \label{eq:energy}
    \mathcal{E}(\bm{u}) := \int_\Omega \Psi(\bm{F}) \:\: d V - \int_\Omega \bm{f}_0\cdot \bm{u} \:\: d V  - \int_{\Gamma_N} \bm{t}_0 \cdot \bm{u} \:\: d S
    .
\end{equation}
Equilibria are stationary points of \eqref{eq:energy}, so
$\mathcal F(\bm u;\bm v)=0$ for all variations $\bm v$ that vanish on $\Gamma_D$.
At an iterate $\bar{\bm u}$, the first variation defines the residual, where
$D_{\bm v}$ denotes the Gateaux derivative in the direction $\bm v$,
\begin{align}
    \label{nonlinres}
    \mathcal F(\bar{\bm u};\bm v)
    :=D_{\bm v}\mathcal E(\bar{\bm u})
    =\int_\Omega \bm P(\bar{\bm u}):\Grad\bm v\,dV
    -\int_\Omega \bm f_0\cdot\bm v\,dV
    -\int_{\Gamma_N}\bm t_0\cdot\bm v\,dS,
\end{align}
where the first Piola--Kirchhoff stress is
\begin{equation}
    \bm P:=\frac{\partial\Psi}{\partial\bm F}.
\end{equation}
Linearising the residual in a direction $\Delta\bm u$ gives the tangent
\begin{align}
    \label{eq:tangent}
    \mathcal K(\bar{\bm u};\Delta\bm u,\bm v)
    :=D_{\Delta\bm u}\mathcal F(\bar{\bm u};\bm v)
    =\int_\Omega \Grad\Delta\bm u:\mathbb L(\bar{\bm u}):\Grad\bm v\,dV,
\end{align}
with
\begin{equation}
    \label{eq:L}
    \mathbb L:=\frac{\partial\bm P}{\partial\bm F}
    =\frac{\partial^2\Psi}{\partial\bm F\otimes\partial\bm F}.
\end{equation}
As a Hessian of $\Psi$, $\mathbb L$ has the major symmetry
$L_{ijkl}=L_{klij}$ \cite{Bonet1997-og}.
Once the scalar penalty and stabilisation parameters are fixed, the formulation depends
on the constitutive model only through $\bm P$ and $\mathbb L$. Both tensors can be
generated directly from the scalar energy $\Psi$ by automatic differentiation. We
use \href{https://www.wolfram.com/products/applications/acegen/}{AceGen} in
Mathematica to export them as pointwise kernels, following
\cite{wichrowski2025matrixfreemethodsfinitestrainelasticity}; changing the material
model therefore requires only a new expression for $\Psi$.

Throughout, elastic material constants carry the subscript $e$ -- the shear modulus $\mu_e$, Lam\'e's first constant
$\lambda_e$, the bulk modulus $\kappa_e$, and Young's modulus $E_e$ -- so that the unsubscripted $\lambda$ and $\mu$
are free for the corner singularity exponent and the mesh-grading parameter, and the unsubscripted $E$ for the
extension operator, all introduced later.

\subsection{Solving the Linearised Problem}
\label{sec:newton}
At iteration $k$, the correction $\Delta\bm u$ and the damped update with step
$\beta\in(0,1]$ are defined by
\begin{align}
    \mathcal K(\bar{\bm u}^{(k)};\Delta\bm u,\bm v)
                        & =-\mathcal F(\bar{\bm u}^{(k)};\bm v)
    \qquad \forall\bm v,                                        \\
    \label{eq:newtonupdate}
    \bar{\bm u}^{(k+1)} & =\bar{\bm u}^{(k)}+\beta\Delta\bm u,
    \qquad 0<\beta\leq1.
\end{align}
Here $\Delta\bm u$ denotes a displacement correction.\footnote{The Laplace operator
    is not used in this work; throughout, $\Delta$ denotes a Newton correction.} The
CutFEM discretisation below constructs the residual and symmetric quasi-Newton
tangent used in this update.

\subsection{Active mesh and finite-element space}
\label{nonmatchingmesh}
Let $\mathcal T^h$ be a Cartesian background mesh of cells $T$, with mesh size
$h:=\max_{T\in\mathcal T^h}\operatorname{diam}T$, generated independently of the
geometry. The reference domain is described implicitly by a level-set function
$\psi:\mathbb R^d\to\mathbb R$, negative inside the body and vanishing on its boundary,
\begin{equation}
    \Omega=\{\bm X\in\mathbb R^d:\psi(\bm X)<0\},
    \qquad \partial\Omega=\{\bm X\in\mathbb R^d:\psi(\bm X)=0\}.
\end{equation}
Until Section~\ref{sec:remedy} the mesh is quasi-uniform, so that $h$ describes
every cell; the local sizes needed for graded meshes are introduced there.
The active and cut-cell sets and the active domain are
\begin{equation}
    T_\Omega^h:=\{T\in\mathcal T^h:T\cap\Omega\neq\emptyset\},\qquad
    T_{\partial\Omega}^h:=\{T\in T_\Omega^h:T\cap\partial\Omega\neq\emptyset\},
    \qquad \Omega_T:=\bigcup_{T\in T_\Omega^h}T.
\end{equation}
We use the conforming tensor-product space
\begin{equation}
    \mathbb V_h:=\{\bm v_h\in[C^0(\Omega_T)]^d:
    \bm v_h|_T\in[\mathcal Q_p(T)]^d\ \forall T\in T_\Omega^h\},
\end{equation}
where $\mathcal Q_p(T)$ is the space of polynomials on $T$ of degree at most $p$
in each coordinate separately, and $p\geq1$ is the polynomial degree.
Figure~\ref{fig:ghostpenalty} illustrates the active mesh, cut cells, and the
ghost-penalty faces $\mathcal F_h$ defined in Section~\ref{sec:ghostpenalty}.
\begin{figure}[t]
    \centering
    \begin{tikzpicture}[scale=1.2]
        \draw[step=1cm, gray, very thin] (-0.2,0) grid (5.2,3.2);

        \node at (-0.35,2.5) {$\partial\Omega$};
        \draw[line width=2pt, red, dashed]
        plot [smooth, tension=0.5] coordinates
            {(-0.2,2.3) (1,2.2) (2,2.5) (3,2.2) (4,1.5) (5.2,1.5)};

        \node at (2.5,0.7) {$\Omega$};
        \node at (3.2,2.7) {$\mathcal F_h$};

        \foreach \x/\y in {0/2,1/2,2/2,3/1,3/2,4/1} {
                \fill[blue, opacity=0.3] (\x,\y) rectangle (\x+1,\y+1);
            }
        \foreach \x/\y in {0/0,0/1,1/0,1/1,2/0,2/1,3/0,4/0} {
                \fill[green, opacity=0.3] (\x,\y) rectangle (\x+1,\y+1);
            }

        \foreach \x/\y in {0/2,1/2,2/2,2/1,3/1,-1/2,4/1} {
                \draw[blue, line width=1.5pt] (\x+1,\y) -- (\x+1,\y+1);
            }
        \foreach \x/\y in {0/1,1/1,2/1,3/0,4/0,3/1} {
                \draw[blue, line width=1.5pt] (\x,\y+1) -- (\x+1,\y+1);
            }
    \end{tikzpicture}
    \caption{Active background mesh near the immersed boundary. Green cells lie
        inside $\Omega$, blue cells are cut by $\partial\Omega$, and the highlighted
        faces belong to $\mathcal F_h$.}
    \label{fig:ghostpenalty}
\end{figure}
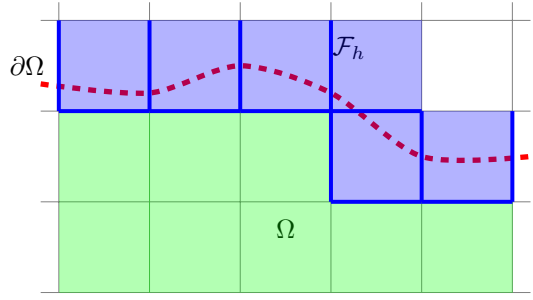

\subsection{Boundary Conditions}
\label{bcconditions}
Dirichlet conditions on the immersed boundary are imposed by augmenting the energy
with symmetric Nitsche terms \cite{Nitsche1971,Benzaken2024}. Writing the Dirichlet
portion as $\partial\Omega$, that is, writing the formulation and its analysis for the pure-Dirichlet case
$\Gamma_D=\partial\Omega$ so that no separate symbol is needed, the iterate-frozen functional for homogeneous data is
\begin{equation}
    \label{eq:nitscheenergy}
    \mathcal{E}_{h}(\bm{u};\bar{\bm{u}}) := \mathcal{E}(\bm{u})
    - \int_{\partial\Omega} \bm{n} \cdot \bm{P}(\bm{u}) \cdot \bm{u} \: dS
    + \int_{\partial\Omega} \frac{\gamma_D \, \eta(\bar{\bm{u}})}{2h} \, |\bm{u}|^2 \: dS ,
\end{equation}
where the penalty weight $\eta(\bar{\bm u})$, defined in \eqref{eq:penaltyweight}
below, is evaluated at the current iterate and held fixed under variation. The
semicolon records this: \eqref{eq:nitscheenergy} is not one functional but a
family indexed by the iterate $\bar{\bm u}$, which enters as a fixed datum at each
quasi-Newton step. Variations are always taken with respect to the first argument
only, with $\eta(\bar{\bm u})$ treated as a constant, so the residual below is the
exact first variation of that member of the family. For mixed conditions the integrals are restricted to $\Gamma_D$;
inhomogeneous data $\bm g_D$ are treated by replacing the undifferentiated boundary
trace $\bm u$ by $\bm u-\bm g_D$. The boundary residual is
\begin{align}
    \mathcal F_{\mathrm{Bound}}(\bar{\bm u},\bm v_h)
    ={} & -\int_{\partial\Omega}\bm n\cdot\bm P(\bar{\bm u})\cdot\bm v_h\,dS \notag \\
    \label{eq:fbound}
        & -\int_{\partial\Omega}\bm n\cdot
    (\mathbb L(\bar{\bm u}):\Grad\bm v_h)\cdot\bar{\bm u}\,dS
    +\int_{\partial\Omega}\frac{\gamma_D\eta(\bar{\bm u})}{h}
    \bar{\bm u}\cdot\bm v_h\,dS .
\end{align}
The assembled symmetric boundary tangent is
\begin{align}
    \mathcal K_{\mathrm{Bound}}(\bar{\bm u};\bm u_h,\bm v_h)
    ={} & -\int_{\partial\Omega}\bm n\cdot
    (\mathbb L(\bar{\bm u}):\Grad\bm u_h)\cdot\bm v_h\,dS \notag \\
    \label{eq:kbound}
        & -\int_{\partial\Omega}\bm n\cdot
    (\mathbb L(\bar{\bm u}):\Grad\bm v_h)\cdot\bm u_h\,dS
    +\int_{\partial\Omega}\frac{\gamma_D\eta(\bar{\bm u})}{h}
    \bm u_h\cdot\bm v_h\,dS .
\end{align}
These terms are the finite-strain counterpart of the symmetric Nitsche forms in
\cite{BADIA2021114093,Hansbo_2017,Sticko2020}, with the elasticity tensor evaluated
at the current iterate.

\paragraph{Remark (symmetrised tangent)} The exact derivative of \eqref{eq:fbound} also contains $-\int_{\partial\Omega}\bm n\cdot((D_{\bm u_h}\mathbb
    L):\Grad\bm v_h) \cdot\bar{\bm u}\,dS$, which involves the third derivative of $\Psi$. We omit it from
\eqref{eq:kbound}; the residual and its fixed points are unchanged, while the tangent remains symmetric. Thus the
method is a symmetric quasi-Newton scheme on each fixed mesh. Lemma~\ref{lem:defect} bounds this Jacobian defect by
$O(h^{p-d/2})$ in the Newton neighbourhood.

\paragraph{Remark (penalty weight)} We take $\gamma_D=5(p+1)p$ and
\begin{equation}
    \label{eq:penaltyweight}
    \eta(\bar{\bm{u}}) := \max\left\{ 2\mu_e , \; \tfrac{1}{2}\,\|\mathbb{L}(\bar{\bm{u}})\|_1 \right\},
    \qquad \|\mathbb{L}\|_1 := \sum_{i,j,k,l} |L_{ijkl}| ,
\end{equation}
where $\mu_e$ is the shear modulus. The componentwise norm is a conservative,
inexpensive bound available directly from the AD-generated tensor; it avoids a
spectral problem at every boundary quadrature point. It tracks the local finite-strain
stiffness, while the floor keeps the weight positive in softened states. Because the
maximum and componentwise norm are nonsmooth, $\eta$ is frozen under variation; its
omitted derivative is included in the defect analysed in Lemma~\ref{lem:defect}.

\paragraph{Remark (fixed weight)} A state-independent value $\eta_{\mathrm{ref}}\geq\max\{2\mu_e,\tfrac12\sup\|\mathbb L\|_1\}$ over the expected
deformation range would make \eqref{eq:nitscheenergy} one fixed functional. It must, however, represent the stiffest
state and therefore over-penalises softer states. We retain the adaptive, iterate-frozen weight
\eqref{eq:penaltyweight}.

\subsection{Ghost Penalty Stabilisation}
\label{sec:ghostpenalty}
Arbitrarily small intersections $|T\cap\Omega|$ lead to cut-dependent conditioning.
The face ghost penalty restores control on the active mesh by penalising jumps of
normal derivatives near cut cells \cite{BURMAN20101217,Sticko2020}. The
ghost-penalty faces are the interior faces of the active mesh having at least one
cut neighbour,
\begin{equation}
    \mathcal F_h:=\{F=\overline T_+\cap\overline T_-:
    T_+\in T_{\partial\Omega}^h,\ T_-\in T_\Omega^h,\ T_+\neq T_-,\
    |F|_{d-1}>0\},
\end{equation}
where $|F|_{d-1}>0$ selects the pairs sharing a full $(d-1)$-dimensional face and
excludes contacts along lower-dimensional edges or vertices. For such a face,
$h_F:=\operatorname{diam}F$ is the face size, $\bm n_F$ is a unit normal to $F$
(either orientation; the form below is independent of the choice), and the jump of
a quantity $q$ across $F$ is the difference of its one-sided traces from the two
neighbours,
\begin{equation}
    \label{eq:jump}
    [\![q]\!](\bm X):=\lim_{s\to0^+}\big(q(\bm X-s\bm n_F)-q(\bm X+s\bm n_F)\big),
    \qquad \bm X\in F .
\end{equation}
Writing $\partial_{n_F}^k:=(\bm n_F\cdot\Grad)^k$ for the $k$-th normal derivative
and $(\cdot,\cdot)_F$ for the $L^2(F)$ inner product, define the face form
\begin{equation}
    \label{ghost}
    g_F(\bm w,\bm v):=\sum_{k=1}^p\frac{h_F^{2k-1}}{(k!)^2}
    \big([\![\partial_{n_F}^k\bm w]\!],[\![\partial_{n_F}^k\bm v]\!]\big)_F .
\end{equation}
The jumps of the tangential derivatives vanish for $\bm w\in\mathbb V_h$, so only
the normal derivatives up to order $p$ carry information; each summand is
quadratic in the jump, hence unaffected by the sign convention in
\eqref{eq:jump}.
The stabilisation energy is
\begin{equation}
    \label{eq:gpenergy}
    \mathcal E_{\mathrm{Stab}}(\bm u):=
    \frac{\gamma_A}{2}\sum_{F\in\mathcal F_h}g_F(\bm u,\bm u),
    \qquad \gamma_A=\alpha E_e>0 .
\end{equation}
Here $E_e$ is Young's modulus and $\alpha>0$ is dimensionless.
Its second and first variations are, respectively,
\begin{align}
    \label{eq:kstab}
    \mathcal K_{\mathrm{Stab}}(\bar{\bm u};\bm u_h,\bm v_h)
     & =\gamma_A\sum_{F\in\mathcal F_h}g_F(\bm u_h,\bm v_h),     \\
    \label{eq:fstab}
    \mathcal F_{\mathrm{Stab}}(\bar{\bm u},\bm v_h)
     & =\gamma_A\sum_{F\in\mathcal F_h}g_F(\bar{\bm u},\bm v_h).
\end{align}
\subsection{Discretisation and Algorithm Summary}
At a fixed iterate $\bar{\bm u}$, all residual contributions are the first variation of the augmented energy
\begin{equation}
    \label{eq:totalenergy}
    \mathcal{E}_{\text{Total}}(\bm{u}_h;\bar{\bm{u}}) := \mathcal{E}(\bm{u}_h)
    - \int_{\partial\Omega} \bm{n} \cdot \bm{P}(\bm{u}_h) \cdot \bm{u}_h \: dS
    + \int_{\partial\Omega} \frac{\gamma_D \, \eta(\bar{\bm{u}})}{2h} \, |\bm{u}_h|^2 \: dS
    + \mathcal{E}_{\text{Stab}}(\bm{u}_h),
\end{equation}
with $\eta(\bar{\bm u})$ frozen; this is the Nitsche functional~\eqref{eq:nitscheenergy}
augmented by the stabilisation energy,
$\mathcal E_{\text{Total}}=\mathcal E_h+\mathcal E_{\text{Stab}}$. The functional
changes between iterations only through this weight, which settles to
$\eta(\bm u_h)$ at convergence. The residual and assembled symmetric tangent are
\begin{align}
    \mathcal{K}_{\text{Total}}(\bar{\bm{u}}; \bm{u}_h, \bm{v}_h) & = \mathcal{K}_{\Omega}(\bar{\bm{u}}; \bm{u}_h, \bm{v}_h) + \mathcal{K}_{\text{Bound}}(\bar{\bm{u}}; \bm{u}_h, \bm{v}_h) + \mathcal{K}_{\text{Stab}}(\bar{\bm{u}}; \bm{u}_h, \bm{v}_h) \\
    \mathcal{F}_{\text{Total}}(\bar{\bm{u}}, \bm{v}_h)           & = \mathcal{F}_{\Omega}(\bar{\bm{u}}, \bm{v}_h) + \mathcal{F}_{\text{Bound}}(\bar{\bm{u}}, \bm{v}_h) + \mathcal{F}_{\text{Stab}}(\bar{\bm{u}}, \bm{v}_h)
\end{align}
where $\mathcal K_\Omega:=\mathcal K$ and $\mathcal F_\Omega:=\mathcal F$ are the
bulk forms~\eqref{eq:tangent} and~\eqref{nonlinres}, with the integrals taken over
$\Omega$ and $\Gamma_N$.
Because it omits the two derivatives identified above, $\mathcal K_{\mathrm{Total}}$
is a symmetric approximation of the second variation; the resulting iteration is
an energy-based symmetric quasi-Newton method.
The quasi-Newton correction solves
$\mathcal K_{\mathrm{Total}}(\bar{\bm u};\Delta\bm u,\bm v_h)
    =-\mathcal F_{\mathrm{Total}}(\bar{\bm u},\bm v_h)$ for all
$\bm v_h\in\mathbb V_h$, followed by \eqref{eq:newtonupdate}.

\section{Stability and accuracy}
\label{sec:theory}
In this section we analyse the linearised problem solved at each quasi-Newton step. Throughout, $\bar{\bm{u}} \in \mathbb{V}_h$ denotes a fixed iterate, $\Omega_T := \bigcup_{T \in T^h_\Omega} T$ the domain covered by the active mesh, and $\|\cdot\|_{D}$ the $L^2$-norm on a set $D$. We write the linearised tangent form compactly as
\begin{equation}
    A_h(\bm{u}_h, \bm{v}_h) := \mathcal{K}_{\text{Total}}(\bar{\bm{u}}; \bm{u}_h, \bm{v}_h),
\end{equation}
and measure stability in the mesh-dependent norm
\begin{equation}
    \label{eq:triplenorm}
    |||\bm{v}|||^2 := \|\Grad \bm{v}\|^2_{\Omega} + \|h^{-1/2}\bm{v}\|^2_{\partial\Omega} + |\bm{v}|_s^2,
    \qquad |\bm{v}|_s^2 := \gamma_A \sum_{F \in \mathcal{F}_h} g_F(\bm{v}, \bm{v}),
\end{equation}
where the ghost-penalty seminorm $|\cdot|_s$ coincides with $\mathcal{K}_{\text{Stab}}(\bar{\bm{u}}; \bm{v}, \bm{v})$ of \eqref{eq:kstab}. Note that the stabilisation form does not depend on the iterate $\bar{\bm{u}}$, so the norm is iterate-independent.
The analysis rests on a single assumption on the structure of the material response at the iterate.

\begin{assumption}[Uniform ellipticity along the quasi-Newton path]
    \label{ass:ellipticity}
    \leavevmode
    \begin{itemize}
        \item[(a)] \emph{(Uniform pointwise strain-ellipticity.)} \\
              There exist constants $0 < c_L \leq C_L$ such that
              \begin{equation}
                  c_L \,|\operatorname{sym}\bm{M}|^2 \;\leq\; \bm{M} : \mathbb{L}(\bar{\bm{u}}(\bm{X})) : \bm{M} \;\leq\; C_L |\bm{M}|^2
                  \qquad \forall \bm{M} \in \mathbb{R}^{d \times d}, \; \bm{X} \in \Omega_T ,
              \end{equation}
              where $\operatorname{sym}\bm M := \tfrac12(\bm M + \bm M^\top)$.
        \item[(b)] \emph{(Non-degeneracy and smoothness.)} \\
              $\Psi \in C^3$ in a neighbourhood of the deformation
              gradients attained on $\Omega_T$, and the Jacobian stays bounded away from degeneracy,
              $J(\bar{\bm{u}}(\bm{X})) \geq J_0$ on $\Omega_T$ for some constant $J_0 > 0$, so that the third derivative is bounded there,
              $\sup_{\Omega_T} |\partial^3\Psi| =: C_{L'} < \infty$.
    \end{itemize}
\end{assumption}
The two parts have different status. Here, a \emph{branch of solutions} means a smooth, continuously parametrised
family of exact solutions $\loadf \mapsto \bm{u}(\loadf)$, where $\loadf$ is the load factor (a loading or continuation parameter); a
point on the branch is \emph{regular} when the tangent operator is an isomorphism there.
For the smooth energies used below, part~(b) holds on compact branch segments away from material collapse and provides
the bulk Lipschitz modulus $C_{L'}$ used for assumption~(J1) of Theorem~\ref{thm:brr}.

Part~(a) is posed on symmetric strains because full-gradient ellipticity is incompatible with frame indifference: at a
stress-free state, $\bm M:\mathbb L:\bm M=0$ for every skew-symmetric $\bm M$. The stated strain bound holds with
$c_L=2\mu_e$ for isotropic linear elasticity and at stress-free states of the hyperelastic models used below, while
Korn's inequality \eqref{eq:korn} restores full-gradient control in the coercivity proof.

Pointwise ellipticity is nevertheless more restrictive than branch regularity: part~(a) implies regularity but is not
implied by it, and neither does it follow from polyconvexity. Regularity is a global variational property, compatible
with pointwise indefiniteness of $\mathbb L$; part~(a) demands pointwise positive definiteness on symmetric strains at
every point of $\Omega_T$, including the exterior strip $\Omega_T\setminus\Omega$. Together with \eqref{eq:korn}, this assumption yields cut-independent coercivity
(Proposition~\ref{prop:coercivity}), an SPD tangent, the $O(h^{-2})$ condition-number bound, and the descent property
of Lemma~\ref{lem:descent}. It does, however, narrow the range of situations the analysis covers: under severe
compression or near buckling, $\mathbb L$ may lose pointwise positive definiteness while the solution branch remains
regular; the present analysis then ceases to apply, although Newton's method may still converge.

We further recall three standard ingredients of the analysis, with positive constants $C_T$, $C_G$ and $C_K$, the
first two of them independent of how $\partial\Omega$ cuts the mesh. First, the inverse trace inequality for polynomials on cut cells \cite{HANSBO20025537,
    massing2014stabilized}: for all $T \in T^h_{\partial\Omega}$ and every function $\bm{w}_h$ that is polynomial of degree
at most $p$ on $T$,
\begin{equation}
    \label{eq:cuttrace}
    h \, \| \bm{w}_h \|^2_{\partial\Omega \cap T} \leq C_T \, \|\bm{w}_h\|^2_{T} ;
\end{equation}
we will apply it to both $\bm{w}_h = \bm{v}_h$ and $\bm{w}_h = \Grad \bm{v}_h$. Note that the right-hand side involves the \emph{full} cell $T$, not only $T \cap \Omega$; this is precisely where the ghost penalty enters, through the second ingredient, the extension property \cite{BURMAN20101217, Burman_Claus_Hansbo_Larson_Massing_2014, Burman_Hansbo_Larson_Zahedi_2025}:
\begin{equation}
    \label{eq:extension}
    \|\Grad \bm{v}_h\|^2_{\Omega_T} \leq C_G \left( \|\Grad \bm{v}_h\|^2_{\Omega} + \gamma_A^{-1} |\bm{v}_h|_s^2 \right)
    \qquad \forall \bm{v}_h \in \mathbb{V}_h ,
\end{equation}
which states that the ghost penalty controls the part of the gradient that lives outside the physical domain.
The third ingredient converts the strain control provided by Assumption~\ref{ass:ellipticity}(a) into control of the
full gradient: since the discrete displacements are $H^1$-conforming, Korn's inequality with a boundary term holds on
the fixed Lipschitz domain $\Omega$,
\begin{equation}
    \label{eq:korn}
    \|\Grad \bm{w}\|^2_{\Omega} \leq C_K \left( \|\operatorname{sym}\Grad \bm{w}\|^2_{\Omega} + \|\bm{w}\|^2_{\partial\Omega} \right)
    \qquad \forall \bm{w} \in \big(H^1(\Omega)\big)^d ,
\end{equation}
with $C_K$ depending only on $\Omega$; it follows from Korn's second inequality \cite{Braess2013-mg} by the usual
compactness (Petree--Tartar) argument, the only rigid motion with vanishing trace on $\partial\Omega$ being zero (cf.
its use in fitted Nitsche formulations of elasticity \cite{Hansbo_2017}). Since $\partial\Omega$ is covered by cut
cells, the boundary term is dominated by the Nitsche penalty scaling, $\|\bm{w}\|^2_{\partial\Omega} \leq h_0
    \,\|h^{-1/2}\bm{w}\|^2_{\partial\Omega}$ for $h \leq h_0$. Here $h_0 > 0$ is a reference mesh size, fixed once and
for all and small enough that the active domain $\Omega_T$ lies in a fixed neighbourhood of $\Omega$; all statements
below are understood for $h \leq h_0$. Being posed on the physical domain alone, \eqref{eq:korn}
is independent of the mesh and of the cut.

\begin{proposition}[Cut-independent coercivity and continuity]
    \label{prop:coercivity}
    Let Assumption~\ref{ass:ellipticity} hold, and set $\eta_{\min} := \inf_{\partial\Omega} \eta(\bar{\bm{u}}) \geq 2\mu_e$ and $\eta_{\max} := \sup_{\partial\Omega} \eta(\bar{\bm{u}})$, with the penalty weight $\eta$ as in \eqref{eq:penaltyweight}. There exists $\gamma_0 > 0$, depending only on $C_T$, $C_G$, $C_K$, $h_0$, $c_L$, $C_L$, $\eta_{\min}$, $\gamma_A/E_e$ and the polynomial degree $p$ (through the combination $C_L^2 / (\min\{c_L,1\}\,\eta_{\min})$ and the trace/extension/Korn constants), such that for all $h \leq h_0$ and $\gamma_D \geq \gamma_0$,
    \begin{align}
        A_h(\bm{v}_h, \bm{v}_h) & \geq c_s \, |||\bm{v}_h|||^2                 &  & \forall \bm{v}_h \in \mathbb{V}_h ,           \\
        A_h(\bm{u}_h, \bm{v}_h) & \leq C_s \, |||\bm{u}_h||| \, |||\bm{v}_h||| &  & \forall \bm{u}_h, \bm{v}_h \in \mathbb{V}_h ,
    \end{align}
    where $c_s, C_s > 0$ are independent of $h$ and of the position of $\partial\Omega$ relative to the mesh.
\end{proposition}

\begin{proof}
    By definition of $A_h$ and symmetry of the two Nitsche terms in \eqref{eq:kbound},
    \begin{equation*}
        A_h(\bm{v}_h, \bm{v}_h) = \int_\Omega \Grad\bm{v}_h : \mathbb{L}(\bar{\bm{u}}) : \Grad\bm{v}_h \, dV
        - 2 \int_{\partial\Omega} \bm{n}\cdot(\mathbb{L}(\bar{\bm{u}}):\Grad\bm{v}_h)\cdot\bm{v}_h \, dS
        + \gamma_D \int_{\partial\Omega} \frac{\eta(\bar{\bm{u}})}{h} |\bm{v}_h|^2 \, dS + |\bm{v}_h|_s^2 .
    \end{equation*}
    The first term is bounded below by $c_L \|\operatorname{sym}\Grad\bm{v}_h\|^2_\Omega$ by Assumption~\ref{ass:ellipticity}(a), and the penalty term by $\gamma_D \eta_{\min} \|h^{-1/2}\bm{v}_h\|^2_{\partial\Omega}$. For the boundary term, Cauchy--Schwarz, the trace inequality \eqref{eq:cuttrace} applied cell-wise, and Young's inequality with $\varepsilon > 0$ give
    \begin{equation*}
        2 \left| \int_{\partial\Omega} \bm{n}\cdot(\mathbb{L}:\Grad\bm{v}_h)\cdot\bm{v}_h \, dS \right|
        \leq 2 C_L \, \|h^{1/2}\Grad\bm{v}_h\|_{\partial\Omega} \|h^{-1/2}\bm{v}_h\|_{\partial\Omega}
        \leq \varepsilon C_L C_T \|\Grad\bm{v}_h\|^2_{\Omega_T} + \frac{C_L}{\varepsilon} \|h^{-1/2}\bm{v}_h\|^2_{\partial\Omega} .
    \end{equation*}
    The term $\|\Grad\bm{v}_h\|^2_{\Omega_T}$ is absorbed using the extension property \eqref{eq:extension} followed by
    Korn's inequality \eqref{eq:korn} with $\|\bm{v}_h\|^2_{\partial\Omega} \leq h_0 \|h^{-1/2}\bm{v}_h\|^2_{\partial\Omega}$:
    \begin{align*}
        \varepsilon C_L C_T \|\Grad\bm{v}_h\|^2_{\Omega_T}
         & \leq \varepsilon C_L C_T C_G \left( \|\Grad\bm{v}_h\|^2_{\Omega} + \gamma_A^{-1} |\bm{v}_h|_s^2 \right)                                \\
         & \leq \varepsilon C_L C_T C_G \left( C_K \|\operatorname{sym}\Grad\bm{v}_h\|^2_{\Omega}
        + C_K h_0 \|h^{-1/2}\bm{v}_h\|^2_{\partial\Omega} + \gamma_A^{-1} |\bm{v}_h|_s^2 \right) .
    \end{align*}
    Choosing $\varepsilon$ small enough that $\varepsilon C_L C_T C_G \max\{C_K, C_K h_0, \gamma_A^{-1}\} \leq \tfrac12\min\{c_L, 1\}$, and then $\gamma_D \geq \gamma_0 := 4 C_L / (\eta_{\min} \varepsilon)$ (so that both $C_L/\varepsilon$ and the absorbed $C_K h_0$-contribution are covered by half the penalty term), all negative contributions are absorbed, leaving
    \begin{equation*}
        A_h(\bm{v}_h, \bm{v}_h) \geq \tfrac{c_L}{2} \|\operatorname{sym}\Grad\bm{v}_h\|^2_{\Omega} + \tfrac{\gamma_D\eta_{\min}}{2} \|h^{-1/2}\bm{v}_h\|^2_{\partial\Omega} + \tfrac12 |\bm{v}_h|_s^2 .
    \end{equation*}
    A final application of \eqref{eq:korn}, again with the boundary term dominated by the penalty term, converts the
    strain term into $\|\Grad\bm{v}_h\|^2_{\Omega}$ and yields coercivity in the triple norm \eqref{eq:triplenorm} with
    $c_s$ depending only on $c_L$, $\gamma_D\eta_{\min}$, $C_K$ and $h_0$.

    For continuity, we bound each contribution to $A_h(\bm{u}_h, \bm{v}_h)$ separately (the two Nitsche terms share the
    same bound). The volume term satisfies, by Assumption~\ref{ass:ellipticity} and Cauchy--Schwarz,
    \begin{equation*}
        \left| \int_\Omega \Grad\bm{u}_h : \mathbb{L} : \Grad\bm{v}_h \, dV \right| \leq C_L \|\Grad\bm{u}_h\|_\Omega \|\Grad\bm{v}_h\|_\Omega \leq C_L \, |||\bm{u}_h||| \, |||\bm{v}_h||| .
    \end{equation*}
    For each of the two Nitsche terms, Cauchy--Schwarz, the trace inequality \eqref{eq:cuttrace} applied to $\Grad\bm{u}_h$ cell-wise, and the extension property \eqref{eq:extension} give
    \begin{align*}
        \left| \int_{\partial\Omega} \bm{n}\cdot(\mathbb{L}:\Grad\bm{u}_h)\cdot\bm{v}_h \, dS \right|
         & \leq C_L \, \|h^{1/2}\Grad\bm{u}_h\|_{\partial\Omega} \, \|h^{-1/2}\bm{v}_h\|_{\partial\Omega}
        \leq C_L \sqrt{C_T} \, \|\Grad\bm{u}_h\|_{\Omega_T} \, \|h^{-1/2}\bm{v}_h\|_{\partial\Omega}      \\
         & \leq C_L \sqrt{C_T C_G \max\{1, \gamma_A^{-1}\}} \; |||\bm{u}_h||| \, |||\bm{v}_h||| ,
    \end{align*}
    and the same bound with the roles of $\bm{u}_h$ and $\bm{v}_h$ exchanged. The penalty term is bounded by
    $$\gamma_D \eta_{\max} \|h^{-1/2}\bm{u}_h\|_{\partial\Omega} \|h^{-1/2}\bm{v}_h\|_{\partial\Omega} \leq \gamma_D \eta_{\max} \, |||\bm{u}_h||| \, |||\bm{v}_h|||,$$
    and the ghost penalty term by $|\bm{u}_h|_s |\bm{v}_h|_s \leq |||\bm{u}_h||| \, |||\bm{v}_h|||$, both directly by Cauchy--Schwarz. Summing these bounds yields the claim with
    \begin{equation*}
        C_s = C_L + 2 C_L \sqrt{C_T C_G \max\{1, \gamma_A^{-1}\}} + \gamma_D \eta_{\max} + 1 .
    \end{equation*}
\end{proof}

Proposition~\ref{prop:coercivity} makes the choice of the adaptive weight \eqref{eq:penaltyweight} precise: the
coercivity threshold is $\gamma_0 = 4C_L/(\eta_{\min}\varepsilon)$, and since $\eta(\bar{\bm{u}}) \geq
    \tfrac{1}{2}\|\mathbb{L}(\bar{\bm{u}})\|_1$ tracks the local stiffness, the ratio $C_L/\eta_{\min}$, and hence the
required $\gamma_D$, remains $O(1)$ even when $\mathbb{L}(\bar{\bm{u}})$ departs substantially from its small-strain
magnitude. A fixed weight scaling with Young's modulus alone would instead require $\gamma_D \gtrsim
    \|\mathbb{L}(\bar{\bm{u}})\|/E_e$.

\begin{corollary}[Well-posedness of the quasi-Newton step]
    \label{cor:wellposed}
    Under the assumptions of Proposition~\ref{prop:coercivity}, each linearised problem
    $\mathcal{K}_{\text{Total}}(\bar{\bm{u}}; \Delta\bm{u}, \bm{v}_h) = -\mathcal{F}_{\text{Total}}(\bar{\bm{u}}, \bm{v}_h)$
    admits a unique solution $\Delta\bm{u} \in \mathbb{V}_h$, satisfying $|||\Delta\bm{u}||| \leq c_s^{-1} \sup_{\bm{v}_h} \mathcal{F}_{\text{Total}}(\bar{\bm{u}}, \bm{v}_h)/|||\bm{v}_h|||$, with stability constant independent of the cut configuration.
\end{corollary}

\begin{proof}
    Follows directly from Lax--Milgram on the finite-dimensional space $\mathbb{V}_h$ equipped with $|||\cdot|||$.
\end{proof}

The proof of the condition number bound follows the argument of \cite{BURMAN20102680,
    Burman_Hansbo_Larson_Zahedi_2025}, with the elasticity form in place of the Laplacian. It requires three discrete
inequalities relating the triple norm \eqref{eq:triplenorm}, the $L^2$-norm on the active mesh, and the Euclidean norm
of the coefficient vector; we collect them in the following lemma. Here and below, $\bm{V} \in \mathbb{R}^N$ denotes
the coefficient vector of $\bm{v}_h \in \mathbb{V}_h$ with respect to the nodal basis, and $|\bm{V}|$ its Euclidean
norm.

\begin{lemma}[Discrete inequalities on the active mesh]
    \label{lem:discrete}
    Let the background mesh be quasi-uniform and $h \leq h_0$. There exist constants $C_{\mathrm{inv}}, C_P, c_M, C_M > 0$, independent of $h$ and of the position of $\partial\Omega$ relative to the mesh, such that for all $\bm{v}_h \in \mathbb{V}_h$:
    \begin{align}
        \label{eq:lemglobalinverse}
        \textnormal{(i)}   & \quad |||\bm{v}_h|||^2 \leq C_{\mathrm{inv}} \, h^{-2} \, \|\bm{v}_h\|^2_{\Omega_T} ,             \\
        \label{eq:lempoincare}
        \textnormal{(ii)}  & \quad \|\bm{v}_h\|_{\Omega_T} \leq C_P \, |||\bm{v}_h||| ,                                        \\
        \label{eq:lemmassequiv}
        \textnormal{(iii)} & \quad c_M \, h^{d} \, |\bm{V}|^2 \leq \|\bm{v}_h\|^2_{\Omega_T} \leq C_M \, h^{d} \, |\bm{V}|^2 .
    \end{align}
\end{lemma}

\begin{proof}
    \textit{(iii)} is the standard norm equivalence for nodal bases on shape-regular, quasi-uniform meshes, see e.g. \cite{Braess2013-mg}; it is unaffected by the cut because the norm is taken over the \emph{full} cells of the active mesh.

    \textit{(i)} We bound the three contributions to $|||\bm{v}_h|||^2$ separately, each by inequalities posed on full cells, hence with cut-independent constants. For the gradient term, the standard element-wise inverse inequality $\|\Grad\bm{v}_h\|_{T} \leq C h^{-1} \|\bm{v}_h\|_{T}$ \cite{Braess2013-mg} gives $\|\Grad\bm{v}_h\|^2_{\Omega} \leq \|\Grad\bm{v}_h\|^2_{\Omega_T} \leq C h^{-2} \|\bm{v}_h\|^2_{\Omega_T}$. For the boundary term, the cut trace inequality \eqref{eq:cuttrace} applied to $\bm{w}_h = \bm{v}_h$ on each cut cell yields
    \begin{equation*}
        \|h^{-1/2}\bm{v}_h\|^2_{\partial\Omega} = \sum_{T \in T^h_{\partial\Omega}} h^{-1} \|\bm{v}_h\|^2_{\partial\Omega \cap T}
        \leq C_T h^{-2} \sum_{T \in T^h_{\partial\Omega}} \|\bm{v}_h\|^2_{T} \leq C_T h^{-2} \|\bm{v}_h\|^2_{\Omega_T} .
    \end{equation*}
    For the ghost penalty term, each face contribution in \eqref{ghost} is bounded by the triangle inequality, the (uncut) trace inverse inequality $\|\bm{w}_h\|^2_{\partial T} \leq C h^{-1} \|\bm{w}_h\|^2_{T}$ on the faces of $T_\pm$, and
    the $k$-th order element-wise inverse inequality $\|D^k \bm{v}_h\|_{T} \leq C_k h^{-k} \|\bm{v}_h\|_{T}$:
    \begin{equation*}
        \frac{h^{2k-1}}{(k!)^2} \big\| [\partial^k_n \bm{v}_h] \big\|^2_F
        \leq \frac{2 h^{2k-1}}{(k!)^2} \sum_{T \in \{T_+, T_-\}} \|\partial^k_n \bm{v}_h\|^2_{\partial T}
        \leq C h^{2k-2} \sum_{T \in \{T_+, T_-\}} \|D^k \bm{v}_h\|^2_{T}
        \leq C_k \, h^{-2} \sum_{T \in \{T_+, T_-\}} \|\bm{v}_h\|^2_{T} .
    \end{equation*}
    Summing over $k = 1, \dots, p$ and over $F \in \mathcal{F}_h$, and noting that each cell appears in at most $2d$ faces, gives $|\bm{v}_h|^2_s \leq C \gamma_A h^{-2} \|\bm{v}_h\|^2_{\Omega_T}$. Adding the three bounds proves (i).

    \textit{(ii)} Since $\partial\Omega$ is Lipschitz and $\Omega_T$ is contained in a fixed neighbourhood of $\Omega$ of width at most $2 h_0$, the Friedrichs inequality with boundary trace control holds on $\Omega_T$ \cite{Evans2022-bl, BURMAN20102680}: there exists $C_{P,0}$, depending only on $\Omega$ and $h_0$, such that
    \begin{equation*}
        \|\bm{w}\|_{\Omega_T} \leq C_{P,0} \left( \|\Grad\bm{w}\|_{\Omega_T} + \|\bm{w}\|_{\partial\Omega} \right) \qquad \forall \bm{w} \in \big(H^1(\Omega_T)\big)^d .
    \end{equation*}
    Applying this to $\bm{v}_h$, the gradient on the right-hand side is controlled through the extension property \eqref{eq:extension}, $\|\Grad\bm{v}_h\|_{\Omega_T} \leq \sqrt{C_G \max\{1, \gamma_A^{-1}\}} \, |||\bm{v}_h|||$, and the trace term through $\|\bm{v}_h\|_{\partial\Omega} = h^{1/2} \|h^{-1/2}\bm{v}_h\|_{\partial\Omega} \leq h_0^{1/2} \, |||\bm{v}_h|||$. Both constants are independent of the cut, which proves (ii).
\end{proof}

\begin{corollary}[Cut-independent conditioning]
    \label{cor:conditioning}
    Under the assumptions of Proposition~\ref{prop:coercivity} and Lemma~\ref{lem:discrete}, the tangent matrix satisfies
    \begin{equation}
        \operatorname{cond}(\mathbf{K}_{\text{Total}}) \leq C h^{-2},
    \end{equation}
    with $C$ independent of the position of $\partial\Omega$ relative to the mesh.
\end{corollary}

\begin{proof}
    By Proposition~\ref{prop:coercivity} the matrix $\mathbf{K}_{\text{Total}}$ is symmetric positive definite, so its condition number is the ratio of the extreme eigenvalues, which we estimate through the Rayleigh quotient. Let $\bm{v}_h \in \mathbb{V}_h \setminus \{\bm{0}\}$ with coefficient vector $\bm{V}$, and note $\bm{V}^T \mathbf{K}_{\text{Total}} \bm{V} = A_h(\bm{v}_h, \bm{v}_h)$.

    For the largest eigenvalue, continuity, the global inverse estimate \eqref{eq:lemglobalinverse} and the norm
    equivalence \eqref{eq:lemmassequiv} give
    \begin{equation*}
        \bm{V}^T \mathbf{K}_{\text{Total}} \bm{V} \leq C_s \, |||\bm{v}_h|||^2
        \leq C_s C_{\mathrm{inv}} \, h^{-2} \|\bm{v}_h\|^2_{\Omega_T}
        \leq C_s C_{\mathrm{inv}} C_M \, h^{d-2} \, |\bm{V}|^2 ,
    \end{equation*}
    hence $\lambda_{\max}(\mathbf{K}_{\text{Total}}) \leq C_s C_{\mathrm{inv}} C_M \, h^{d-2}$.

    For the smallest eigenvalue, coercivity, the Poincar\'e--Friedrichs inequality \eqref{eq:lempoincare} and
    \eqref{eq:lemmassequiv} give
    \begin{equation*}
        \bm{V}^T \mathbf{K}_{\text{Total}} \bm{V} \geq c_s \, |||\bm{v}_h|||^2
        \geq c_s C_P^{-2} \, \|\bm{v}_h\|^2_{\Omega_T}
        \geq c_s C_P^{-2} c_M \, h^{d} \, |\bm{V}|^2 ,
    \end{equation*}
    hence $\lambda_{\min}(\mathbf{K}_{\text{Total}}) \geq c_s C_P^{-2} c_M \, h^{d}$. Taking the ratio,
    \begin{equation*}
        \operatorname{cond}(\mathbf{K}_{\text{Total}}) = \frac{\lambda_{\max}(\mathbf{K}_{\text{Total}})}{\lambda_{\min}(\mathbf{K}_{\text{Total}})}
        \leq \frac{C_s C_{\mathrm{inv}} C_M}{c_s C_P^{-2} c_M} \, h^{-2} ,
    \end{equation*}
    where every constant is independent of $h$ and, by Proposition~\ref{prop:coercivity} and Lemma~\ref{lem:discrete}, of the position of $\partial\Omega$ relative to the mesh.
\end{proof}

The two corollaries together capture the practical content of the stabilisation: without the ghost penalty, both the
stability constant in Corollary~\ref{cor:wellposed} and the condition number in Corollary~\ref{cor:conditioning}
degenerate as the cut volume $|T \cap \Omega| \to 0$ (cf. Section~\ref{sec:ghostpenalty}).

The energy point of view of the previous sections also yields a statement about the outer, nonlinear iteration: each
correction of the damped quasi-Newton scheme is a descent direction for the current iterate-frozen energy.

\begin{lemma}[Descent direction for the damped quasi-Newton scheme]
    \label{lem:descent}
    Let Assumption~\ref{ass:ellipticity} hold at the iterate $\bar{\bm{u}}$, with $\gamma_D \geq \gamma_0$ as in Proposition~\ref{prop:coercivity}, and let $\Delta\bm{u} \neq \bm{0}$ solve the quasi-Newton step of Corollary~\ref{cor:wellposed}. Then $\Delta\bm{u}$ is a descent direction for $\mathcal{E}_{\text{Total}}$:
    \begin{equation}
        D_{\Delta\bm{u}} \mathcal{E}_{\text{Total}}(\bar{\bm{u}}) = -A_h(\Delta\bm{u}, \Delta\bm{u}) \leq -c_s |||\Delta\bm{u}|||^2 < 0 ,
    \end{equation}
    and consequently $\mathcal{E}_{\text{Total}}(\bar{\bm{u}} + \beta\Delta\bm{u}) < \mathcal{E}_{\text{Total}}(\bar{\bm{u}})$ for all sufficiently small $\beta > 0$.
\end{lemma}

\begin{proof}
    Since $\mathcal{F}_{\text{Total}}(\bar{\bm{u}}, \cdot)$ is the exact first variation of $\mathcal{E}_{\text{Total}}$ with the penalty weight $\eta$ frozen at $\bar{\bm{u}}$ (Section~\ref{bcconditions}), the quasi-Newton step gives
    $D_{\Delta\bm{u}} \mathcal{E}_{\text{Total}}(\bar{\bm{u}}) = \mathcal{F}_{\text{Total}}(\bar{\bm{u}}, \Delta\bm{u}) = -\mathcal{K}_{\text{Total}}(\bar{\bm{u}}; \Delta\bm{u}, \Delta\bm{u})$, and the bound follows from Proposition~\ref{prop:coercivity}. The decrease for small $\beta$ is then immediate from the differentiability of $\mathcal{E}_{\text{Total}}$ on the finite-dimensional space $\mathbb{V}_h$.
\end{proof}

We emphasise that the lemma holds even though $\mathcal{K}_{\text{Total}}$ is only an approximation of the exact second
variation (cf. the symmetrised-tangent remark in Section~\ref{bcconditions}): descent only requires the residual to be
the exact first variation and the tangent to be positive definite. Because the surrogate functional changes between
iterations through the frozen weight $\eta$, this is a \emph{local} descent statement: each quasi-Newton correction is
a descent direction for the energy with the penalty weight frozen at the current iterate, and the damping parameter
$\beta$ of \eqref{eq:newtonupdate}, combined with a standard line-search criterion, guarantees decrease of that
iterate-frozen energy at every step. The lemma does not by itself establish monotone decrease of one fixed global
energy along the whole iteration, since successive steps refer to functionals differing by a contribution proportional
to $\|\bar{\bm{u}}\|^2_{\partial\Omega}$; these differences vanish as the weight settles at convergence.

For \emph{linear} elasticity, the symmetric Nitsche CutFEM with ghost-penalty stabilisation satisfies the optimal a
priori estimate
\begin{equation}
    \label{eq:linearestimate}
    \| \bm{u} - \bm{u}_h \|_{L^2(\Omega)} \leq C h^{p+1} | \bm{u} |_{H^{p+1}(\Omega)},
\end{equation}
see \cite{Hansbo_2017, Sticko2020}. For the nonlinear problem, the abstract approximation theory for branches of
nonsingular solutions due to Brezzi, Rappaz and Raviart \cite{brezzi1980finite, caloz1997numerical} yields the
following result for regular, corner-free branches; it excludes the corner-singular junctions of
Section~\ref{sec:mixedjunction}.

We first establish the part of the boundary Lipschitz hypothesis~(J1) that excludes $D\eta$. Although the maximum and
componentwise $1$-norm in \eqref{eq:penaltyweight} need not be differentiable, their values are Lipschitz, which
controls differences between the state-evaluated, frozen-weight penalty tangents. Thus only the terms containing
$D\eta$ remain a genuine hypothesis in (J1).

\begin{lemma}[Boundary-tangent Lipschitz estimate excluding $D\eta$]
    \label{lem:boundarylip}
    Let $\bm u \in (H^{p+1}(\Omega_T))^d \cap (W^{1,\infty}(\Omega_T))^d$ have zero trace on $\partial\Omega$, let
    $p > d/2$, and let Assumption~\ref{ass:ellipticity} hold together with $\Psi \in C^{3,1}$ on its non-degenerate
    range, $|\partial^3\Psi(\bm F_1) - \partial^3\Psi(\bm F_2)| \leq C_{L''} |\bm F_1 - \bm F_2|$. Assume the standard
    cut interpolation trace estimate $\|h^{-1/2}(\bm u - I_h \bm u)\|_{\partial\Omega} \leq C h^p$
    \cite{Hansbo_2017, Sticko2020}. Let $DF_h^{\partial}$ denote the boundary tangent obtained with the adaptive weight
    held fixed while differentiating at each state (but evaluated at that state), i.e. the derivative of
    $\mathcal{F}_{\text{Bound}}$ in \eqref{eq:fbound} excluding the terms generated by $D\eta$. Then, in the discrete
    neighbourhood
    $\mathcal B_h = \{\bm v_h \in \mathbb V_h : ||| I_h\bm u - \bm v_h ||| \leq \delta h^{d/2}\}$ of the interpolant,
    with $\delta$ sufficiently small,
    there exist $C > 0$ and $h_0 > 0$, independent of $h$ and of the cut, such that for all $h \leq h_0$ and
    $\bm w, \bm z \in \mathcal B_h$,
    \begin{equation}
        \label{eq:boundarylip}
        \big\| DF_h^{\partial}(\bm w) - DF_h^{\partial}(\bm z) \big\|_{\mathcal L(\mathbb V_h, \mathbb V_h')}
        \;\leq\; C \, h^{-d/2} \, ||| \bm w - \bm z ||| ,
    \end{equation}
    the operator norm being induced by $|||\cdot|||$ and its dual norm.
\end{lemma}

\begin{proof}
    By \eqref{eq:kbound}, freezing $\eta$ while differentiating at a given state gives the penalty tangent
    $\int_{\partial\Omega}(\gamma_D/h)\eta(\bm w)\,\bm{\delta u}\cdot\bm v\,dS$. Since the frozen value is nevertheless
    evaluated at the state, its difference between $\bm w$ and $\bm z$ is
    \begin{equation*}
        T_{\eta,0}=\int_{\partial\Omega}\frac{\gamma_D}{h}
        \big(\eta(\bm w)-\eta(\bm z)\big)\,\bm{\delta u}\cdot\bm v\,dS .
    \end{equation*}
    The remaining difference consists of the two Nitsche terms with coefficient
    $\mathbb L(\bm w) - \mathbb L(\bm z)$,
    \begin{equation*}
        T_1 = -\int_{\partial\Omega} \bm n \cdot \big( (\mathbb L(\bm w) - \mathbb L(\bm z)) : \Grad \bm{\delta u} \big) \cdot \bm v \, dS
        \qquad (\text{and its transpose } T_1' \text{ with } \bm{\delta u}, \bm v \text{ exchanged}),
    \end{equation*}
    and the two generated by the third-derivative term,
    \begin{equation*}
        T_2 = -\int_{\partial\Omega} \bm n \cdot \big( (D\mathbb L(\bm w) - D\mathbb L(\bm z))[\Grad\bm{\delta u}] : \Grad \bm v \big) \cdot \bm w \, dS,
        \qquad
        T_3 = -\int_{\partial\Omega} \bm n \cdot \big( D\mathbb L(\bm z)[\Grad\bm{\delta u}] : \Grad \bm v \big) \cdot (\bm w - \bm z) \, dS .
    \end{equation*}
    Three ingredients are used, each with a cut-independent constant. \emph{(1) $L^\infty$ inverse bounds:} as in the
    bulk argument, $\|\Grad\bm\chi_h\|_{L^\infty(\Omega_T)} \leq C_{\mathrm{inv}} h^{-d/2} |||\bm\chi_h|||$ for
    $\bm\chi_h \in \mathbb V_h$ (elementwise inverse inequality plus \eqref{eq:extension}), whence
    $\|\mathbb L(\bm w) - \mathbb L(\bm z)\|_{L^\infty} \leq C_{L'} C_{\mathrm{inv}} h^{-d/2} |||\bm w - \bm z|||$ and,
    by the $C^{3,1}$ assumption,
    $\|D\mathbb L(\bm w) - D\mathbb L(\bm z)\|_{L^\infty} \leq C_{L''} C_{\mathrm{inv}} h^{-d/2} |||\bm w - \bm z|||$;
    for $\delta$ small the deformation gradients of $\bm w, \bm z$ stay in the non-degenerate range, since
    $\|\Grad(\bm w - I_h\bm u)\|_{L^\infty} \leq C_{\mathrm{inv}}\delta$ and $\|\Grad I_h\bm u\|_{L^\infty} \leq C
        \|\bm u\|_{W^{1,\infty}(\Omega_T)}$. \emph{(2) Trace control of discrete functions:} the cut trace inequality
    \eqref{eq:cuttrace} with \eqref{eq:extension} gives $\|h^{1/2}\Grad\bm\chi_h\|_{\partial\Omega} \leq C
        |||\bm\chi_h|||$, while $\|h^{-1/2}\bm\chi_h\|_{\partial\Omega} \leq |||\bm\chi_h|||$ holds by definition of the
    triple norm. \emph{(3) Trace smallness of the iterates:} since $\bm u$ has zero trace,
    \begin{equation*}
        \|h^{-1/2}\bm w\|_{\partial\Omega}
        \leq \|h^{-1/2}(\bm u - I_h\bm u)\|_{\partial\Omega} + ||| \bm w - I_h\bm u |||
        \leq C ( h^p + \delta h^{d/2} ) \leq C h^{d/2}
        \qquad (p > d/2, \; h \leq h_0).
    \end{equation*}
    The five contributions are bounded by distributing these factors. First, the elementary inequalities
    \begin{equation*}
        |\max\{a,s\}-\max\{a,t\}|\leq |s-t|,
        \qquad
        |\|\mathbb A\|_1-\|\mathbb B\|_1|\leq\|\mathbb A-\mathbb B\|_1,
    \end{equation*}
    together with equivalence of finite-dimensional tensor norms and ingredient~(1), give
    \begin{equation*}
        \|\eta(\bm w)-\eta(\bm z)\|_{L^\infty(\partial\Omega)}
        \leq C\|\mathbb L(\bm w)-\mathbb L(\bm z)\|_{L^\infty(\partial\Omega)}
        \leq C h^{-d/2}|||\bm w-\bm z|||.
    \end{equation*}
    Consequently,
    \begin{equation*}
        |T_{\eta,0}|
        \leq \gamma_D\|\eta(\bm w)-\eta(\bm z)\|_{L^\infty}
        \|h^{-1/2}\bm{\delta u}\|_{\partial\Omega}\|h^{-1/2}\bm v\|_{\partial\Omega}
        \leq C h^{-d/2}|||\bm w-\bm z|||\,|||\bm{\delta u}|||\,|||\bm v|||.
    \end{equation*}
    For the Nitsche coefficient terms,
    \begin{equation*}
        |T_1| \leq \|\mathbb L(\bm w) - \mathbb L(\bm z)\|_{L^\infty} \, \|h^{1/2}\Grad\bm{\delta u}\|_{\partial\Omega} \, \|h^{-1/2}\bm v\|_{\partial\Omega}
        \leq C h^{-d/2} \, |||\bm w - \bm z||| \, |||\bm{\delta u}||| \, |||\bm v||| ,
    \end{equation*}
    and the same bound for $T_1'$. For $T_3$, place the $L^\infty$ bound on $\Grad\bm{\delta u}$ and note that
    $|D\mathbb L(\bm z)| \leq C_{L'}$ pointwise by Assumption~\ref{ass:ellipticity}(b):
    \begin{equation*}
        |T_3| \leq C_{L'} \, \|\Grad\bm{\delta u}\|_{L^\infty} \, \|h^{1/2}\Grad\bm v\|_{\partial\Omega} \, \|h^{-1/2}(\bm w - \bm z)\|_{\partial\Omega}
        \leq C h^{-d/2} \, |||\bm{\delta u}||| \, |||\bm v||| \, |||\bm w - \bm z||| .
    \end{equation*}
    For $T_2$, the same distribution, with ingredient (3) absorbing the second inverse factor:
    \begin{equation*}
        |T_2| \leq \|D\mathbb L(\bm w) - D\mathbb L(\bm z)\|_{L^\infty} \, \|\Grad\bm{\delta u}\|_{L^\infty} \, \|h^{1/2}\Grad\bm v\|_{\partial\Omega} \, \|h^{-1/2}\bm w\|_{\partial\Omega}
        \leq C h^{-d/2} \cdot h^{-d/2} \cdot h^{d/2} \, |||\bm w - \bm z||| \, |||\bm{\delta u}||| \, |||\bm v||| ,
    \end{equation*}
    which is again $C h^{-d/2} |||\bm w - \bm z||| \, |||\bm{\delta u}||| \, |||\bm v|||$. Summing proves
    \eqref{eq:boundarylip}; every constant arises from \eqref{eq:cuttrace}, \eqref{eq:extension} and elementwise
    inverse inequalities on full cells, hence is independent of the position of $\partial\Omega$ relative to the mesh.
\end{proof}

The same three ingredients bound the Jacobian defect $DF_h-\widetilde{DF}_h$ on $\mathcal B_h$: its omitted terms
contain the undifferentiated boundary trace of the state, whose smallness absorbs the inverse factor.

\begin{lemma}[Jacobian-defect bound in the Newton neighbourhood]
    \label{lem:defect}
    Adopt the assumptions of Lemma~\ref{lem:boundarylip}, with $\Psi \in C^3$, i.e. the bound $|\partial^3\Psi| \leq
        C_{L'}$ of Assumption~\ref{ass:ellipticity}(b), sufficing in place of $C^{3,1}$. Let $F_h$ denote the algorithmic
    residual and $\widetilde{DF}_h$ the assembled symmetric quasi-Newton tangent of Section~\ref{bcconditions} (the
    notation of Theorem~\ref{thm:brr} below). The defect $DF_h - \widetilde{DF}_h$ consists of the omitted
    third-derivative boundary term
    \begin{equation*}
        T^{(3)}(\bm w)[\bm{\delta u}, \bm v] = -\int_{\partial\Omega} \bm n \cdot \big( D\mathbb L(\bm w)[\Grad\bm{\delta u}] : \Grad\bm v \big) \cdot \bm w \, dS
    \end{equation*}
    and the weight term
    $T^{(\eta)}(\bm w)[\bm{\delta u}, \bm v] = \int_{\partial\Omega} \tfrac{\gamma_D}{h} \big(D\eta(\bm w)[\bm{\delta u}]\big) \, \bm w \cdot \bm v \, dS$,
    the latter defined wherever the active branch of the maximum and the signs in the $1$-norm defining $\eta$ are
    locally constant. There exists $C > 0$, independent of $h$ and of the cut, such that for all $\bm w \in \mathcal B_h$
    and $h \leq h_0$,
    \begin{equation}
        \label{eq:defectbound}
        \big\| T^{(3)}(\bm w) \big\|_{\mathcal L(\mathbb V_h, \mathbb V_h')} \;\leq\; C \big( h^{\,p - d/2} + \delta \big) ,
    \end{equation}
    and, where defined, the same bound for $T^{(\eta)}$. In particular, for every $\vartheta \in (0,1)$ there exist
    $\delta_0 > 0$ and $h_0 > 0$ such that the defect is bounded by $\vartheta c_s$ for all $\delta \leq \delta_0$ and
    $h \leq h_0$.
\end{lemma}

\begin{proof}
    The stated decomposition of the defect follows by differentiating the boundary residual \eqref{eq:fbound}
    exactly and subtracting the assembled tangent \eqref{eq:kbound}: the bulk and ghost-penalty tangents are exact,
    the three displayed Nitsche terms reproduce the remaining boundary derivatives,
    and what is left is precisely $T^{(3)}$ and $T^{(\eta)}$. For $T^{(3)}$, since
    $|D\mathbb L(\bm w)| \leq C_{L'}$ pointwise (for $\delta$ small the deformation gradients stay in the
    non-degenerate range, as in Lemma~\ref{lem:boundarylip}), distributing the factors as there:
    the $L^\infty$ inverse bound on $\Grad\bm{\delta u}$, the weighted trace on $\Grad\bm v$, and the trace smallness
    of the iterate,
    \begin{equation*}
        |T^{(3)}(\bm w)[\bm{\delta u}, \bm v]|
        \leq C_{L'} \, \|\Grad\bm{\delta u}\|_{L^\infty(\partial\Omega)} \, \|h^{1/2}\Grad\bm v\|_{\partial\Omega} \, \|h^{-1/2}\bm w\|_{\partial\Omega}
        \leq C_{L'} C_{\mathrm{inv}} \, h^{-d/2} \cdot C \big( h^p + \delta h^{d/2} \big) \, |||\bm{\delta u}||| \, |||\bm v||| ,
    \end{equation*}
    which is \eqref{eq:defectbound}. For $T^{(\eta)}$, on a fixed branch $|D\eta(\bm w)[\bm{\delta u}]| \leq
        \tfrac12 \|D\mathbb L(\bm w)[\Grad\bm{\delta u}]\|_1 \leq C C_{L'} |\Grad\bm{\delta u}|$ pointwise (and $D\eta = 0$
    on the $2\mu_e$ branch), so
    \begin{equation*}
        |T^{(\eta)}(\bm w)[\bm{\delta u}, \bm v]|
        \leq C \gamma_D C_{L'} \, \|\Grad\bm{\delta u}\|_{L^\infty(\partial\Omega)} \, \|h^{-1/2}\bm w\|_{\partial\Omega} \, \|h^{-1/2}\bm v\|_{\partial\Omega} ,
    \end{equation*}
    and the same combination of factors gives \eqref{eq:defectbound}. Since $p > d/2$, the first contribution
    vanishes as $h \to 0$; the final claim follows by choosing $\delta_0 \leq \vartheta c_s / (2C)$ and $h_0$ with
    $C h_0^{\,p-d/2} \leq \vartheta c_s / 2$. All constants arise from the same cut-independent ingredients as in
    Lemma~\ref{lem:boundarylip}.
\end{proof}

\begin{theorem}[Quasi-optimal convergence near regular solutions, after \cite{brezzi1980finite}]
    \label{thm:brr}
    Let $\bm{u}$ be a regular solution of the continuous problem, i.e. one at which the exact tangent operator is an
    isomorphism, and assume that its extension to the active mesh belongs to
    $(H^{p+1}(\Omega_T))^d\cap(W^{1,\infty}(\Omega_T))^d$, with deformation gradients in the non-degenerate $C^3$
    range of Assumption~\ref{ass:ellipticity}. Let $F_h$ denote the algorithmic residual (including the adaptive value
    $\eta(\bm v_h)$, but not its derivative in the residual itself) and let $\widetilde{DF}_h$ denote the assembled
    symmetric quasi-Newton tangent. For some $\delta>0$ independent of $h$ and of the cut, assume in the shrinking
    neighbourhood $\mathcal B_h=\{\bm v_h:|||I_h\bm u-\bm v_h|||\le\delta h^{d/2}\}$ that:
    \begin{enumerate}
        \item[(J0)] there is a constant $c_s>0$, independent of $h$, of the cut, and of
              $\bm v_h\in\mathcal B_h$, such that
              \begin{equation}
                  \label{eq:uniformcoercivityball}
                  \langle\widetilde{DF}_h(\bm v_h)\bm w_h,\bm w_h\rangle
                  \ge c_s|||\bm w_h|||^2
                  \qquad\forall\bm w_h\in\mathbb V_h;
              \end{equation}
              Proposition~\ref{prop:coercivity} supplies this condition when its ellipticity and penalty hypotheses
              hold uniformly throughout $\mathcal B_h$;
        \item[(J1)] $F_h$ is Fr\'echet differentiable on $\mathcal B_h$ and $DF_h$ has Lipschitz modulus at most
              $C h^{-d/2}$, with $C$ independent of $h$ and of the cut;
        \item[(J2)] for some $\vartheta\in[0,1)$, independent of $h$ and of the cut,
              \begin{equation}
                  \label{eq:jacobiandefect}
                  \sup_{\bm v_h\in\mathcal B_h}
                  \|DF_h(\bm v_h)-\widetilde{DF}_h(\bm v_h)\|_{\mathcal L(\mathbb V_h,\mathbb V_h')}
                  \le \vartheta c_s,
              \end{equation}
              where the operator norm is induced by $|||\cdot|||$ and its dual norm, and $c_s$ is the uniform
              coercivity constant in \eqref{eq:uniformcoercivityball}; Lemma~\ref{lem:defect} supplies this bound,
              with $\vartheta \leq C(h^{p-d/2} + \delta)$, wherever the active branch of $\eta$ is fixed.
    \end{enumerate}
    Suppose also that the standard CutFEM consistency and quasi-optimal interpolation estimates
    $\|F_h(I_h\bm u)\|_{\mathbb V_h'}\le C|||\bm u-I_h\bm u|||$ and
    $|||\bm u-I_h\bm u|||\le C\inf_{\bm v_h\in\mathbb V_h}|||\bm u-\bm v_h|||$ hold uniformly in the cut. Then,
    provided the polynomial degree satisfies $p > d/2$ (so $p \geq 2$ in two dimensions, $p \geq 1$ in one), $\delta$
    may be chosen sufficiently small and there exists $h_0>0$ such that for all $h \leq h_0$ the discrete problem
    possesses a unique solution $\bm{u}_h$ in $\mathcal B_h$, and the error is quasi-optimal:
    \begin{equation}
        ||| \bm{u} - \bm{u}_h ||| \leq C \inf_{\bm{v}_h \in \mathbb{V}_h} ||| \bm{u} - \bm{v}_h ||| .
    \end{equation}
\end{theorem}

A corresponding conditional result for the graded and balanced-adaptive families, formulated with local-size analogues
of (J0)--(J2), is given in Theorem~\ref{thm:gradedqo}.

\begin{proof}
    We use a quantitative local form of the Brezzi--Rappaz--Raviart framework
    \cite{brezzi1980finite, caloz1997numerical}. The solution $\bm u_h$, consistency estimate and error bound concern
    the algorithmic residual $F_h$ and its exact derivative $DF_h$; the assembled quasi-Newton tangent enters only
    through \eqref{eq:jacobiandefect} to establish stability. Writing the continuous problem as $F(\bm u)=\bm 0$,
    regularity makes $DF(\bm u)$ an isomorphism, while smoothness makes it locally Lipschitz. The required ingredients
    are:
    \begin{enumerate}
        \item[(i)] \emph{Stability.} The discrete tangent $DF_h(\bm{v}_h)$ is invertible with
              $\|DF_h(\bm{v}_h)^{-1}\| \leq \gamma^{-1}$ uniformly in $h$ for $\bm{v}_h$ near $\bm{u}$. The uniform
              coercivity of the assembled symmetrised tangent $\widetilde{DF}_h$ is hypothesis~(J0); as noted
              there, Proposition~\ref{prop:coercivity} supplies it if its assumptions hold uniformly throughout
              $\mathcal B_h$. Hypothesis~(J2), i.e.~\eqref{eq:jacobiandefect}, accounts for both omitted contributions:
              the third-derivative boundary term and the derivative of the adaptive weight. Combining (J0) and (J2),
              for any $\bm w_h\in\mathbb V_h$,
              \[
                  \langle DF_h(\bm v_h)\bm w_h,\bm w_h\rangle
                  \ge \langle\widetilde{DF}_h(\bm v_h)\bm w_h,\bm w_h\rangle
                  - \|DF_h(\bm v_h)-\widetilde{DF}_h(\bm v_h)\|_{\mathcal L(\mathbb V_h,\mathbb V_h')}\,|||\bm w_h|||^2
                  \ge (1-\vartheta)c_s\,|||\bm w_h|||^2,
              \]
              so $DF_h$ is in fact coercive with constant $(1-\vartheta)c_s$, uniformly in $h$ and in the cut, and hence $\|DF_h(\bm
                  v_h)^{-1}\|\le[(1-\vartheta)c_s]^{-1}$ on $\mathcal B_h$.
        \item[(ii)] \emph{Consistency.} The assumed consistency and interpolation estimates give
              $\| F_h(I_h \bm{u}) \| \le C \inf_{\bm{v}_h}|||\bm{u}-\bm{v}_h||| \to 0$ as $h \to 0$, where $I_h$
              is the interpolant onto $\mathbb{V}_h$. The consistency estimate rests on the fact that the exact regular
              solution satisfies the extended discrete residual identity $F_h(\bm{u}) = 0$. Indeed, integration by
              parts in the bulk residual produces the boundary traction term, which is cancelled by the Nitsche
              first consistency integral in \eqref{eq:fbound}. The symmetric and
              penalty contributions vanish because the exact solution has zero trace on the Dirichlet boundary (written
              as $\partial\Omega$ in Section~\ref{bcconditions}), while natural-boundary contributions are already part
              of the continuous weak problem. The ghost-penalty terms \eqref{eq:fstab}, which are weighted sums of the normal-derivative jumps
              $[\![\partial_{n_F}^k\bm{u}]\!]$ over interior faces, vanish because $\bm{u}$ is smooth
              across those faces. Here $\bm{u}$ denotes a fixed Sobolev (Calder\'on--Zygmund/Stein) extension of the exact
              solution from $\Omega$ to the active mesh $\Omega_T$, preserving $\bm{u}\in (H^{p+1}(\Omega_T))^d \cap
                  (W^{1,\infty}(\Omega_T))^d$; this is needed because some faces of the active mesh lie in the fictitious
              region $\Omega_T\setminus\Omega$ where the ghost penalty and $F_h$ are evaluated.
              The asserted residual bound is precisely the standard local continuity/consistency estimate applied to
              $F_h(I_h\bm u)-F_h(\bm u)$, followed by the assumed quasi-optimal interpolation estimate in the triple norm;
              it is not inferred from (J1), which is posed only on the discrete neighbourhood $\mathcal B_h$.
        \item[(iii)] \emph{Lipschitz tangent.} Hypothesis~(J1) gives the required Lipschitz estimate for the complete
              derivative $DF_h$ on $\mathcal B_h$. The expected $h^{-d/2}$ scaling can already be seen in its bulk
              part: the bulk contribution to $DF_h(\bm{w})-DF_h(\bm{z})$ is $\int_\Omega \Grad\,\cdot\,:(\mathbb{L}(\bm{w})
                  -\mathbb{L}(\bm{z})):\Grad\,\cdot\,$, and since $\mathbb{L}=\partial^2\Psi$ is itself $C^1$ with
              $|\partial^3\Psi|\le C_{L'}$ on $\Omega_T$ by Assumption~\ref{ass:ellipticity},
              \[
                  \|\mathbb{L}(\bm{w})-\mathbb{L}(\bm{z})\|_{L^\infty} \le C_{L'}\|\Grad(\bm{w}-\bm{z})\|_{L^\infty}.
              \]
              Unlike the energy norm, $|||\cdot|||$ does \emph{not} control $\|\Grad\,\cdot\,\|_{L^\infty}$ uniformly in $h$; instead
              we use the polynomial inverse inequality $\|\Grad\bm{v}_h\|_{L^\infty}\le C_{\mathrm{inv}}\,
                  h^{-d/2}\|\Grad\bm{v}_h\|_{L^2}\le C_{\mathrm{inv}}\,h^{-d/2}|||\bm{v}_h|||$ on the active mesh. For
              $\bm{w},\bm{z}\in\mathcal{B}_h$ this gives
              \[
                  \|\mathbb{L}(\bm{w})-\mathbb{L}(\bm{z})\|_{L^\infty} \le C_{L'}C_{\mathrm{inv}}\,h^{-d/2}\,|||\bm{w}-\bm{z}|||,
              \]
              while the same inverse inequality keeps $\|\Grad\bm{w}\|_{L^\infty},\|\Grad\bm{z}\|_{L^\infty}$ within
              $C_{\mathrm{inv}}\delta$ of $\|\Grad I_h\bm{u}\|_{L^\infty}\le C\|\bm{u}\|_{W^{1,\infty}(\Omega_T)}$, so the iterates
              stay in the non-degenerate region $J\ge J_0$ where $C_{L'}$ is valid (here the $W^{1,\infty}$ regularity of $\bm{u}$
              enters). The bulk Lipschitz constant on $\mathcal{B}_h$ is thus bounded by $C_{L'}C_{\mathrm{inv}}\,h^{-d/2}$. The
              state-dependent boundary contributions not containing $D\eta$, including the third-derivative term, obey the same
              $h^{-d/2}$ modulus by Lemma~\ref{lem:boundarylip}, under the $C^{3,1}$ regularity assumed there. Hypothesis~(J1) is
              therefore effective only for the derivative of the adaptive weight $\eta$: keeping the active branch of the maximum and
              the signs entering the componentwise $1$-norm fixed removes the nonsmoothness obstruction, but the corresponding
              operator bound is not derived here (a fixed penalty weight would remove this last gap altogether). Writing $L_h\le C
                  h^{-d/2}$ for the resulting Lipschitz modulus of the complete derivative, the contraction factor is $L_h\delta
                  h^{d/2}\le C\delta$, which is made $<1$ by choosing $\delta$ small, independently of $h$. We stress that the modulus
              $C_{L'}$ requires the deformation gradient to stay in the non-degenerate, $C^3$ range of
              Assumption~\ref{ass:ellipticity}, which fails for instance at the mixed Dirichlet--Neumann junctions analysed in
              Section~\ref{sec:mixedjunction}, where the singular expansion \eqref{eq:cornerexpansion} drives
              $|\Grad\bm{u}|\to\infty$ as $r\to 0$ and $\bm{u}\notin W^{1,\infty}$. The theorem therefore applies only to the regular
              (corner-free) branch; at an unresolved junction the lower bound of Proposition~\ref{prop:junctionlower} takes over and
              caps the rate, and recovering quasi-optimality there requires the graded refinement of Section~\ref{sec:remedy} rather
              than this assumption.
    \end{enumerate}
    Under (i)--(iii) the BRR theorem applies. Concretely, define the discrete Newton map
    $G_h(\bm{v}_h) = \bm{v}_h - DF_h(I_h\bm{u})^{-1} F_h(\bm{v}_h)$. By (i) and (iii), $G_h$ is a contraction on the
    shrinking-radius ball $\mathcal{B}_h = \{ \bm{v}_h : ||| I_h\bm{u} - \bm{v}_h ||| \leq \delta\,h^{d/2} \}$, with
    $\delta$ chosen small enough that the contraction factor $L = C\delta < 1$ as in (iii),
    independently of $h$. For $G_h$ to map $\mathcal{B}_h$ into itself we need the consistency residual to fit inside the
    ball, $\| DF_h(I_h\bm{u})^{-1}F_h(I_h\bm{u})\| \le (1-L)\,\delta\,h^{d/2}$; by (i) and (ii) the left-hand side is
    $O(\inf_{\bm{v}_h}|||\bm{u}-\bm{v}_h|||) = O(h^p)$, so this holds for $h\le h_0$ precisely because $p > d/2$, the optimal
    error decaying faster than the ball radius. The Banach fixed-point theorem then yields a unique
    $\bm{u}_h$ in $\mathcal{B}_h$ with $F_h(\bm{u}_h)=\bm{0}$, which is the desired discrete solution.

    For the error bound, subtract the consistency identity. Since $F_h(\bm{u}_h)=\bm{0}$, a first-order Taylor expansion of
    $F_h$ about $I_h\bm{u}$ gives
    \[
        DF_h(I_h\bm{u})\,(\bm{u}_h - I_h\bm{u}) = -F_h(I_h\bm{u}) + R_h,
        \qquad \|R_h\| \leq \tfrac{1}{2}L_h\,|||\bm{u}_h - I_h\bm{u}|||^2,
    \]
    with $L_h\le C h^{-d/2}$ from (iii). Since $\bm{u}_h\in\mathcal{B}_h$ obeys $|||\bm{u}_h - I_h\bm{u}|||\le\delta
        h^{d/2}$, the remainder satisfies $\|R_h\|\le \tfrac12 C \delta\,|||\bm{u}_h - I_h\bm{u}|||$ and is absorbed into the
    left-hand side for $\delta$ small, independently of $h$. Using the stability bound (i) and estimating the consistency
    residual $F_h(I_h\bm{u})$ by the best approximation error (ii),
    \[
        ||| \bm{u}_h - I_h\bm{u} ||| \leq C \| F_h(I_h\bm{u}) \| \leq C \inf_{\bm{v}_h\in\mathbb{V}_h}
        ||| \bm{u} - \bm{v}_h |||.
    \]
    A triangle inequality with the quasi-optimal interpolation bound $|||\bm{u} - I_h\bm{u}|||\le
        C\inf_{\bm{v}_h}|||\bm{u}-\bm{v}_h|||$ then gives the quasi-optimal estimate.
\end{proof}

Combined with the approximation properties of the cut finite element space \cite{Hansbo_2017,
    Burman_Hansbo_Larson_Zahedi_2025}, quasi-optimality yields the rate $O(h^p)$ in the energy norm, and, by the usual
duality argument, $O(h^{p+1})$ in $L^2$ as in \eqref{eq:linearestimate}, provided the solution is sufficiently regular.

\paragraph{Remark (status of the hypotheses (J0)--(J2))} Their status is:
\begin{itemize}
    \item[(J0)] \emph{proved} by Proposition~\ref{prop:coercivity}, provided its hypotheses hold uniformly on
          $\mathcal B_h$.
    \item[(J1)] \emph{proved up to the weight term}: the bulk part in step (iii) of the proof and the boundary Nitsche
          terms in Lemma~\ref{lem:boundarylip} (requiring $\Psi \in C^{3,1}$); only the derivative of the nonsmooth
          adaptive weight $\eta$ remains a hypothesis.
    \item[(J2)] \emph{proved in the Newton neighbourhood} by Lemma~\ref{lem:defect}:
          $\|DF_h - \widetilde{DF}_h\| \leq C(h^{p-d/2} + \delta)$ on $\mathcal B_h$, so any $\vartheta < 1$ is met
          for $\delta$ small and $h \leq h_0$; at the discrete solution the defect is $O(h^{p-d/2})$, subject to the
          same $\eta$ caveat.
\end{itemize}
In particular, with a fixed penalty weight (J0)--(J2) are fully proved. The numerical results of
Section~\ref{testcase} confirm $O(h^{p+1})$ convergence in $L^2$ for $p = 2, 3$ on a sufficiently regular
(corner-free) solution. These results are consistent with the theorem in two dimensions; the $p=1$
case lies outside its degree threshold and is discussed in the following remark.

\paragraph{Remark (the degree threshold $p>d/2$ and the lowest order case)} The restriction $p>d/2$ in Theorem~\ref{thm:brr} is an artefact of closing the contraction in the energy norm: the bulk
Lipschitz modulus of the tangent carries an inverse-inequality factor $h^{-d/2}$, which is offset by the shrinking ball
radius $\delta h^{d/2}$ only when the optimal error $O(h^p)$ decays faster than the radius, i.e. $p>d/2$. This covers
$p\ge2$ in two dimensions and every order in one. The borderline lowest-order case $p=1$ in $d=2$ is not reached by
this argument, closing it rigorously would require a discrete $W^{1,\infty}$ stability bound on the linearised CutFEM
operator, which for unfitted discretisations is a substantial result in its own right and lies outside the present
scope. Accordingly, no optimal-rate claim is made here for $p=1$.

\section{Numerical Results}
\label{testcase}
The experiments use two geometries, a disc and a pole, two hyperelastic models, $\Psi_1$ and $\Psi_2$, and a body force
whose magnitude is controlled by a single dimensionless load factor $\loadf \ge 0$. We collect these settings here
before reporting the convergence studies.

\paragraph{Geometries and loading} The disc $\Omega_\circ$ is a smooth, corner-free domain (Section~\ref{sec:disc}); the pole $\Omega_{\sqcap}$ is a
vertical shaft of width $W$ and height $H$, flat at the bottom and closed at the top by a semicircular cap
(Section~\ref{sec:pole}). Each is loaded by a body force $\bm f_0$ entering the energy \eqref{eq:energy}, written as a
fixed base field scaled by the load factor $\loadf$. The two geometries carry different right-hand sides:
\begin{equation}
    \label{eq:rhs}
    \bm f_0^{\,\circ}(\bm X) = \loadf\,\bm b_\circ(\bm X)
    \qquad\text{(disc)},
    \qquad\qquad
    \bm f_0^{\,\sqcap}(\bm X) = \loadf\,\bm b_{\sqcap}(\bm X)
    \qquad\text{(pole)},
\end{equation}
where the fixed base body forces are the rotational field $\bm b_\circ(\bm X) = [\,X_2,\,-X_1\,]^\top$ for the disc
and the (essentially horizontal) field $\bm b_{\sqcap} = [\,10^{-2},\,-2\times10^{-4}\,]^\top\,\mathrm{N/mm^3}$ for the
pole. The disc has radius $10\,\mathrm{mm}$, and the pole has width $W = 20\,\mathrm{mm}$ and height $H =
    40\,\mathrm{mm}$. Increasing $\loadf$
drives the body deeper into the finite-strain regime; the load factors used in each experiment are reported
with the corresponding results. Figure~\ref{fig:setup} shows both geometries in their deformed configuration
on the coarsest mesh used in the experiments.

\begin{figure}[h]
    \centering
    \begin{minipage}{0.48\textwidth}
        \centering
        \includegraphics[height=5cm]{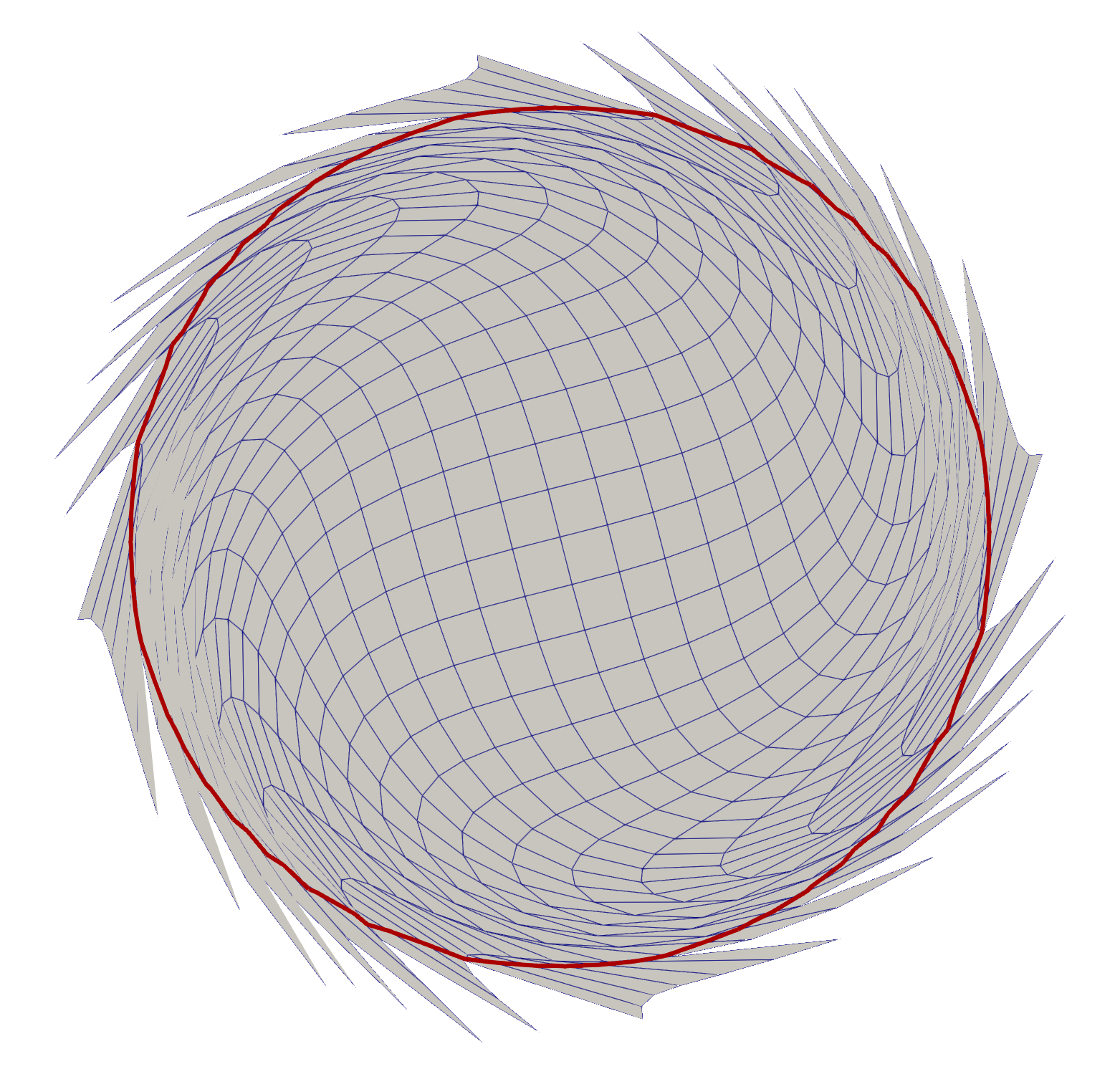}
    \end{minipage}\hfill
    \begin{minipage}{0.48\textwidth}
        \centering
        \includegraphics[height=5cm]{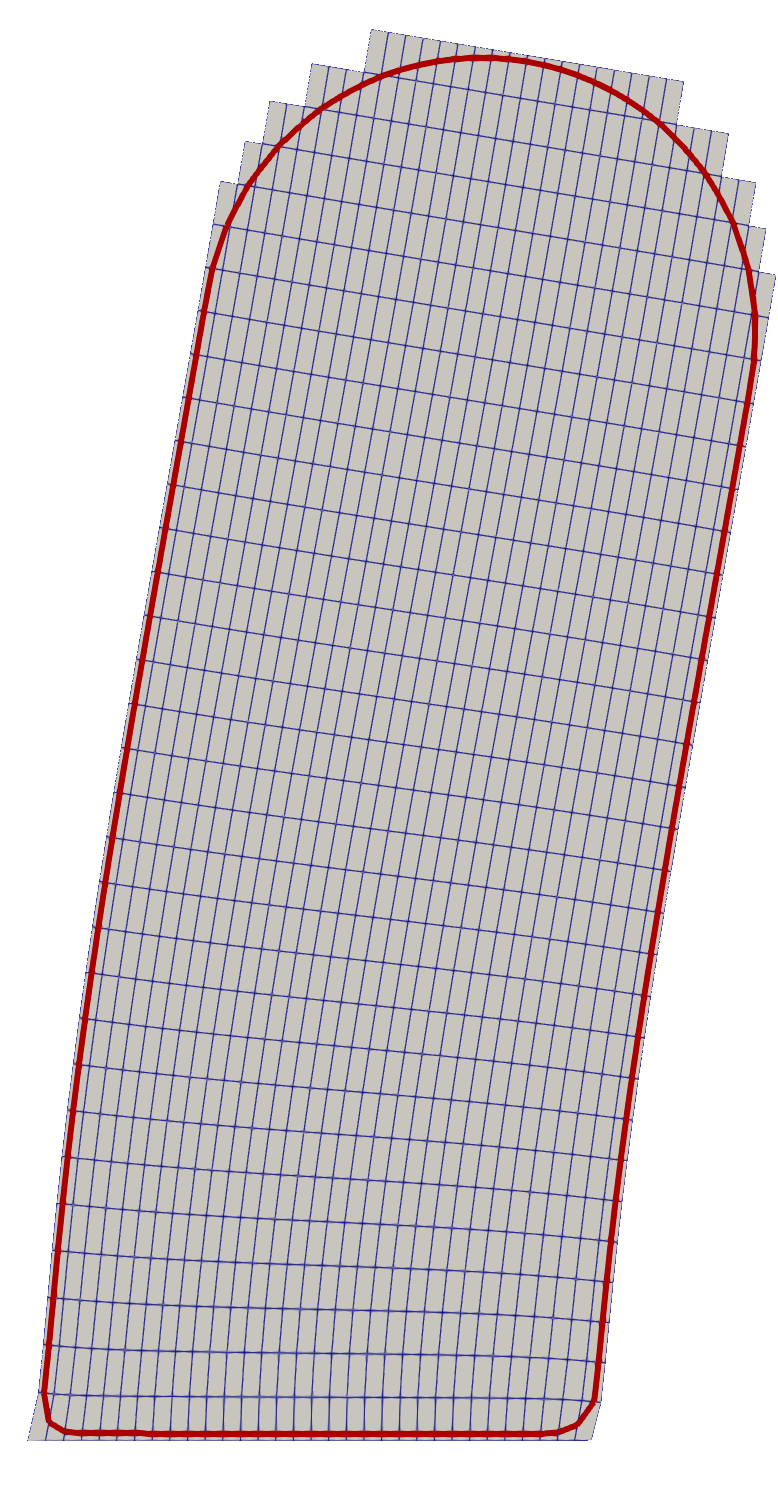}
    \end{minipage}
    \caption{Problem settings: the two geometries in their deformed configuration under the load
        \eqref{eq:rhs}, shown on the coarsest background mesh used in the experiments, with the immersed boundary
        $\partial\Omega$ superimposed. \emph{Left (a):} the disc $\Omega_\circ$ at load factor $\loadf = 0.5$. \emph{Right (b):} the pole
        $\Omega_{\sqcap}$ at load factor $\loadf = 0.25$. We use those loads as defaults in our experiments. The apparent
        rounding of the pole corners is a visualisation artefact because the output follows the level-set representation.
        In the implementation, the standard level-set quadrature is replaced on corner-containing cells by a dedicated
        corner quadrature that integrates the sharp geometry exactly.}
    \label{fig:setup}
\end{figure}

\paragraph{Constitutive models} To exercise the model-independence of the formulation we use two hyperelastic strain-energy densities, $\Psi_1$ and
$\Psi_2$, run through the \emph{identical} code paths: only the scalar energy passed to the AD layer changes, while the
residual, tangent, and solver are untouched.

The first is the compressible neo-Hookean model, widely adopted for rubber-like materials owing to its simplicity and
good predictive capability,
\begin{equation}\label{eq:hypref}
    \Psi_1 = \frac{\mu_e}{2} \left( \operatorname{tr} \bm{C} - \operatorname{tr} \bm{I} - 2 \log J \right) + \lambda_e \log^2 J,
\end{equation}
where $\mu_e$ is the shear modulus and $\lambda_e$ the Lam\'e constant. The volumetric term is often written
with an additional factor $1/2$, i.e. $\frac{\lambda_e}{2}\log^2 J$; the two conventions differ only by a
rescaling of $\lambda_e$, and we follow the form without this factor, as used e.g. in \cite{davydov2020matrix}.
This model was examined with \emph{manually} derived tangents in \cite{davydov2020matrix, Schussnig_2025},
providing a reliable baseline for our automatically generated code. We refer to \cite{ogden1972large,
    treloar1976mechanics, simo1984remarks, simo1998numerical} for general discussions of hyperelastic strain
energies.

The second model, \modelB, separates the isochoric and volumetric responses,
\begin{equation}\label{eq:hypref2}
    \Psi_2 = \frac{\mu_e}{2}\left( J^{-2/d}\operatorname{tr}\bm{C} - \operatorname{tr}\bm{I} \right)
    + \frac{\kappa_e}{2}\left( J - 1 \right)^2,
\end{equation}
with bulk modulus $\kappa_e$. It differs qualitatively from $\Psi_1$ in its volumetric coupling, yet requires
no change to the implementation beyond \eqref{eq:hypref2} itself, consistent with the
AD-based, energy-only formulation.

In all experiments the moduli $\mu_e$, $\lambda_e$, and $\kappa_e$ are derived from a fixed Young's modulus $E_e =
    6.0\,\mathrm{MPa}$ and the Poisson ratio $\nu$, the latter being varied as indicated in each study, through the
standard relations
\begin{equation}
    \label{eq:moduli}
    \mu_e = \frac{E_e}{2(1+\nu)}, \qquad
    \lambda_e = \frac{E_e\,\nu}{(1+\nu)(1-2\nu)}, \qquad
    \kappa_e = \lambda_e + \tfrac{2}{3}\mu_e = \frac{E_e}{3(1-2\nu)} .
\end{equation}
Here $\lambda_e$ is the value inserted into \eqref{eq:hypref} in the convention adopted there, i.e. with the
volumetric term $\lambda_e\log^2 J$ and no additional factor $1/2$.

\paragraph{Parameters and solver}
\label{sec:implementation}
The Nitsche penalty is $\gamma_D = m\,(p+1)p$ with $m = 5$ throughout, and the ghost penalty
\eqref{eq:kstab} is scaled by the shear modulus, $\gamma_A = \mu_e$, i.e. $\alpha = 1/(2(1+\nu))$ in the
notation $\gamma_A = \alpha E_e$ of \eqref{eq:gpenergy}; sensitivity to both parameters is examined in
Section~\ref{sec:condnumber}. All linear systems are solved with the sparse direct solver \texttt{UMFPACK}
\cite{davis2004algorithm}, and the Newton iteration is terminated when the Euclidean norm of the assembled
residual falls below the absolute tolerance $10^{-6}$, with $\beta = 1$ and load continuation as in
Section~\ref{sec:newtonconv}. The direct solver is a convenience for the two-dimensional problem sizes
considered here and is not a recommendation: better options are available for cut discretisations, in
particular geometric multigrid with vertex-patch smoothers, for which convergence bounds independent of the
mesh size and of the cut position have been established \cite{wichrowski2026multigridcutfem,
    ludescher2020multigrid, cui2025multigrid}.

\subsection{Convergence of Newton's method}
\label{sec:newtonconv}

Lemma~\ref{lem:descent} shows that each quasi-Newton correction is a descent direction for the iterate-frozen energy,
while Lemma~\ref{lem:defect} bounds the Jacobian defect near the discrete solution by $O(h^{p-d/2})$. We assess whether
the unfitted discretisation nevertheless degrades nonlinear robustness by comparing its iteration counts with a
body-fitted reference for the Neo-Hookean material at $\nu=0.45$.

In Table~\ref{tab:newton}, \emph{Steps} is the minimal number of equal load increments required to reach the target
$\zeta$, and the remaining columns give Newton iterations per increment for the CutFEM and matching discretisations.
The counts span all increments, $p=1,2,3$, and refinement levels $1,\dots,5$. Both methods required the same number of
load increments. No line search was used: $\beta=1$, with load continuation used throughout.
The pole converges in one increment except at the largest load on the finest, highest-degree grids; the more strongly
deformed disc requires up to three. Iterations per increment vary little with $h$ or $p$, and the variation with $p$, where a
trend is discernible at all, is a decrease with degree (from $8$--$11$ at $p=1$ to $6$--$8$ at $p=3$ for the disc at $\zeta=0.5$). CutFEM exceeds the
matching count by at most one iteration for $p\geq2$ and three for $p=1$, with no deterioration under refinement.

\begin{table}[h]
    \centering
    \begin{tabular}{llccc}
        \toprule
                              & $\zeta$ & Steps & CutFEM its./step & Matching its./step \\
        \midrule
        \multirow{3}{*}{Disc} & 0.10    & 1     & 6--8             & 5--6               \\
                              & 0.25    & 2     & 6--10            & 6--7               \\
                              & 0.50    & 3     & 6--11            & 6--8               \\
        \midrule
        \multirow{3}{*}{Pole} & 0.10    & 1     & 5                & 4--5               \\
                              & 0.25    & 1     & 6--7             & 5--6               \\
                              & 0.50    & 1     & 7--9             & 5--8               \\
        \bottomrule
    \end{tabular}

    \caption{Newton iterations per load increment (Neo-Hookean, $\nu=0.45$), given as the range over all increments,
        polynomial degrees $p=1,2,3$ and refinement levels $1,\dots,5$. The \emph{Steps} column is for $p=2$ at
        refinement level~3.}
    \label{tab:newton}
\end{table}

\subsection{Error measurement against a matching discretisation}
\label{sec:errormethod}
Because closed-form solutions are unavailable, on each refinement level we solve the same problem with a body-fitted
discretisation $\bm u_h^{\mathrm m}$ and with CutFEM $\bm u_h^{\mathrm c}$, and report
$\|\bm u_h^{\mathrm c}-\bm u_h^{\mathrm m}\|$ in the relevant norm. Accordingly, the $L^2$-difference plotted in
Figures~\ref{fig:discconvergence}, \ref{fig:poleconvergence}, and \ref{fig:penaltysensitivity} is
$\|\bm u_h^{\mathrm c}-\bm u_h^{\mathrm m}\|_{L^2}$.

This difference is enough to certify the CutFEM rate, by the triangle inequality. Writing $\bm u$ for the exact
solution,
\begin{equation}
    \label{eq:triangle}
    \|\bm u - \bm u_h^{\mathrm c}\| \;\le\; \|\bm u - \bm u_h^{\mathrm m}\|
    + \|\bm u_h^{\mathrm m} - \bm u_h^{\mathrm c}\|,
    \qquad
    \|\bm u - \bm u_h^{\mathrm m}\| \;\le\; \|\bm u - \bm u_h^{\mathrm c}\|
    + \|\bm u_h^{\mathrm m} - \bm u_h^{\mathrm c}\| .
\end{equation}
If both the matching error and the measured difference are $O(h^k)$, \eqref{eq:triangle} proves that the CutFEM error
is $O(h^k)$; the second inequality gives the converse. The difference alone certifies no rate, since both discrete
solutions approximate the same exact solution. Thus the certified rate is imported from body-fitted theory, never
inferred from a possibly faster-decaying difference.

On smooth domains, an $O(h^{p+1})$ difference therefore certifies the optimal CutFEM $L^2$ rate
(Section~\ref{sec:disc}). At a mixed Dirichlet--Neumann junction, the matching rate is instead capped at
$h^{\min(p+1,\,2\lambda)}$ (Section~\ref{sec:mixedjunction}), so the difference certifies the same CutFEM rate; the
lower bounds of Proposition~\ref{prop:junctionlower} and Corollary~\ref{cor:junctionlowerL2} exclude improvement beyond
$h^\lambda$ in energy and $h^{1+\lambda}$ in $L^2$ on the same space.

Unlike a smooth manufactured solution, this comparison retains the corner phenomenon while isolating the cost of the
cut. The AD-generated constitutive code is validated separately against the manual implementations of
\cite{davydov2020matrix, Schussnig_2025} in \cite{wichrowski2025matrixfreemethodsfinitestrainelasticity}.

\subsection{Smooth domain, all Dirichlet: optimal convergence}
\label{sec:disc}
We first verify the method on a configuration free of both boundary-condition junctions and geometric
corners: the domain $\Omega = \Omega_\circ$ is a disc, and homogeneous Dirichlet conditions are imposed weakly
on the \emph{entire} boundary $\partial\Omega$ through the Nitsche terms of Section~\ref{bcconditions}. The
body is loaded by the body force $\bm f_0^{\,\circ} = \loadf\,\bm b_\circ$ from \eqref{eq:rhs}.
Because the boundary is smooth and carries a single condition type, the exact solution is smooth and the
regularity assumption of Theorem~\ref{thm:brr} holds, so the optimal rate is expected for every polynomial
degree. The error is measured against a matching discretisation as described in
Section~\ref{sec:errormethod}; since the disc boundary is smooth, the matching method converges optimally and
the difference therefore certifies the CutFEM rate directly.

Figure~\ref{fig:discconvergence} shows the results for both constitutive models $\Psi_1$ and $\Psi_2$ at Poisson ratio
$\nu = 0.45$, for two load levels: load factor $0.1$ (dashed lines) and $0.5$ (solid lines). At the lower load the $p =
    2$ and $p = 3$ errors converge at the optimal rate $O(h^{p+1})$ for both models, confirming Theorem~\ref{thm:brr} in
the nearly-linear regime, while the $p = 1$ rates approach the optimal value from below over the levels computed. At
load factor $0.5$ the two models separate: for the neo-Hookean model~(a) the $p = 2$ and $p = 3$ rates rise under
refinement and reach their optimal slopes on the finest levels, whereas for the \modelBshort{} model~(b) they are still
rising and remain short of the optimal values --- in both cases a preasymptotic regime whose onset moves to finer
meshes as the load increases. The $p = 1$ sequences degrade in both panels, for different reasons; we recall that $p =
    1$ in two dimensions lies outside the hypotheses of Theorem~\ref{thm:brr}, which requires $p > d/2$. For the
neo-Hookean model a Poisson-ratio sweep at this load shows the rate improving monotonically as the material is made
more compressible, the signature of incipient volumetric locking of low-order elements~\cite{babuvska1992locking}; the
value $\nu = 0.45$ used here sits at the stiff end of that sweep. Lowering $\nu$ as far as $0.1$ improves the observed
rate but does not restore the optimal $p=1$ rate over the computed levels, showing that locking is compounded by a
load-dependent preasymptotic effect. For the \modelBshort{} model the error instead stalls at an $h$-independent floor
that persists throughout the same Poisson-ratio sweep, which excludes locking: at this load the deformation is extreme,
beyond what a piecewise-linear displacement field can represent, and the $p = 1$ sequence does not converge in this
regime.

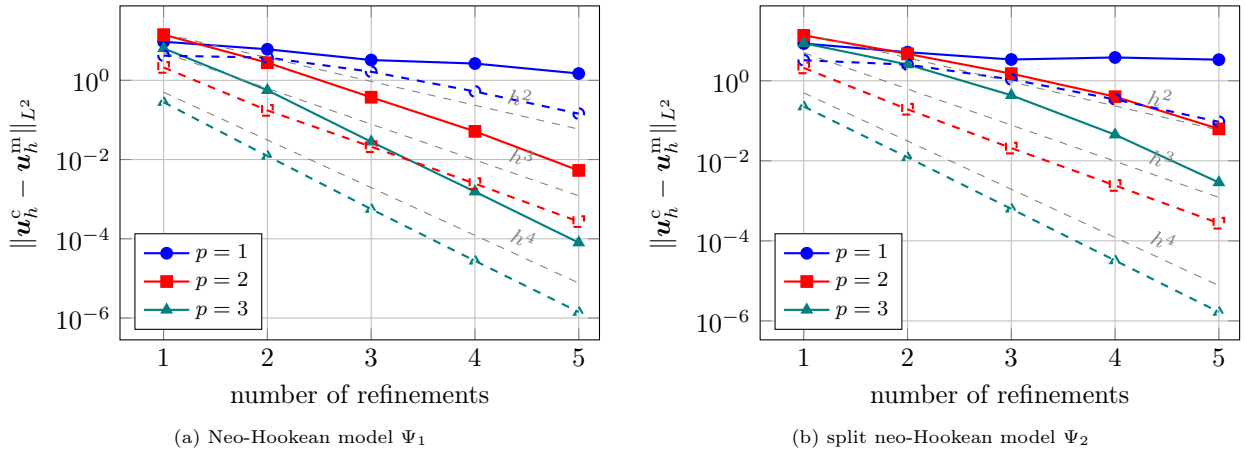
\begin{figure}[h]
    \centering
    \subcaptionbox{Neo-Hookean model $\Psi_1$\label{fig:discconvergence:nh}}{%
    \begin{tikzpicture}
        \begin{semilogyaxis}[
                width=0.48\textwidth, height=6cm,
                xlabel={number of refinements},
                ylabel={$\|\bm u_h^{\mathrm c}-\bm u_h^{\mathrm m}\|_{L^2}$},
                xtick={1,2,3,4,5},
                xmax=5.2,
                legend pos=south west,
                legend cell align=left,
                legend style={font=\footnotesize},
                grid=major,
            ]
            \addplot[blue, thick, mark=*] coordinates {
                    (1,9.41573) (2,6.07453) (3,3.25052) (4,2.64941) (5,1.48594)
                };
            \addlegendentry{$p=1$}
            \addplot[red, thick, mark=square*] coordinates {
                    (1,14.2025) (2,2.77918) (3,0.374031) (4,0.0514442) (5,0.00532243)
                };
            \addlegendentry{$p=2$}
            \addplot[teal, thick, mark=triangle*] coordinates {
                    (1,6.41491) (2,0.561008) (3,0.028134) (4,0.00154552) (5,8.00081e-05)
                };
            \addlegendentry{$p=3$}
            \addplot[blue, thick, dashed, mark=o] coordinates {
                    (1,4.19005) (2,3.74981) (3,1.64512) (4,0.523096) (5,0.144226)
                };
            \addplot[red, thick, dashed, mark=square] coordinates {
                    (1,2.1131) (2,0.174262) (3,0.0211216) (4,0.00246798) (5,0.0002717)
                };
            \addplot[teal, thick, dashed, mark=triangle] coordinates {
                    (1,0.279974) (2,0.0122881) (3,0.000561646) (4,2.77411e-05) (5,1.4171e-06)
                };
            \addplot[gray, dashed, domain=1:5, samples=2] {60*4^(-x)}
            node[pos=0.85, above, sloped, font=\footnotesize] {$h^{2}$};
            \addplot[gray, dashed, domain=1:5, samples=2] {40*8^(-x)}
            node[pos=0.85, above, sloped, font=\footnotesize] {$h^{3}$};
            \addplot[gray, dashed, domain=1:5, samples=2] {8*16^(-x)}
            node[pos=0.85, above, sloped, font=\footnotesize] {$h^{4}$};
        \end{semilogyaxis}
    \end{tikzpicture}%
}\hfill
\subcaptionbox{\modelBshort{} model $\Psi_2$\label{fig:discconvergence:psi2}}{%
    \begin{tikzpicture}
        \begin{semilogyaxis}[
                width=0.48\textwidth, height=6cm,
                xlabel={number of refinements},
                ylabel={$\|\bm u_h^{\mathrm c}-\bm u_h^{\mathrm m}\|_{L^2}$},
                xtick={1,2,3,4,5},
                xmax=5.2,
                legend pos=south west,
                legend cell align=left,
                legend style={font=\footnotesize},
                grid=major,
            ]
            \addplot[blue, thick, mark=*] coordinates {
                    (1,8.5036) (2,5.18564) (3,3.40128) (4,3.82411) (5,3.36728)
                };
            \addlegendentry{$p=1$}
            \addplot[red, thick, mark=square*] coordinates {
                    (1,13.6113) (2,4.7096) (3,1.5063) (4,0.400188) (5,0.0627609)
                };
            \addlegendentry{$p=2$}
            \addplot[teal, thick, mark=triangle*] coordinates {
                    (1,8.67143) (2,2.54796) (3,0.434074) (4,0.0449883) (5,0.00287657)
                };
            \addlegendentry{$p=3$}
            \addplot[blue, thick, dashed, mark=o] coordinates {
                    (1,3.26072) (2,2.53315) (3,1.08315) (4,0.342479) (5,0.0944233)
                };
            \addplot[red, thick, dashed, mark=square] coordinates {
                    (1,2.10101) (2,0.194234) (3,0.0208533) (4,0.00242529) (5,0.000276216)
                };
            \addplot[teal, thick, dashed, mark=triangle] coordinates {
                    (1,0.22688) (2,0.0122334) (3,0.000630999) (4,3.23418e-05) (5,1.66889e-06)
                };
            \addplot[gray, dashed, domain=1:5, samples=2] {60*4^(-x)}
            node[pos=0.85, above, sloped, font=\footnotesize] {$h^{2}$};
            \addplot[gray, dashed, domain=1:5, samples=2] {40*8^(-x)}
            node[pos=0.85, above, sloped, font=\footnotesize] {$h^{3}$};
            \addplot[gray, dashed, domain=1:5, samples=2] {8*16^(-x)}
            node[pos=0.85, above, sloped, font=\footnotesize] {$h^{4}$};
        \end{semilogyaxis}
    \end{tikzpicture}%
}
    \caption{Convergence of the displacement difference $\|\bm u_h^{\mathrm c}-\bm u_h^{\mathrm m}\|_{L^2}$ for the all-Dirichlet disc test case
        ($p=1,2,3$, Poisson ratio $\nu = 0.45$).
        (a)~neo-Hookean model $\Psi_1$: solid lines load factor $0.5$, dashed lines load factor $0.1$;
        reference slopes $h^{p+1}$ shown in gray.
        (b)~\modelBshort{} model $\Psi_2$: solid lines load factor $0.5$, dashed lines load factor $0.1$;
        reference slopes $h^{p+1}$ shown in gray.
        Theorem~\ref{thm:brr} predicts the optimal rates $O(h^{p+1})$ for $p \geq 2$; see the text for the
        $p = 1$ behaviour at load factor $0.5$.}
    \label{fig:discconvergence}
\end{figure}

\subsection{Pole under horizontal load: the corner cap}
\label{sec:pole}
We now turn to a geometry that deliberately contains corners: the pole $\Omega_{\sqcap}$ loaded by the body force $\bm f_0^{\,\sqcap} = \loadf\,\bm b_{\sqcap}$ of
\eqref{eq:rhs}. On the same geometry we consider two boundary-condition settings, which isolate
the two corner types analysed in Section~\ref{sec:mixedjunction}.

In the \emph{mixed} setting, homogeneous Dirichlet conditions (clamping) are imposed only on the flat bottom
$\Gamma_D$, while the remainder $\Gamma_N = \partial\Omega \setminus \Gamma_D$ is traction-free. The two bottom
corners, where $\Gamma_D$ meets $\Gamma_N$, are sharp corners between straight segments; the boundary integrals on the
corner-containing cells are split \emph{exactly} at the corner points, so the Nitsche terms act precisely on
$\Gamma_D$, while the traction-free $\Gamma_N$ is natural and carries no boundary term. In the \emph{pure-Dirichlet}
setting, the same homogeneous condition is imposed weakly on the entire boundary, so the two bottom corners become
clamped--clamped right-angle corners.

Both settings degrade the convergence, but to different extents. At the mixed Dirichlet--Neumann junctions the singular
exponent satisfies $\lambda(\pi/2,\nu) < 1$, capping the $L^2$ rate at $2\lambda < 2$ \emph{independently of the
    polynomial degree}: the $p = 1, 2, 3$ curves collapse onto a single sub-optimal slope. At the clamped--clamped corners
of the pure-Dirichlet setting the exponent satisfies $\lambda \in (1, 1.63)$; the duality gain then saturates at one
full order, so the cap $\min(p+1,\, 1+\lambda)$ with $1+\lambda \in (2, 2.63)$ is milder: first-order elements ($p =
    1$) recover the optimal rate, whereas $p \geq 2$ are capped below $p+1$ and gain no higher-order convergence; the right
panel makes this failure to improve with degree explicit. Neither cap can be removed by refining uniformly, by aligning
the corner with a mesh vertex, or by tuning the stabilisation; the quantitative dependence on $\nu$ for both corner
types, measured in linear elasticity, is reported in Figure~\ref{fig:nusweep}.

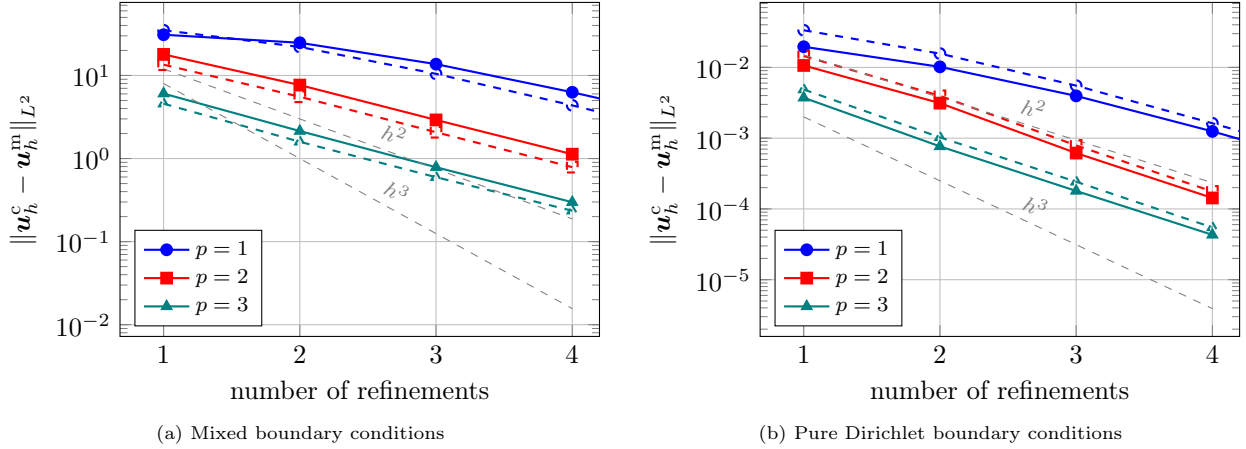
\begin{figure}[h]
    \centering
    \subcaptionbox{Mixed boundary conditions\label{fig:poleconvergence:mixed}}{%
    \begin{tikzpicture}
        \begin{semilogyaxis}[
                width=0.48\textwidth, height=6cm,
                xlabel={number of refinements},
                ylabel={$\|\bm u_h^{\mathrm c}-\bm u_h^{\mathrm m}\|_{L^2}$},
                xtick={1,2,3,4},
                xmax=4.2,
                legend pos=south west,
                legend cell align=left,
                legend style={font=\footnotesize},
                grid=major,
            ]
            \addplot[blue, thick, mark=*] coordinates {
                    (1,30.932) (2,24.7521) (3,13.6955) (4,6.28091) (5,2.66578)
                };
            \addplot[red, thick, mark=square*] coordinates {
                    (1,17.9329) (2,7.66447) (3,2.91658) (4,1.12803)
                };
            \addplot[teal, thick, mark=triangle*] coordinates {
                    (1,6.05183) (2,2.14147) (3,0.786099) (4,0.298213)
                };
            \addplot[blue, thick, dashed, mark=o] coordinates {
                    (1,35.0145) (2,22.0442) (3,10.4569) (4,4.36494) (5,1.7286)
                };
            \addplot[red, thick, dashed, mark=square] coordinates {
                    (1,13.4797) (2,5.54249) (3,2.07356) (4,0.787864)
                };
            \addplot[teal, thick, dashed, mark=triangle] coordinates {
                    (1,4.57338) (2,1.59177) (3,0.596209) (4,0.235441)
                };
            \addplot[gray, dashed, domain=1:4, samples=2] {3*4^(2-x)}
            node[pos=0.55, above, sloped, font=\footnotesize] {$h^2$};
            \addplot[gray, dashed, domain=1:4, samples=2] {8^(2-x)}
            node[pos=0.55, above, sloped, font=\footnotesize] {$h^3$};
            \addlegendimage{blue, thick, mark=*}
            \addlegendentry{$p=1$}
            \addlegendimage{red, thick, mark=square*}
            \addlegendentry{$p=2$}
            \addlegendimage{teal, thick, mark=triangle*}
            \addlegendentry{$p=3$}
        \end{semilogyaxis}
    \end{tikzpicture}%
}\hfill
\subcaptionbox{Pure Dirichlet boundary conditions\label{fig:poleconvergence:dirichlet}}{%
    \begin{tikzpicture}
        \begin{semilogyaxis}[
                width=0.48\textwidth, height=6cm,
                xlabel={number of refinements},
                ylabel={$\|\bm u_h^{\mathrm c}-\bm u_h^{\mathrm m}\|_{L^2}$},
                xtick={1,2,3,4},
                xmax=4.2,
                legend pos=south west,
                legend cell align=left,
                legend style={font=\footnotesize},
                grid=major,
            ]
            \addplot[blue, thick, mark=*] coordinates {
                    (1,0.0195804) (2,0.0101803) (3,0.00395486) (4,0.00124686) (5,0.000354988)
                };
            \addplot[red, thick, mark=square*] coordinates {
                    (1,0.0106545) (2,0.0031265) (3,0.000615684) (4,0.000142821)
                };
            \addplot[teal, thick, mark=triangle*] coordinates {
                    (1,0.00372694) (2,0.000765502) (3,0.000179587) (4,4.30407e-05)
                };
            \addplot[blue, thick, dashed, mark=o] coordinates {
                    (1,0.0334751) (2,0.0156684) (3,0.0055224) (4,0.0016235) (5,0.000454916)
                };
            \addplot[red, thick, dashed, mark=square] coordinates {
                    (1,0.0144796) (2,0.00396658) (3,0.00078978) (4,0.000176533)
                };
            \addplot[teal, thick, dashed, mark=triangle] coordinates {
                    (1,0.00486622) (2,0.00102674) (3,0.000242178) (4,5.47558e-05)
                };
            \addplot[gray, dashed, domain=1:4, samples=2] {0.015*4^(1-x)}
            node[pos=0.55, above, sloped, font=\footnotesize] {$h^2$};
            \addplot[gray, dashed, domain=1:4, samples=2] {0.002*8^(1-x)}
            node[pos=0.55, above, sloped, font=\footnotesize] {$h^3$};
            \addlegendimage{blue, thick, mark=*}
            \addlegendentry{$p=1$}
            \addlegendimage{red, thick, mark=square*}
            \addlegendentry{$p=2$}
            \addlegendimage{teal, thick, mark=triangle*}
            \addlegendentry{$p=3$}
        \end{semilogyaxis}
    \end{tikzpicture}%
}
    \caption{Convergence of the difference $\|\bm u_h^{\mathrm c}-\bm u_h^{\mathrm m}\|_{L^2}$ against a body-fitted
        reference for the pole test case (perturbed background grid), for the neo-Hookean model $\Psi_1$
        (solid lines) and linear elasticity (dashed lines).
        \emph{(a)~Mixed} Dirichlet--Neumann boundary conditions: all degrees $p=1,2,3$ collapse onto a single
        sub-optimal slope, the corner cap $2\lambda(\pi/2,\nu)<2$ of \eqref{eq:l2cap}.
        \emph{(b)~Pure-Dirichlet} boundary conditions: the milder clamped--clamped cap $1+\lambda\in(2,2.63)$ leaves
        $p=1$ optimal while $p\ge2$ saturate below $p+1$. Reference slopes $h^2$ and $h^3$ are shown in gray.}
    \label{fig:poleconvergence}
\end{figure}

\subsection{Conditioning: cut independence, scaling in $h$, and the stabilisation parameters}
\label{sec:condnumber}

The condition number $\operatorname{cond}(\bm K)$ tests the combined coercivity and continuity bounds of
Proposition~\ref{prop:coercivity}. We evaluate it at $\bm u=\bm 0$ while translating the background grid through the
small-cut limit (Figure~\ref{fig:condnumber}a). With the ghost penalty, $\operatorname{cond}(\bm K)$ has no trend with
the cut position and varies by at most a factor $3.2$; without stabilisation the tangent loses definiteness. Over the
same translation sweep, the maximum-to-minimum variation in the $L^2$ error is $13\%$ for the disc and $4\%$ for the
pole at $p=1$, and $20\%$ for the disc and $28\%$ for the pole at $p=2$. At the
worst observed cut, uniform refinement gives rates close to $h^{-2}$ for $p=1,2$ on both geometries
(Figure~\ref{fig:condnumber}b), as predicted by Corollary~\ref{cor:conditioning}.

\begin{figure}[h]
    \centering
\subcaptionbox{Cut position: $\operatorname{cond}(\bm K)$ vs.\ grid shift\label{fig:condnumber:shift}}{%
    \begin{tikzpicture}
        \begin{semilogyaxis}[
                width=0.50\textwidth, height=6.3cm,
                xlabel={grid shift $\delta/h$},
                ylabel={$\operatorname{cond}(\bm K)$},
                xtick={0,0.1,0.2,0.3,0.4,0.5},
                xmin=-0.02, xmax=0.52,
                ymin=60, ymax=1e8,
                legend pos=north west,
                legend cell align=left,
                legend style={font=\footnotesize},
                grid=major,
            ]
            \addplot[blue, thick, mark=*, mark size=1.4] coordinates {
                    (0.00,111.9) (0.05,111.2) (0.10,117.5) (0.15,104.5) (0.20,103.1)
                    (0.25,99.3) (0.30,99.7) (0.35,99.0) (0.40,97.6) (0.45,99.6) (0.50,99.4)
                };
            \addplot[red, thick, mark=square*, mark size=1.2] coordinates {
                    (0.00,3029) (0.05,3380) (0.10,3431) (0.15,4320) (0.20,6266)
                    (0.25,4975) (0.30,3945) (0.35,4549) (0.40,3399) (0.45,4116) (0.50,3360)
                };
            \addplot[teal, thick, mark=triangle*, mark size=1.6] coordinates {
                    (0.00,3.005e5) (0.05,3.949e5) (0.10,4.383e5) (0.15,6.222e5) (0.20,9.652e5)
                    (0.25,6.987e5) (0.30,5.076e5) (0.35,6.402e5) (0.40,3.990e5) (0.45,5.330e5)
                    (0.50,3.808e5)
                };
            \addplot[blue, thick, dashed, mark=o, mark size=1.4] coordinates {
                    (0.00,1.282e5) (0.05,1.345e5) (0.10,1.327e5) (0.15,1.338e5) (0.20,1.337e5)
                    (0.25,1.336e5) (0.30,1.335e5) (0.35,1.334e5) (0.40,1.334e5) (0.45,1.334e5)
                    (0.50,1.334e5)
                };
            \addplot[red, thick, dashed, mark=square, mark size=1.2] coordinates {
                    (0.00,1.437e6) (0.05,1.517e6) (0.10,1.494e6) (0.15,1.508e6) (0.20,1.505e6)
                    (0.25,1.502e6) (0.30,1.499e6) (0.35,1.495e6) (0.40,1.492e6) (0.45,1.488e6)
                    (0.50,1.484e6)
                };
            \addlegendimage{thick, solid, mark=*, blue}
            \addlegendentry{$p=1$}
            \addlegendimage{thick, solid, mark=square*, red}
            \addlegendentry{$p=2$}
            \addlegendimage{thick, solid, mark=triangle*, teal}
            \addlegendentry{$p=3$}
        \end{semilogyaxis}
    \end{tikzpicture}%
}\hfill
\subcaptionbox{Refinement: $\operatorname{cond}(\bm K)$ vs.\ mesh size\label{fig:condnumber:refinement}}{%
    \begin{tikzpicture}
        \begin{semilogyaxis}[
                width=0.50\textwidth, height=6.3cm,
                xlabel={number of refinements},
                ylabel={$\operatorname{cond}(\bm K)$},
                xtick={1,2,3,4,5},
                xmin=0.8, xmax=5.2,
                ymin=30, ymax=6e8,
                legend pos=north west,
                legend cell align=left,
                legend style={font=\footnotesize},
                grid=major,
            ]
            \addplot[blue, thick, mark=*, mark size=1.4] coordinates {
                    (1,71.84) (2,103.1) (3,417.0) (4,1516) (5,6969)
                };
            \addplot[red, thick, mark=square*, mark size=1.2] coordinates {
                    (1,4.30e3) (2,6.27e3) (3,8.57e3) (4,3.42e4) (5,1.37e5)
                };
            \addplot[teal, thick, mark=triangle*, mark size=1.6] coordinates {
                    (1,6.13e5) (2,9.65e5) (3,6.58e5) (4,9.06e5) (5,3.42e6)
                };
            \addplot[blue, thick, dashed, mark=o, mark size=1.4] coordinates {
                    (1,3.633e4) (2,1.337e5) (3,4.653e5) (4,1.705e6) (5,6.466e6)
                };
            \addplot[red, thick, dashed, mark=square, mark size=1.2] coordinates {
                    (1,4.819e5) (2,1.505e6) (3,5.311e6) (4,1.982e7) (5,7.594e7)
                };
            \addplot[gray, dashed, mark=none, domain=1.6:5, samples=2] {6*4^x}
            node[pos=0.45, below, sloped, font=\footnotesize] {$h^{-2}$};
            \addplot[gray, dashed, mark=none, domain=2:5, samples=2] {1.3e6*4^(x-1)}
            node[pos=0.5, above, sloped, font=\footnotesize] {$h^{-2}$};
            \addlegendimage{thick, solid, mark=*, blue}
            \addlegendentry{$p=1$}
            \addlegendimage{thick, solid, mark=square*, red}
            \addlegendentry{$p=2$}
            \addlegendimage{thick, solid, mark=triangle*, teal}
            \addlegendentry{$p=3$}
        \end{semilogyaxis}
    \end{tikzpicture}%
}
    \caption{Condition number at $\bm u=\bm 0$ with $\gamma_A=\mu_e$: disc (solid, filled markers) and pole
        (dashed, open markers). \emph{(a)}~Grid-translation sweep in the $x$ direction at refinement level~2; $\operatorname{cond}(\bm K)$ varies by at most a
        factor $3.2$ and does not deteriorate near small cuts. \emph{(b)}~Uniform refinement with
        reference slope $h^{-2}$. Rates over the three finest levels are $2.03,2.00$ for the disc and $1.90,1.92$
        for the pole at $p=1,2$; the computed $p=3$ rates are lower, consistent with the upper bound of
        Corollary~\ref{cor:conditioning}.}
    \label{fig:condnumber}
\end{figure}
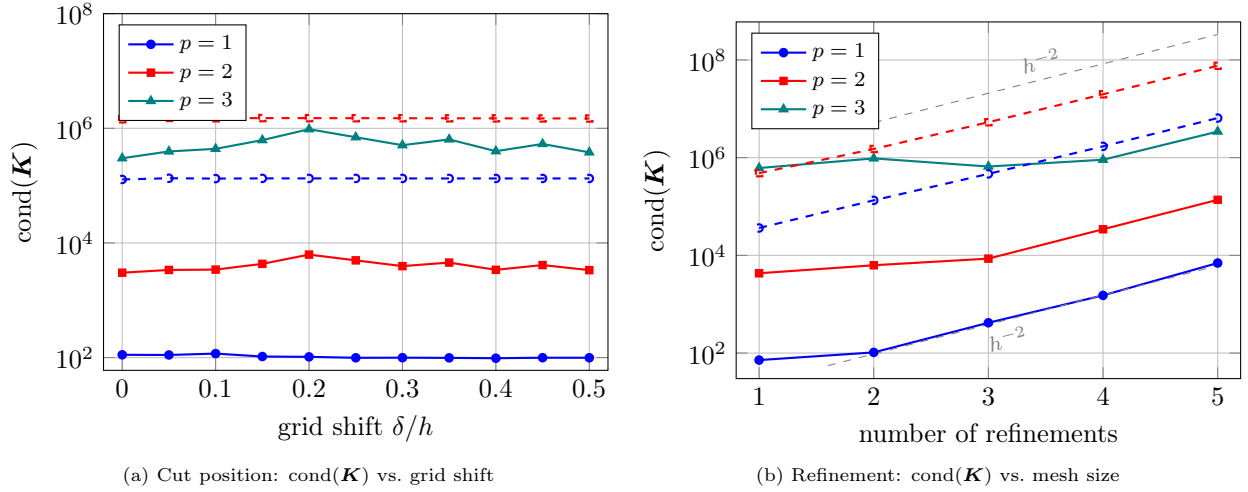

At the worst cut we also sweep $\gamma_A$ and $\gamma_D=m(p+1)p$ (Figure~\ref{fig:penaltysensitivity}). Below a
degree-dependent threshold the tangent becomes indefinite; above it, the error and condition number vary slowly over a
wide plateau. This threshold is not confined to the small-cut limit: already at a cut of finite fraction there is a
strictly positive weight $\gamma^*$ below which the stabilisation no longer controls the part of the cut cells lying
outside $\Omega$, the negative Nitsche consistency terms in \eqref{eq:kbound} are no longer absorbed, and the tangent
loses positive definiteness. Since its eigenvalues depend continuously on the cut position, any weight fixed below
$\gamma^*$ makes the translation sweep of Figure~\ref{fig:condnumber}a pass through a singular configuration, just as
the unstabilised case does. An aligned two-cell model in \cite{wichrowski2026multigridcutfem} makes this quantitative
and shows that the admissible weight decreases with the polynomial degree like $p^{-2}$. The ordering in
Figure~\ref{fig:penaltysensitivity}a is consistent with it: the $p=1$ threshold lies between $\gamma_A=10^{-1}\mu_e$
and $3\cdot10^{-1}\mu_e$, above the $p=2$ threshold, which lies below $10^{-1}\mu_e$. The defaults lie within this
plateau: $\gamma_A=\mu_e$ is near the minimum of $\operatorname{cond}(\bm
    K)$, and $m=5$ is about $2.5$ times the Nitsche coercivity threshold.

\begin{figure}[h]
    \centering
\subcaptionbox{Ghost penalty $\gamma_A$\label{fig:penalty:ghost}}{%
    \begin{tikzpicture}
        \begin{axis}[
                width=0.33\textwidth, height=5.4cm, scale only axis,
                xmode=log, ymode=log,
                xmin=0.01, xmax=2000,
                ymin=0.3, ymax=300,
                xtick={0.01,0.1,1,10,100,1000},
                xlabel={$\gamma_A/\mu_e$},
                ylabel={$\|\bm u_h^{\mathrm c}-\bm u_h^{\mathrm m}\|_{L^2}$},
                axis y line*=left,
                grid=major,
                legend pos=north west,
                legend cell align=left,
                legend style={font=\footnotesize},
            ]

            \draw[gray, densely dotted] (axis cs:1,0.3) -- (axis cs:1,300);
            \addplot[blue, thick, mark=*, mark size=1.4] coordinates {
                    (0.3,18.7) (1,27.4) (3,39.0) (10,59.3) (100,119.3) (1000,192.8)
                };
            \addplot[red, thick, mark=square*, mark size=1.2] coordinates {
                    (0.1,0.391) (0.3,0.595) (1,1.003) (3,1.844) (10,3.971)
                    (100,20.66) (1000,71.06)
                };
            \addlegendimage{thick, solid, mark=*, blue}
            \addlegendentry{$p=1$}
            \addlegendimage{thick, solid, mark=square*, red}
            \addlegendentry{$p=2$}
        \end{axis}
        \begin{axis}[
                width=0.33\textwidth, height=5.4cm, scale only axis,
                xmode=log, ymode=log,
                xmin=0.01, xmax=2000,
                ymin=60, ymax=4e6,
                axis x line=none,
                axis y line*=right,
                ylabel={$\operatorname{cond}(\bm K)$},
            ]
            \addplot[blue, thick, dashed, mark=o, mark size=1.4] coordinates {
                    (0.3,160) (1,103) (3,194) (10,527) (100,4528) (1000,3.912e4)
                };
            \addplot[red, thick, dashed, mark=square, mark size=1.2] coordinates {
                    (0.1,7533) (0.3,5367) (1,6266) (3,10603) (10,25287)
                    (100,1.924e5) (1000,1.781e6)
                };
        \end{axis}
    \end{tikzpicture}%
}\hfill
\subcaptionbox{Nitsche penalty $\gamma_D$\label{fig:penalty:nitsche}}{%
    \begin{tikzpicture}
        \begin{axis}[
                width=0.33\textwidth, height=5.4cm, scale only axis,
                xmode=log, ymode=log,
                xmin=0.7, xmax=1600,
                ymin=0.8, ymax=60,
                xtick={1,2,5,10,50,200,1000},
                xticklabels={$1$,$2$,$5$,$10$,$50$,$200$,$1000$},
                xlabel={$m$, \; $\gamma_D=m\,(p+1)p$},
                ylabel={$\|\bm u_h^{\mathrm c}-\bm u_h^{\mathrm m}\|_{L^2}$},
                axis y line*=left,
                grid=major,
            ]
            \draw[gray, densely dotted] (axis cs:5,0.8) -- (axis cs:5,60);
            \addplot[blue, thick, mark=*, mark size=1.4] coordinates {
                    (2,24.03) (3,25.08) (5,27.40) (8,28.55) (12,29.20) (20,29.76)
                    (50,30.59) (200,32.72) (1000,40.79)
                };
            \addplot[red, thick, mark=square*, mark size=1.2] coordinates {
                    (2,1.0249) (3,1.0009) (5,1.0033) (8,1.0081) (12,1.0121) (20,1.0167)
                    (50,1.0249) (200,1.0396) (1000,1.0616)
                };
        \end{axis}
        \begin{axis}[
                width=0.33\textwidth, height=5.4cm, scale only axis,
                xmode=log, ymode=log,
                xmin=0.7, xmax=1600,
                ymin=60, ymax=4e5,
                axis x line=none,
                axis y line*=right,
                ylabel={$\operatorname{cond}(\bm K)$},
            ]
            \addplot[blue, thick, dashed, mark=o, mark size=1.4] coordinates {
                    (2,1281) (3,131.9) (5,103.1) (8,119.3) (12,145.6) (20,203.9)
                    (50,450.6) (200,1765) (1000,8480)
                };
            \addplot[red, thick, dashed, mark=square, mark size=1.2] coordinates {
                    (2,10451) (3,6272) (5,6266) (8,6550) (12,7075) (20,8304)
                    (50,13570) (200,41777) (1000,1.931e5)
                };
        \end{axis}
    \end{tikzpicture}%
}
    \caption{Parameter sensitivity for the disc at refinement level~2: $L^2$ error (solid) and
        $\operatorname{cond}(\bm K)$ (dashed). Dotted lines mark the defaults; shaded ranges mark indefinite tangents.
        \emph{(a)}~Ghost penalty: $\operatorname{cond}(\bm K)$ is minimised near $\gamma_A=\mu_e$, while the error grows with $\gamma_A$;
        the coercivity threshold is degree dependent, and the $p=1$ curves terminate one point later than the $p=2$
        curves because the $p=1$ tangent is already indefinite at $\gamma_A=10^{-1}\mu_e$.
        \emph{(b)}~Nitsche penalty: $\operatorname{cond}(\bm K)$ is minimised at $m=5$, with the coercivity threshold between $m=1$ and
        $m=2$ and nearly constant error near the default.}
    \label{fig:penaltysensitivity}
\end{figure}
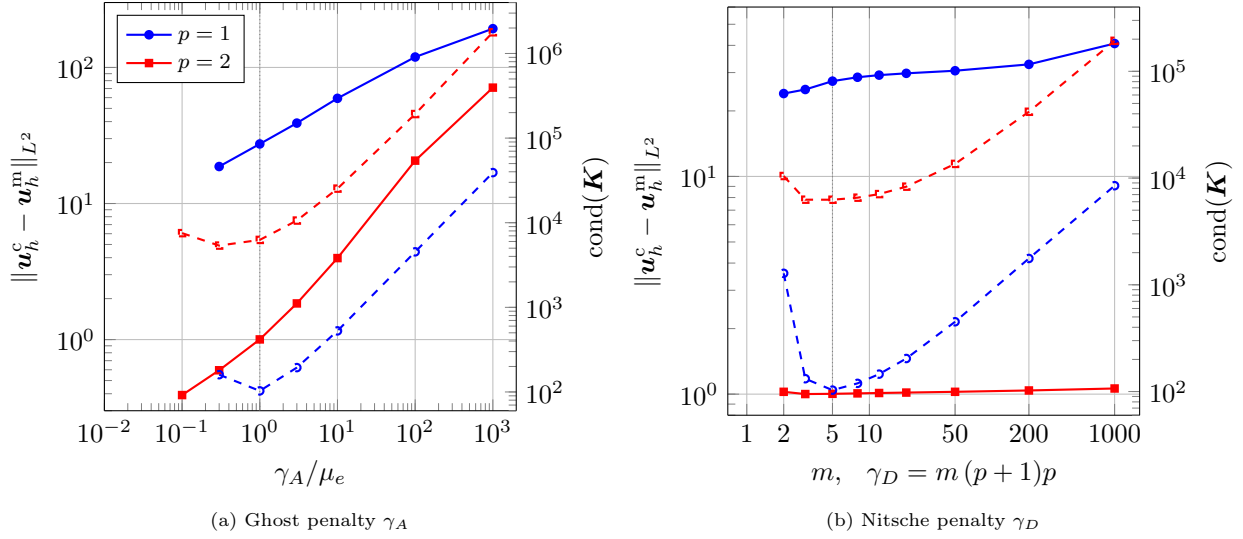

\section{Accuracy limit at unfitted Dirichlet--Neumann junctions}
\label{sec:mixedjunction}

The estimates of Section~\ref{sec:theory} assume sufficient regularity of the exact solution, concretely, the bounded,
non-degenerate deformation gradient that supplies the Lipschitz-tangent assumption~(J1) of Theorem~\ref{thm:brr}
(Assumption~\ref{ass:ellipticity}). This fails in a specific and practically relevant situation: mixed boundary
conditions, where the Dirichlet part $\Gamma_D$ and the Neumann part $\Gamma_N$ of $\partial\Omega$ meet at a junction
point $\bm{x}_0$, at which the singular expansion \eqref{eq:cornerexpansion} makes $|\Grad\bm{u}|$ unbounded and
assumption~(J1) inapplicable. As observed in the pole test case of Section~\ref{sec:pole}
(Figure~\ref{fig:poleconvergence}), the $L^2$ convergence rate drops to a reduced value, independently of the
polynomial degree $p$. In this section we show that this degradation is not a defect of the implementation but an
approximation-theoretic obstruction: no quadrature, stabilisation, penalty tuning, or alignment of the background mesh
with the junction can remove it on quasi-uniform meshes.

We first consider whether the boundary-condition type is assigned segment-wise on cut cells, effectively moving the
junction to the nearest cell boundary. Such an assignment would introduce a consistency error: the Dirichlet penalty
would act on a spurious arc $\gamma_h \subset \Gamma_N$ of length $O(h)$, producing, by a standard Strang-type
argument, a consistency error of order $h\,|\Grad\bm{u}(\bm{x}_0)|$ in the energy norm with no duality gain, hence a
first-order cap in $L^2$ for every $p$. In our implementation this mechanism is excluded by construction: the junction
is a sharp corner between two straight boundary segments, and the boundary integrals in the junction cell are split
\emph{exactly} at $\bm{x}_0$, so that the Nitsche terms are evaluated only on the $\Gamma_D$ portion of the boundary,
while the traction-free $\Gamma_N$ portion carries no boundary term. The discrete forms are therefore consistent, and
the observed loss of accuracy must have a different origin.

That origin is the regularity of the exact solution. Near a Dirichlet--Neumann junction, when the leading singular
exponent is real, the solution of the (linearised) elasticity problem admits the local Kondratiev-type expansion
\cite{grisvard1985elliptic, rossle2000corner}
\begin{equation}
    \label{eq:cornerexpansion}
    \bm{u}(r, \theta) = \Lambda \, r^{\lambda} \, \bm{\Phi}(\theta) + \bm{w},
    \qquad \bm{w} \in \big(H^{1+\lambda+\varepsilon}\big)^d \text{ locally, } \varepsilon > 0,
\end{equation}
Here $\bm\Phi(\theta)$ is the vector-valued angular eigenfunction associated with $\lambda$, with its normalisation
absorbed into the stress-intensity factor $\Lambda$. The polar coordinates are centred at $\bm{x}_0$, and the real
singular exponent $\lambda = \lambda(\omega, \nu) \in (0, 1]$ is the smallest positive root of the characteristic
equation of the Lam\'e operator with mixed boundary conditions at
the interior opening angle $\omega$ \cite{rossle2000corner}. Two cases deserve mention. If the junction lies on a
straight part of the boundary ($\omega = \pi$), the leading roots form a complex conjugate pair with
$\operatorname{Re}\lambda = 1/2$ regardless of the material parameters; the corresponding real mode and its phase are
described below. At a right-angle corner ($\omega = \pi/2$) the leading exponent is real, depends on $\nu$, and is in
general strictly smaller than for the Laplacian. In either case the regularity threshold is governed by
$\operatorname{Re}\lambda$: $\bm{u} \notin H^{1+\operatorname{Re}\lambda+\varepsilon}$ in any neighbourhood of
$\bm{x}_0$, and the generic stress intensity factor satisfies $\Lambda \neq 0$.

The following proposition turns this into an unconditional lower bound on the discretisation error. We state it in $d =
    2$, matching our experiments. We call the junction \emph{unresolved} when it is not resolved by the mesh in the sense
that it lies strictly inside a single active cell, at distance at least $\rho h$ from that cell's boundary, so that the
singularity \eqref{eq:cornerexpansion} must be reproduced by one polynomial on that cell; on a quasi-uniform background
mesh this is the generic situation, since the junction is placed by the geometry and not by the mesh.

\begin{proposition}[Lower bound at an unresolved junction]
    \label{prop:junctionlower}
    Let $\bm{u}$ satisfy \eqref{eq:cornerexpansion} with $\Lambda \neq 0$ and a real exponent $\lambda \notin
        \mathbb{N}$, $0 < \lambda < p$ (at a mixed junction $\lambda < 1$; at the clamped--clamped corner of
    Section~\ref{sec:nudependence}, $\lambda \in (1,2)$), and suppose the junction
    $\bm{x}_0$ lies in the interior of a cell $K^* \in T^h_\Omega$ with $\mathrm{dist}(\bm{x}_0, \partial K^*) \geq \rho h$
    for some fixed $\rho > 0$. Then there exist $c > 0$ and $h_1 > 0$, independent of $h$ and of the discrete solution,
    such that for all $h \leq h_1$,
    \begin{equation}
        \| \Grad (\bm{u} - \bm{u}_h) \|_{\Omega} \;\geq\; c \, |\Lambda| \, h^{\lambda}
        \qquad \text{for every } \bm{u}_h \in \mathbb{V}_h .
    \end{equation}
\end{proposition}

\begin{proof}
    Let $B := B_{\rho h}(\bm{x}_0) \subset K^*$. The restriction of $\bm{u}_h$ to $K^*$ belongs to the local polynomial
    space $[\mathcal{Q}_p]^2$, the tensor-product space of the $\mathbb{Q}_p$ elements used in this work; for
    simplicial $\mathcal{P}_p$ elements the argument below is identical, so
    \begin{equation*}
        \| \Grad (\bm{u} - \bm{u}_h) \|_{\Omega} \geq \| \Grad (\bm{u} - \bm{u}_h) \|_{B \cap \Omega}
        \geq \inf_{\bm{q} \in [\mathcal{Q}_p]^2} \| \Grad (\bm{u} - \bm{q}) \|_{B \cap \Omega} .
    \end{equation*}
    Because the two boundary segments meeting at $\bm{x}_0$ are straight, $\Omega$ coincides near $\bm{x}_0$ with a cone
    of opening $\omega$, which is invariant under the scaling $\bm{x} = \bm{x}_0 + h \hat{\bm{x}}$. Set
    \begin{equation*}
        d_0 := \inf_{\bm{q} \in [\mathcal{Q}_p]^2} \big\| \Grad \big( r^\lambda \bm{\Phi} - \bm{q} \big) \big\|_{B_\rho(0) \cap \widehat{\Omega}} > 0 ,
    \end{equation*}
    which is strictly positive since $r^\lambda \bm{\Phi}$ with $\lambda \notin \mathbb{N}$ is not a polynomial, and is a
    fixed constant independent of $h$. Under the scaling, $\| \Grad v \|_{B_{\rho h} \cap \Omega} = \|
        \widehat{\Grad} \hat{v} \|_{B_\rho \cap \widehat{\Omega}}$ in $d = 2$, and the singular term transforms as
    $\widehat{r^\lambda \bm{\Phi}} = h^\lambda \, \hat{r}^\lambda \bm{\Phi}$, while the translation--dilation maps
    $[\mathcal{Q}_p]^2$ onto itself. Hence
    \begin{equation*}
        \inf_{\bm{q}} \| \Grad ( \Lambda r^\lambda \bm{\Phi} - \bm{q} ) \|_{B \cap \Omega} = |\Lambda| \, d_0 \, h^{\lambda} .
    \end{equation*}
    For the remainder, the Bramble--Hilbert lemma on the scaled domain gives
    $$\inf_{\bm{q}} \| \Grad (\bm{w} - \bm{q}) \|_{B \cap \Omega} \leq C h^{\lambda + \varepsilon} \, |\bm{w}|_{H^{1+\lambda+\varepsilon}(B \cap \Omega)},$$
    the bound over $[\mathcal{P}_p]^2$ holding a fortiori over the larger space $[\mathcal{Q}_p]^2$ (here $\lambda +
        \varepsilon \leq p$ is used).
    The triangle inequality then yields
    \begin{equation*}
        \inf_{\bm{q}} \| \Grad (\bm{u} - \bm{q}) \|_{B \cap \Omega}
        \geq |\Lambda| \, d_0 \, h^{\lambda} - C h^{\lambda + \varepsilon}
        \geq \tfrac{1}{2} |\Lambda| \, d_0 \, h^{\lambda}
    \end{equation*}
    for all $h \leq h_1$ with $h_1$ small enough.
\end{proof}

The same scaling argument yields an unconditional lower bound in $L^2$ as well, with a different exponent.

\begin{corollary}[$L^2$ lower bound at an unresolved junction]
    \label{cor:junctionlowerL2}
    Under the assumptions of Proposition~\ref{prop:junctionlower}, there exist $c > 0$ and $h_1 > 0$, independent of
    $h$ and of the discrete solution, such that for all $h \leq h_1$,
    \begin{equation}
        \| \bm{u} - \bm{u}_h \|_{L^2(\Omega)} \;\geq\; c \, |\Lambda| \, h^{1+\lambda}
        \qquad \text{for every } \bm{u}_h \in \mathbb{V}_h .
    \end{equation}
\end{corollary}

\begin{proof}
    The proof of Proposition~\ref{prop:junctionlower} applies verbatim with $\| \Grad (\cdot) \|_{B \cap \Omega}$
    replaced by $\| \cdot \|_{B \cap \Omega}$: in $d = 2$ the scaling $\bm{x} = \bm{x}_0 + h \hat{\bm{x}}$ gives
    $\| v \|_{B_{\rho h} \cap \Omega} = h \, \| \hat{v} \|_{B_\rho \cap \widehat{\Omega}}$, so the singular term
    contributes $|\Lambda| \, d_0' \, h^{1+\lambda}$ with $d_0' := \inf_{\bm{q} \in [\mathcal{Q}_p]^2} \| r^\lambda
        \bm{\Phi} - \bm{q} \|_{B_\rho(0) \cap \widehat{\Omega}} > 0$, while the Bramble--Hilbert bound for the remainder
    becomes $C h^{1+\lambda+\varepsilon}$.
\end{proof}

\paragraph{Remark (complex exponents)} At a straight mixed junction the leading exponent is the complex pair $\lambda = \tfrac12 \pm i\varepsilon$ of
Section~\ref{sec:nudependence}, and the real singular mode is $r^{1/2}\big[ \cos(\varepsilon\ln r)\,\bm\Phi_1(\theta) +
        \sin(\varepsilon\ln r)\,\bm\Phi_2(\theta)\big]$. Under the scaling $\bm x = \bm x_0 + h\hat{\bm x}$ this mode is
reproduced only up to the phase shift $\varepsilon\ln h$, so the scaled infimum is no longer the single constant $d_0$.
Redefining
\begin{equation*}
    d_0 := \inf_{\varphi \in [0,2\pi]} \; \inf_{\bm{q} \in [\mathcal{Q}_p]^2}
    \big\| \Grad \big( \hat r^{1/2} [\cos(\varepsilon\ln\hat r + \varphi)\,\bm\Phi_1 + \sin(\varepsilon\ln\hat r + \varphi)\,\bm\Phi_2] - \bm{q} \big) \big\|_{B_\rho(0) \cap \widehat\Omega} ,
\end{equation*}
which is positive as the minimum of a positive continuous function of $\varphi$ over the compact phase circle (no
phase renders the mode polynomial), the proofs of Proposition~\ref{prop:junctionlower} and
Corollary~\ref{cor:junctionlowerL2} carry over verbatim with $\lambda$ replaced by $\operatorname{Re}\lambda =
    \tfrac12$ in the rates.

Since $\lambda < 1$ at a mixed junction, the exponent satisfies $1 + \lambda < 2 \leq p + 1$: no polynomial degree
attains the optimal $L^2$ rate at an unresolved junction. At a clamped--clamped corner, where $\lambda \in (1,2)$, the
same bound $1+\lambda < 3 \leq p+1$ rules out the optimal rate for every $p \geq 2$.

Several comments are in order. First, both bounds hold for \emph{every} discrete function, not merely the CutFEM
solution: they are statements about the richness of the space $\mathbb{V}_h$, and are therefore independent of the
Nitsche parameter, the ghost penalty, and the quadrature. This is consistent with the exact junction quadrature
described above. Second, the rate actually observed is sharper than the proved $L^2$ lower bound: combined with
energy-norm quasi-optimality (which for the linearised (linear-elastic) problem follows from
Proposition~\ref{prop:coercivity} and consistency via C\'ea's lemma, with no regularity hypothesis; the hypotheses of
Theorem~\ref{thm:brr} are unavailable here, since $\bm u \notin W^{1,\infty}$ at the junction), the energy-norm rate at
an unresolved junction is exactly $\min(p, \lambda)$, and since the dual problem carries the same corner singularity,
the usual Aubin--Nitsche argument gains at most an additional factor $h^{\min(1,\lambda)}$, predicting the $L^2$ rate
cap
\begin{equation}
    \label{eq:l2cap}
    \| \bm{u} - \bm{u}_h \|_{L^2(\Omega)} \lesssim h^{\min\left(p+1, \; \min(p,\lambda) + \min(1,\lambda)\right)}
\end{equation}
\cite{grisvard1985elliptic}, where for a complex leading exponent $\lambda$ is to be read as $\operatorname{Re}\lambda$
here and throughout. At a mixed junction, where $\lambda < 1$, the second argument reduces to $2\lambda$,
independently of the polynomial degree; once $\lambda > 1$ the duality gain saturates at one full order and the cap
becomes $\min(p+1, 1+\lambda)$. We emphasise the asymmetry in status: \eqref{eq:l2cap}
is an upper bound obtained by a duality argument, and unlike Proposition~\ref{prop:junctionlower} and
Corollary~\ref{cor:junctionlowerL2} it does not by itself preclude a better rate; it is, however, consistent with the
corollary (since $\min(p,\lambda) + \min(1,\lambda) \leq 1 + \lambda$, with equality for $1 \leq \lambda \leq p$), and the
experiments below, in particular the
$\nu$-sweep of Figure~\ref{fig:nusweep}, show that it is attained sharply. For a junction located on a straight boundary
segment, $2\operatorname{Re}\lambda = 1$; at the right-angle corners of the pole test case the cap is $2\lambda(\pi/2, \nu)$ in the mixed
setting, in agreement with Figure~\ref{fig:poleconvergence}. Third, we emphasise that aligning the junction with a vertex of the
background mesh does \emph{not} restore the rate: while the lower-bound argument of
Proposition~\ref{prop:junctionlower} is stated for a junction interior to a cell, the attainable rate is governed by
the global regularity $\bm{u} \in H^{1+\lambda-\epsilon}$ near $\bm{x}_0$, and the best approximation of $r^\lambda
    \bm{\Phi}$ on a quasi-uniform mesh is $O(h^\lambda)$ in the energy norm regardless of where the cell boundaries
lie.\footnote{This regularity-limited rate is classical and method-independent: the same cap is attained by a
    body-fitted discretisation on a quasi-uniform mesh, since it is a property of the best approximation of
    $r^\lambda\bm\Phi$ by the local polynomial space, not of how the boundary is meshed \cite{grisvard1985elliptic,
        babuska1987optimal, schatz1979maximum}. We therefore do not include a separate body-fitted comparison; the
    quantitative match of the measured rates to the predicted exponent across the $\nu$-sweep
    (Figure~\ref{fig:nusweep}) is the sharper confirmation that the observed cap is the intrinsic one.} A
derivative kink admitted by the piecewise polynomial space at a vertex can capture an exact $r^1$-type singular term,
i.e. the case $\lambda = 1$, but no $r^\lambda$ with $\lambda < 1$. We tested this remedy directly: snapping the
background mesh so that a vertex coincides with the junction $\bm{x}_0$ produces errors and convergence rates
essentially identical to the generic, non-aligned configuration. As the two curves are visually indistinguishable we do
not report a separate plot.

Restoring optimal rates at an unfitted junction requires local mesh grading towards the corners, which we analyse in
Section~\ref{sec:remedy}. The mixed-boundary-condition results of Section~\ref{sec:pole} thus exhibit precisely the
capped rate \eqref{eq:l2cap}, floored by the lower bounds of Proposition~\ref{prop:junctionlower} and
Corollary~\ref{cor:junctionlowerL2}, while the all-Dirichlet disc test case of Section~\ref{sec:disc}, free of corners
and junctions, attains the rates of Theorem~\ref{thm:brr}.

\subsection{Dependence of the singular exponent on the material parameters}
\label{sec:nudependence}

The dependence of the cap \eqref{eq:l2cap} on the Poisson ratio enters through $\lambda(\omega,\nu)$ and differs
between straight and right-angle junctions.

If the junction lies on a straight part of the boundary ($\omega = \pi$), the answer is negative for every material:
the mixed Dirichlet--Neumann edge exponent of the Lam\'e operator is the classical punch-problem singularity
\cite{grisvard1985elliptic, rossle2000corner}
\begin{equation}
    \lambda = \tfrac{1}{2} \pm i \varepsilon, \qquad \varepsilon = \frac{1}{2\pi} \ln (3 - 4\nu),
\end{equation}
so $\nu$ affects only the oscillatory part $\varepsilon$, while $\operatorname{Re}\lambda=1/2$ and the resulting
first-order $L^2$ cap are material-independent.

At the right-angle corner of the pole geometry ($\omega = \pi/2$), the exponent does depend on $\nu$. For linear
elasticity the characteristic equation can be derived in closed form. Representing the eigenfunction through the
Kolosov--Muskhelishvili potentials $\varphi(z) = a z^\lambda$, $\psi(z) = b z^\lambda$, the displacement and traction
resultants on a ray $\theta = \mathrm{const}$ of the wedge $0 \leq \theta \leq \omega$ are
\begin{equation*}
    2\mu_e\, \bm{u} \sim r^\lambda \big[ \varkappa a e^{i\lambda\theta} - \lambda \bar{a} e^{i(2-\lambda)\theta} - \bar{b} e^{-i\lambda\theta} \big],
    \qquad
    \bm{T} \sim r^\lambda \big[ a e^{i\lambda\theta} + \lambda \bar{a} e^{i(2-\lambda)\theta} + \bar{b} e^{-i\lambda\theta} \big],
\end{equation*}
with the Kolosov constant $\varkappa = 3 - 4\nu$ (plane strain). Imposing the clamped condition $\bm{u} = \bm{0}$ at
$\theta = 0$ and the traction-free condition $\bm{T} = \bm{0}$ at $\theta = \omega$ yields a homogeneous $2 \times 2$
real system for $a \in \mathbb{C}$, whose determinant vanishes if and only if
\begin{equation}
    \label{eq:charequation}
    (\varkappa+1)^2 \cos^2(\lambda\omega) + (\varkappa-1)^2 \sin^2(\lambda\omega) = 4 \lambda^2 \sin^2\omega ,
    \qquad \text{equivalently} \quad
    \sin^2(\lambda\omega) = \frac{(\varkappa+1)^2 - 4\lambda^2 \sin^2\omega}{4\varkappa} .
\end{equation}
For $\omega=\pi$, \eqref{eq:charequation} recovers the punch singularity above
($e^{2\pi\varepsilon}=\varkappa$). For $\omega=\pi/2$, the smallest root is real and decreases monotonically from
$\lambda=1$ at $\nu=0$ to $0.5946\ldots$ in the incompressible limit. Hence the mixed-corner cap
$\min(p+1,2\lambda)$ can reach order two only at $\nu=0$: it may be optimal for $p=1$, but never for higher-order
elements.

The picture changes qualitatively when the two boundary segments meeting at the corner carry the \emph{same}
(Dirichlet) condition, as at the bottom corners of the pure-Dirichlet pole of Section~\ref{sec:pole}. Imposing $\bm{u}
    = \bm{0}$ on both faces $\theta = 0$ and $\theta = \omega$ of the wedge, the same Kolosov--Muskhelishvili calculation
eliminates $\psi$ through $\bar{b} = \varkappa a - \lambda \bar{a}$ and collapses the two clamped conditions onto the
single complex relation $\varkappa a \sin(\lambda\omega) = \lambda \bar{a}\, e^{i(1-\lambda)\omega} \sin\omega$; the
vanishing of the associated $2 \times 2$ real determinant gives the clamped--clamped characteristic equation
\begin{equation}
    \label{eq:charDD}
    \varkappa^2 \sin^2(\lambda\omega) = \lambda^2 \sin^2\omega .
\end{equation}
For $\omega=\pi$, \eqref{eq:charDD} has only integer roots, as expected for a regular point within a single Dirichlet
segment; the straight-boundary punch singularity is therefore exclusive to mixed conditions. For $\omega=\pi/2$,
$\varkappa\sin(\lambda\pi/2)=\lambda$, and the smallest root ranges from $\lambda\to1$ in the incompressible limit to
$\lambda\approx1.63$ at $\nu=0$. Thus a clamped--clamped right-angle corner has more than $H^2$ regularity, unlike the
mixed corner where $\lambda<1$.

Because $\lambda>1$, the $L^2$ cap is $\min(p+1,1+\lambda)$ with $1+\lambda\in(2,2.63)$: $p=1$ remains optimal, whereas
every $p\geq2$ is capped below order three. For example, $\nu\approx0.43$ gives $\lambda\approx1.20$ and the observed
cap $1+\lambda\approx2.2$. This exponent matches the unconditional lower bound of Corollary~\ref{cor:junctionlowerL2},
which applies here since $\lambda\in(1,2)$ is not an integer, so the clamped--clamped rate is sharp. Optimal
higher-order rates require the local grading of Section~\ref{sec:remedy}.

Figure~\ref{fig:nusweep} compares the measured $L^2$ rates with $2\lambda(\pi/2,\nu)$ for the mixed corner and
$1+\lambda(\pi/2,\nu)$ for the clamped--clamped corner.
The prediction is authoritative for linear elasticity; agreement of the neo-Hookean data only indicates that the linear
exponent governs the loads considered, while the nonlinear near-tip field remains open (Section~\ref{sec:remedy}). To
reduce coarse-level oscillations, particularly at small $\nu$, the measured rates are least-squares slopes of
$\log(\text{error})$ against $\log h$ over the last three levels.
The figure also includes a disc clamped on its lower half and traction-free on its upper half, with straight mixed
junctions at $(\pm R,0)$ integrated exactly. The predicted cap is $2\operatorname{Re}\lambda=1$, independent of $\nu$
and $p$; across $\nu\in[0.1,0.45]$, the measured rates are $0.93$--$0.97$ for $p=3$ and $0.90$--$0.96$ for $p=2$ in
both materials.
\begin{figure}[h]
    \centering
    \begin{tikzpicture}
        \begin{axis}[
                width=0.65\textwidth, height=6cm,
                xlabel={Poisson ratio $\nu$},
                ylabel={$L^2$ convergence rate cap},
                xmin=0, xmax=0.5, ymin=0.75, ymax=4.2,
                legend pos=north east,
                grid=major,
            ]
            \addplot[blue, thick] coordinates {
                    (0.00, 2.0000)
                    (0.05, 1.8452)
                    (0.10, 1.7341)
                    (0.15, 1.6424)
                    (0.20, 1.5621)
                    (0.25, 1.4895)
                    (0.30, 1.4223)
                    (0.35, 1.3594)
                    (0.40, 1.2998)
                    (0.45, 1.2432)
                    (0.49, 1.1998)
                };
            \addlegendentry{Dirichlet--Neumann (clamped--free)}
            \addplot[red, thick, dashed] coordinates {
                    (0.00, 2.6335)
                    (0.05, 2.6100)
                    (0.10, 2.5832)
                    (0.15, 2.5523)
                    (0.20, 2.5160)
                    (0.25, 2.4730)
                    (0.30, 2.4209)
                    (0.35, 2.3561)
                    (0.40, 2.2731)
                    (0.45, 2.1616)
                    (0.49, 2.0382)
                };
            \addlegendentry{Dirichlet--Dirichlet (clamped--clamped)}
            \addplot[gray, ultra thick] coordinates {(0.00, 1.0) (0.5, 1.0)};
            \addlegendentry{straight junction ($\omega = \pi$), Dirichlet--Neumann}
            \addplot[blue, only marks, mark=*] coordinates {
                    (0.10, 1.548) (0.25, 1.595) (0.35, 1.530) (0.45, 1.422)
                };
            \addplot[blue, only marks, mark=o] coordinates {
                    (0.10, 1.634) (0.25, 1.516) (0.35, 1.478) (0.45, 1.379)
                };
            \addplot[red, only marks, mark=square*] coordinates {
                    (0.10, 2.639) (0.25, 2.440) (0.35, 2.278) (0.45, 2.077)
                };
            \addplot[red, only marks, mark=square] coordinates {
                    (0.10, 2.684) (0.25, 2.549) (0.35, 2.401) (0.45, 2.115)
                };
            \addplot[black, only marks, mark=triangle*] coordinates {
                    (0.10, 0.929) (0.25, 0.942) (0.35, 0.956) (0.45, 0.970)
                };
            \addplot[black, only marks, mark=triangle] coordinates {
                    (0.10, 0.926) (0.25, 0.926) (0.35, 0.932) (0.45, 0.953)
                };
        \end{axis}
    \end{tikzpicture}
    \caption{Predicted and measured plane-strain $L^2$ rates. Curves show the right-angle caps
        $2\lambda(\pi/2,\nu)$ for mixed conditions and $1+\lambda(\pi/2,\nu)$ for pure Dirichlet conditions;
        the thick gray line is the straight-junction cap $2\operatorname{Re}\lambda=1$. Pole rates use $p=3$ and
        least-squares fits over the last three levels: circles~=~mixed, squares~=~pure Dirichlet; filled
        markers~=~neo-Hookean, open markers~=~linear elasticity. Black triangles are the straight mixed junction.
        The right-angle data broadly follow their respective caps, while the straight-junction rates are nearly independent
        of $\nu$ and approach one from below.}
    \label{fig:nusweep}
\end{figure}
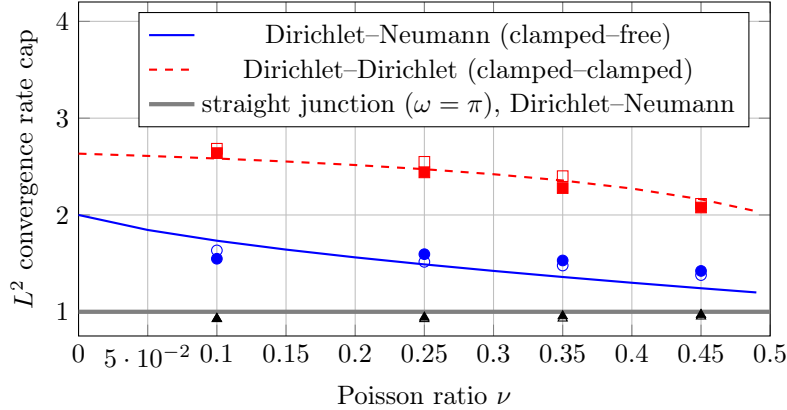

\section{Restoring optimal rates by local mesh grading/adaptivity}
\label{sec:remedy}

Proposition~\ref{prop:junctionlower} and Corollary~\ref{cor:junctionlowerL2} show that the rate cap \eqref{eq:l2cap} is
an \emph{approximation} obstruction: on a quasi-uniform mesh the space $\mathbb V_h$ cannot represent
$r^\lambda\bm\Phi$ with $\lambda\notin\mathbb N$ to better than $O(h^\lambda)$ in energy, regardless of where the cell
boundaries fall (Section~\ref{sec:mixedjunction}). The classical remedy is to refine the mesh locally near the
junction, a priori, by grading, or a posteriori, by adaptivity, so that the elements shrink fast enough to compensate
for the lost regularity of $r^\lambda\bm\Phi$. For body-fitted discretisations this is classical theory: the singular
exponents, the correct grading parameter, and the resulting optimal rates are all known \cite{babuska1979direct,
    apel1999anisotropic, grisvard1985elliptic}, and adaptive refinement attains the same rate per degree of freedom without
prior knowledge of the exponent.

Our purpose is not to re-derive that theory but to show that \emph{it transfers to the unfitted setting}: the rate
a matching mesh achieves with a given local refinement, the CutFEM discretisation on the \emph{same background
    refinement} inherits, with constants independent of how $\partial\Omega$ cuts the mesh. This is the same
cut-independence already established for stability and conditioning in Section~\ref{sec:theory}
(Proposition~\ref{prop:coercivity}, Corollary~\ref{cor:conditioning}), now at the level of approximation. The
corner-specific analysis, the exponents, the right amount of refinement, and, at finite strain, the genuinely nonlinear
near-tip field are thereby isolated as matching-mesh questions: classical for linear elasticity, partly open at finite
strain, but in either case an \emph{input} to the result below rather than an obstruction introduced by the cut.

\subsection{Local-size discretisation}
\label{sec:gradinglinear}

We instantiate the discretisation of Section~\ref{sec:theory} with the local mesh-size field in place of the global
parameter: the Nitsche penalty is taken as $\gamma/h_T$, the boundary term as $\sum_T\|h_T^{-1/2}\bm
    v\|^2_{\partial\Omega\cap T}$, and the ghost penalty with the local face size $h_F$, and we write $|||\cdot|||_g$ for
the resulting mesh-dependent norm. On a quasi-uniform mesh $h_T\sim h_F\sim h$, so $|||\cdot|||_g$ reduces to the
global triple norm \eqref{eq:triplenorm} and Section~\ref{sec:theory} is the special case of uniform refinement.

A concrete refinement realising a prescribed local size is radial grading towards the junction $\bm x_0$ with parameter
$\mu\in(0,1]$: with $h_T=\operatorname{diam} T$ and $r_T=\operatorname{dist}(T,\bm x_0)$,
\begin{equation}
    \label{eq:grading}
    h_T \ \simeq\
    \begin{cases}
        h^{1/\mu},         & r_T \le h^{1/\mu} \quad(\text{the core patch around }\bm x_0), \\[2pt]
        h\, r_T^{\,1-\mu}, & r_T > h^{1/\mu},
    \end{cases}
\end{equation}
the standard radial construction \cite{babuska1979direct, apel1999anisotropic}, with uniform shape-regularity and a
bounded neighbour-size ratio; here $\mu=1$ is quasi-uniform and smaller $\mu$ grades more strongly. On a graded mesh
$h$ is no longer the diameter of every element, as it was in the quasi-uniform setting of
Sections~\ref{sec:theory}--\ref{sec:mixedjunction}; in \eqref{eq:grading} and throughout this section it is a global
refinement parameter, namely the size away from the junction ($h_T\simeq h$ for $r_T\simeq 1$), from which the local
sizes $h_T$ are prescribed. All rates below are rates in this parameter. A direct count
gives $\#\mathcal T_h=O(h^{-2})$ elements for every $\mu\in(0,1]$, so the rate in $h$ coincides with the rate per
degree of freedom. The reduction below is, however, agnostic to how the local sizes are produced, by
\eqref{eq:grading} or by adaptive refinement.

The stability analysis of Section~\ref{sec:theory} was carried out for a global $h$ (the coercivity estimate and the
discrete inverse/trace lemma assume quasi-uniformity), so its transfer to $|||\cdot|||_g$ is not automatic; we isolate
it, together with the matching-mesh approximability that we import, as two assumptions. The first of these is in fact
provable under mild geometric hypotheses on the refinement; we state it as an assumption here to keep the exposition
self-contained and defer the proof, a local-size generalisation of Proposition~\ref{prop:coercivity} and
Theorem~\ref{thm:brr}, to \ref{app:gradedstab}.

\begin{assumption}[Local-size CutFEM stability]
    \label{ass:gradedstab}
    On the refined family $\{\mathcal T_h\}$, the local-size cut trace inequality, the ghost-penalty extension
    property, the Nitsche coercivity of $A_h$ in $|||\cdot|||_g$, and the local interpolation estimates hold with the
    local sizes $h_T,h_F$ and constants independent of how $\partial\Omega$ cuts the mesh and of the refinement level.
    These properties imply the linear quasi-optimality estimate in $|||\cdot|||_g$ and provide the uniform-coercivity
    ingredient for the nonlinear result. The latter additionally requires the local-size differentiability and
    Jacobian-defect hypotheses stated in Theorem~\ref{thm:gradedqo}.
\end{assumption}

This holds under shape-regularity, a bounded neighbour-size ratio, and a bounded active-layer width, all met by the
radial grading \eqref{eq:grading} and by graded-conforming adaptivity; the corresponding geometric stability statements
are proved in \ref{app:gradedstab} (Proposition~\ref{prop:gradedcoercivity}). They provide the linear result directly
and the coercivity input to the conditional nonlinear Theorem~\ref{thm:gradedqo}. This is plausible because the
ingredients of Section~\ref{sec:theory} are elementwise and local: within each refinement region the mesh is
quasi-uniform and the bounded neighbour-ratio controls the transitions. The one ingredient deserving comment is the
ghost penalty, which couples cut cells to neighbours reaching into the fictitious domain. Here the unfitted geometry is
in fact benign: the singular mode $r^\lambda\bm\Phi$ continues analytically in $\theta$ into the fictitious sector (its
angular profile is built from $e^{i\lambda\theta}$-type terms) and is smooth away from $\bm x_0$, with the same
$r^{\lambda-k}$ derivative growth. The extension of the solution across $\partial\Omega$ needed below is therefore
intrinsic rather than constructed. Because it is tied to the geometric corner, it is independent of how the level set
cuts the cells. This is precisely what lets the matching-mesh approximability of Assumption~\ref{ass:matchapprox}
survive the passage to the cut space.

\begin{assumption}[Matching-mesh approximability]
    \label{ass:matchapprox}
    There is a stable extension $E\bm u\in \big(H^1(\Omega_T)\big)^d$ of the solution, with $E\bm u=\bm u$ on $\Omega$,
    and a rate $R(h)$ such that the standard conforming interpolant $I_h$ on the background family $\{\mathcal T_h\}$
    satisfies, with the interpolation error $\bm e_h := E\bm u-I_h E\bm u$,
    \begin{equation}
        \label{eq:matchrate}
        |||\,\bm e_h\,|||_g \le C\,R(h),
        \qquad
        \|\bm e_h\|_{L^2(\Omega_T)} \le C\,h\,R(h),
    \end{equation}
    with $C$ independent of $h$. The first bound is posed in the full local-size norm \eqref{eq:gradednorm}: bulk
    gradient, Nitsche boundary term $\sum_T\|h_T^{-1/2}\bm e_h\|^2_{\partial\Omega\cap T}$ \emph{and} ghost-penalty
    seminorm; since the interior-face jumps of $E\bm u$ vanish, the ghost term is that of $I_h E\bm u$ alone. In
    addition, the same approximability is assumed for the dual problem: for every
    $\bm g \in (L^2(\Omega))^d$, the solution $\bm\varphi_{\bm g}$ of the linearised adjoint problem with data $\bm g$
    admits a stable extension with
    $|||\,E\bm\varphi_{\bm g}-I_h E\bm\varphi_{\bm g}\,|||_g \le C\,h\,\|\bm g\|_{L^2(\Omega)}$.
\end{assumption}

The bulk-gradient part of \eqref{eq:matchrate} is nothing but the body-fitted approximation result, applied to the
extended field on the active mesh; the extension exists with the same approximability by the analytic continuation
noted above. We stress that the Nitsche and ghost-penalty parts of the local-size norm do \emph{not} follow from the
two global bulk bounds alone: on a graded family $h_T \ll h$ precisely in the boundary cells the grading targets, so
the locally weighted sums must be controlled with their local weights. This is, however, exactly what the classical
graded interpolation theory provides, since it establishes the bulk rate through \emph{elementwise} weighted estimates
away from the core patch, $h_T^{-2}\|\bm e_h\|_T^2 + \|\Grad\bm e_h\|_T^2 \le C\,h_T^{2p}\,|E\bm u|^2_{H^{p+1}(T)}$,
with a weighted-space estimate on the core patch. These estimates sum to $R(h)^2$ \cite{babuska1979direct,
    apel1999anisotropic}. The Nitsche term then follows cell-by-cell from the cut interpolation trace estimate
$\|h_T^{-1/2}\bm e_h\|^2_{\partial\Omega\cap T} \le C\,\big(h_T^{-2}\|\bm e_h\|^2_T + \|\Grad\bm e_h\|^2_T\big)$, whose
constant is independent of the position of $\partial\Omega$ within the cell \cite{HANSBO20025537, Hansbo_2017}, and the
ghost-penalty term from the standard patchwise weak-consistency estimate for the interpolant; face-trace and inverse
inequalities reduce each face contribution to a two-cell best approximation error of the same elementwise form, as in
the cut interpolation theory of \cite{Hansbo_2017, Sticko2020, Burman_Hansbo_Larson_Zahedi_2025}. For the
Dirichlet--Neumann corner Assumption~\ref{ass:matchapprox} therefore holds with $R(h)=h^p$: radial grading with the
parameter governed by $\lambda_1$, the singular exponent of smallest positive real part in the Kondratiev expansion
\eqref{eq:cornerexpansion} at $\bm x_0$ (written $\lambda$ in Section~\ref{sec:mixedjunction}, where only the leading
mode was needed),
\begin{equation}
    \label{eq:gradingcond}
    \mu < \frac{\operatorname{Re}\lambda_1}{p},
\end{equation}
realises it \cite{babuska1979direct, apel1999anisotropic, grisvard1985elliptic}, and adaptive refinement realises the
same rate per degree of freedom without using the exponent. The dual clause is the same body-fitted statement at the
lowest rate: the adjoint solution carries the identical corner exponents, so the grading \eqref{eq:gradingcond} chosen
for the primal rate $R(h)=h^p$ in particular suffices for the dual rate $h$, which only requires $\mu <
    \operatorname{Re}\lambda_1$.

\subsection{CutFEM inherits the matching-mesh rate}
\label{sec:gradinginherit}

\begin{theorem}[Cut-independent inheritance]
    \label{thm:cutinherit}
    Under Assumption~\ref{ass:gradedstab} and the primal estimates of Assumption~\ref{ass:matchapprox}, and, for the
    nonlinear problem, under the additional local-size BRR hypotheses of
    Theorem~\ref{thm:gradedqo}\textnormal{(ii)}, for $h\le h_0$ the CutFEM solution on the background family
    $\{\mathcal T_h\}$ satisfies
    \begin{equation}
        \label{eq:cutinheritrate}
        |||\bm u-\bm u_h|||_g \le C\,R(h),
    \end{equation}
    with $C$ independent of how $\partial\Omega$ cuts the mesh. For linear elasticity, the dual clause of
    Assumption~\ref{ass:matchapprox} additionally gives
    \begin{equation}
        \label{eq:cutinheritl2}
        \|\bm u-\bm u_h\|_{L^2(\Omega)} \le C\,h\,R(h).
    \end{equation}
    In particular, for the linear Dirichlet--Neumann junction problem, radial grading gives $R(h)=h^p$, so
    $|||\bm u-\bm u_h|||_g\le C\,h^p$ and $\|\bm u-\bm u_h\|_{L^2(\Omega)}\le C\,h^{p+1}$: the cap
    \eqref{eq:l2cap} is removed.
\end{theorem}

\begin{proof}
    Assumption~\ref{ass:gradedstab} gives quasi-optimality in $|||\cdot|||_g$ for the linear problem by C\'ea's lemma;
    for a regular nonlinear branch the same estimate follows conditionally from
    Theorem~\ref{thm:gradedqo}\textnormal{(ii)}. It therefore suffices to exhibit one discrete competitor. Take $\bm
        v_h=(I_h E\bm u)|_{\Omega_T}\in\mathbb V_h$. Since $\bm u=E\bm u$ on $\Omega\subset\Omega_T$ and the
    interior-face jumps of $E\bm u$ vanish, each of the three contributions to $|||\bm u-\bm v_h|||_g$, namely the bulk
    gradient on $\Omega$, the local-size Nitsche boundary term $\sum_T\|h_T^{-1/2}(\bm u-\bm
        v_h)\|^2_{\partial\Omega\cap T}$, and the ghost-penalty seminorm, coincides with the corresponding term of
    $|||\,E\bm u-I_h E\bm u\,|||_g$, so Assumption~\ref{ass:matchapprox} gives directly
    \begin{equation*}
        \inf_{\bm v_h\in\mathbb V_h}|||\bm u-\bm v_h|||_g \le |||\,E\bm u-I_h E\bm u\,|||_g \le C\,R(h)
    \end{equation*}
    and the energy estimate in \eqref{eq:cutinheritrate}; the cut enters only through the cut-independent constants of
    Assumption~\ref{ass:gradedstab}, the approximability being posed on full cells of the background family. For linear
    elasticity, the Aubin--Nitsche argument and the dual clause of Assumption~\ref{ass:matchapprox} give an $O(h)$
    energy best approximation of the dual solution. Pairing it with the primal energy error $O(R(h))$ proves
    \eqref{eq:cutinheritl2}.
\end{proof}

For clarity, the nonlinear part of Theorem~\ref{thm:cutinherit} is the conditional energy estimate
\eqref{eq:cutinheritrate}; the $L^2$ estimate \eqref{eq:cutinheritl2} is a linear-elastic result.

The theorem is a \emph{reduction} result: the singular structure enters through $R(h)$ in
Assumption~\ref{ass:matchapprox}, whereas the cut enters through the cut-independent constants of
Assumption~\ref{ass:gradedstab} and the interpolation trace estimate. Three consequences follow.

\emph{First, refinement is required already for linear elasticity.} The corner singularity limits the quasi-uniform
matching-mesh rate to $R(h)=h^{\operatorname{Re}\lambda_1}$, with Proposition~\ref{prop:junctionlower} ruling out a
better energy rate; grading satisfying \eqref{eq:gradingcond}, or adaptivity without prior knowledge of $\lambda_1$,
restores $R(h)=h^p$.
Theorem~\ref{thm:cutinherit} transfers either rate to CutFEM without restricting $\lambda_1$: the cut does not alter the
refinement requirement of the body-fitted problem.

\emph{Second, the finite-strain problem inherits the rate conditionally} under the local-size BRR hypotheses of
Theorem~\ref{thm:gradedqo}. Finite strain introduces no additional \emph{cut} difficulty, but when the corner gradient
is singular the near-tip field need not follow the linear characteristic equation \eqref{eq:charequation}
\cite{knowles1973asymptotic, knowles1974finite}. The appropriate grading exponent, whether the matching mesh attains
$R(h)=h^p$, and the relevant, possibly weighted, norm are then properties of the continuous problem and its
body-fitted approximation, not of the cut. These questions remain open; exponent-free adaptivity sidesteps them.

\emph{Third, the obstruction is localised at the corner.} Interior estimates
\cite{schatz1979maximum, babuska1987optimal} give convergence of approximately
$h^{2\operatorname{Re}\lambda_1}$ away from the junction, better than the global energy rate
$h^{\operatorname{Re}\lambda_1}$ but still suboptimal for $p\ge2$. The global cap \eqref{eq:l2cap} remains an
approximation property: with the standard polynomial space fixed, quadrature, stabilisation, and penalty choices
cannot remove it. One must improve the approximation near $\bm x_0$ by local refinement
(Section~\ref{sec:remedy}) or singular enrichment (Section~\ref{sec:gradingvsenrich}). We recommend adaptivity, which
realises Assumption~\ref{ass:matchapprox} without prescribing an exponent. Assumption~\ref{ass:gradedstab} is
established in~\ref{app:gradedstab} under shape-regularity, a bounded neighbour-size ratio, and a bounded
active-layer width.

Figure~\ref{fig:graded_error} confirms the reduction numerically. Applying the radial grading \eqref{eq:grading}
towards the two Dirichlet--Neumann junctions of the pole, with the parameter chosen from \eqref{eq:gradingcond}, we
measure the CutFEM self-convergence $\|\bm u_h^{\mathrm c}-\bm u_{h_{\mathrm{ref}}}^{\mathrm c}\|_{L^2}$ for both
linear elasticity and the neo-Hookean model, at Poisson ratio $\nu=0.45$, where
$\bm u_{h_{\mathrm{ref}}}^{\mathrm c}$ is the CutFEM solution on a finer graded reference mesh. On the graded family $N=O(h^{-2})$, so the linear-elastic $L^2$ rate
$h^{p+1}$ proved in Theorem~\ref{thm:cutinherit} corresponds to the slopes $N^{-1}$ for $p=1$ and $N^{-3/2}$ for $p=2$;
we use the same slopes only as empirical benchmarks for the neo-Hookean results. For $p=1$ both models track the
$N^{-1}$ slope, recovering the order that was capped on the quasi-uniform mesh (Figure~\ref{fig:poleconvergence}). For
$p=2$ the evidence is partial: the linear-elastic curve approaches the optimal $N^{-3/2}$ slope (the per-level rate
clusters near $3$, with mild preasymptotic scatter), but the neo-Hookean $p=2$ curve has not reached it over the levels
computed: its effective slope drifts downward in Figure~\ref{fig:graded_error}, from $\approx2.8$ on the coarse levels toward $\approx2$
on the finest, so the asymptotic optimal slope is not yet attained. We read this as preasymptotic rather than as a
failure of the inheritance result: the grading parameter \eqref{eq:gradingcond} is fixed by the \emph{linear} exponent
$\lambda_1$, whereas at $p=2$ the discretisation already resolves the genuinely nonlinear near-tip field, whose
exponent, and hence the amount of grading that is actually optimal, is the open finite-strain question flagged above.
Pushing the neo-Hookean $p=2$ test case into its asymptotic regime, and the systematic study on adaptively refined cut
meshes where the exponent is realised a posteriori, are left for future work.

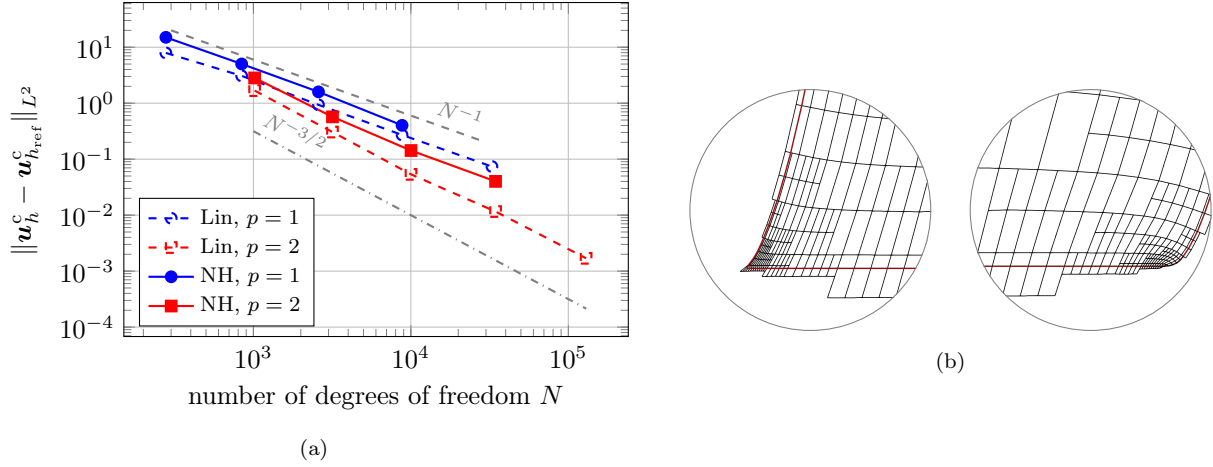
\begin{figure}[h]
    \centering
    \begin{minipage}{0.5\textwidth}
        \centering
        \begin{tikzpicture}
            \begin{loglogaxis}[
                    width=\textwidth, height=6cm,
                    xlabel={number of degrees of freedom $N$},
                    ylabel={$\|\bm u_h^{\mathrm c}-\bm u_{h_{\mathrm{ref}}}^{\mathrm c}\|_{L^2}$},
                    legend pos=south west,
                    legend cell align=left,
                    legend style={font=\footnotesize},
                    grid=major,
                ]
                \addplot[blue, thick, dashed, mark=o] coordinates {
                        (278, 7.9074e+00) (842, 3.0777e+00) (2590, 9.3419e-01)
                        (8794, 2.7095e-01) (32734, 7.3754e-02)
                    };
                \addlegendentry{Lin, $p=1$}
                \addplot[red, thick, dashed, mark=square] coordinates {
                        (1018, 1.6788e+00) (3178, 3.0539e-01) (10022, 5.3576e-02)
                        (34502, 1.1598e-02) (129690, 1.6952e-03)
                    };
                \addlegendentry{Lin, $p=2$}
                \addplot[blue, thick, mark=*] coordinates {
                        (278, 1.5009e+01) (842, 5.0226e+00) (2590, 1.5936e+00)
                        (8794, 4.0189e-01)
                    };
                \addlegendentry{NH, $p=1$}
                \addplot[red, thick, mark=square*] coordinates {
                        (1018, 2.8238e+00) (3178, 5.7034e-01) (10022, 1.4265e-01)
                        (34502, 4.0125e-02)
                    };
                \addlegendentry{NH, $p=2$}
                \addplot[gray, dashed, thick, domain=300:30000, samples=2] {6e3*x^(-1)}
                node[pos=0.9, above, sloped, font=\footnotesize] {$N^{-1}$};
                \addplot[gray, dashdotted, thick, domain=1000:130000, samples=2] {10e3*x^(-1.5)}
                node[pos=0.1, above, sloped, font=\footnotesize] {$N^{-3/2}$};
            \end{loglogaxis}
        \end{tikzpicture}
        \\[2pt] {\footnotesize (a)}
    \end{minipage}\hfill
    \begin{minipage}{0.48\textwidth}
        \centering
        \newcommand{\cornerR}{1.6cm}
        \newcommand{\cornerW}{4.16cm}
        \newcommand{\cornercircle}[2][0pt]{
            \raisebox{#1}{%
                \begin{tikzpicture}
                    \clip (0,0) circle (\cornerR);
                    \node[anchor=center] at (0,0)
                    {\includegraphics[width=\cornerW,trim={0 0 0 256},clip]{#2}};
                    \draw[gray, thick] (0,0) circle (\cornerR);
                \end{tikzpicture}}}
        \cornercircle{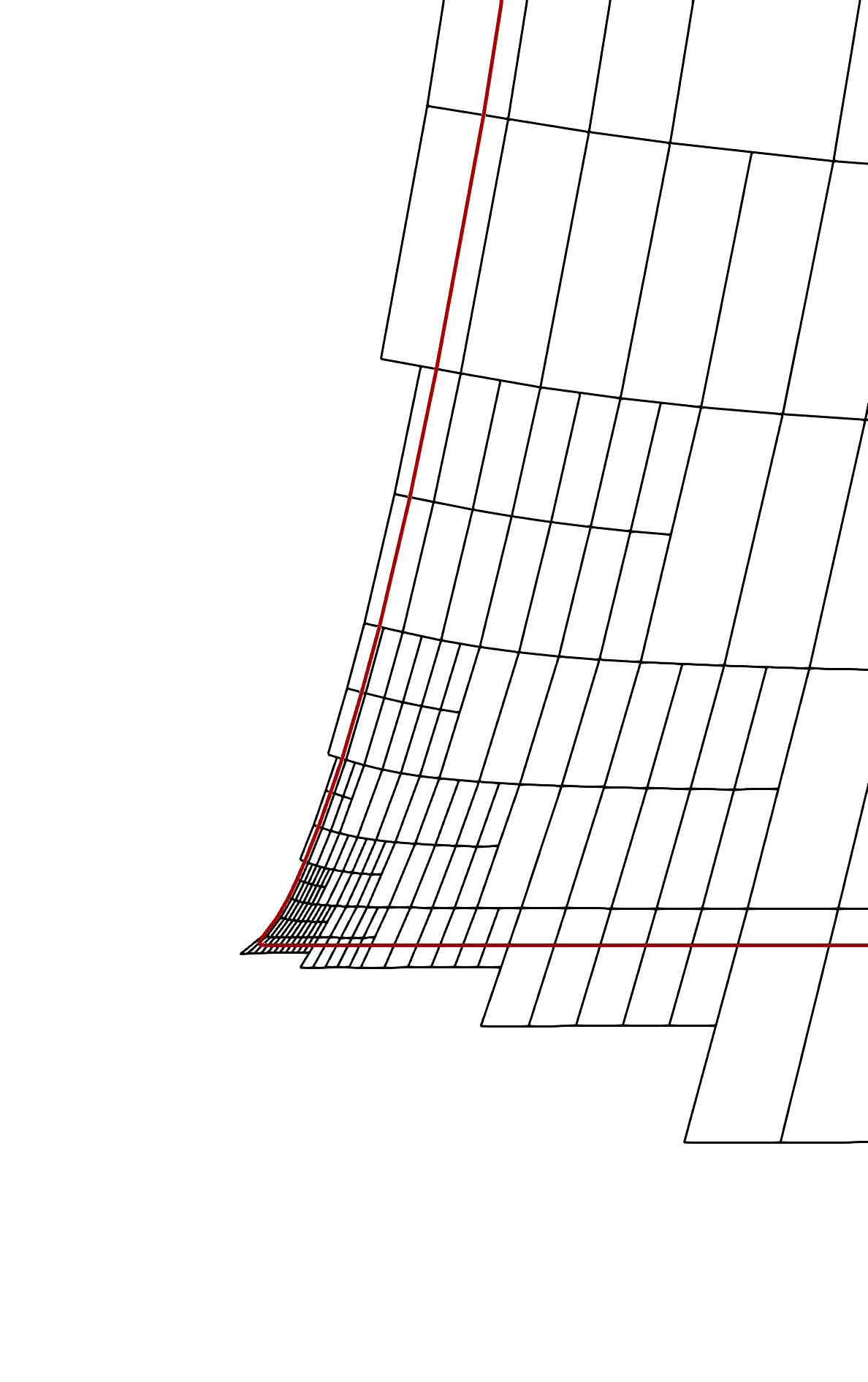}\hfil
        \cornercircle[0pt]{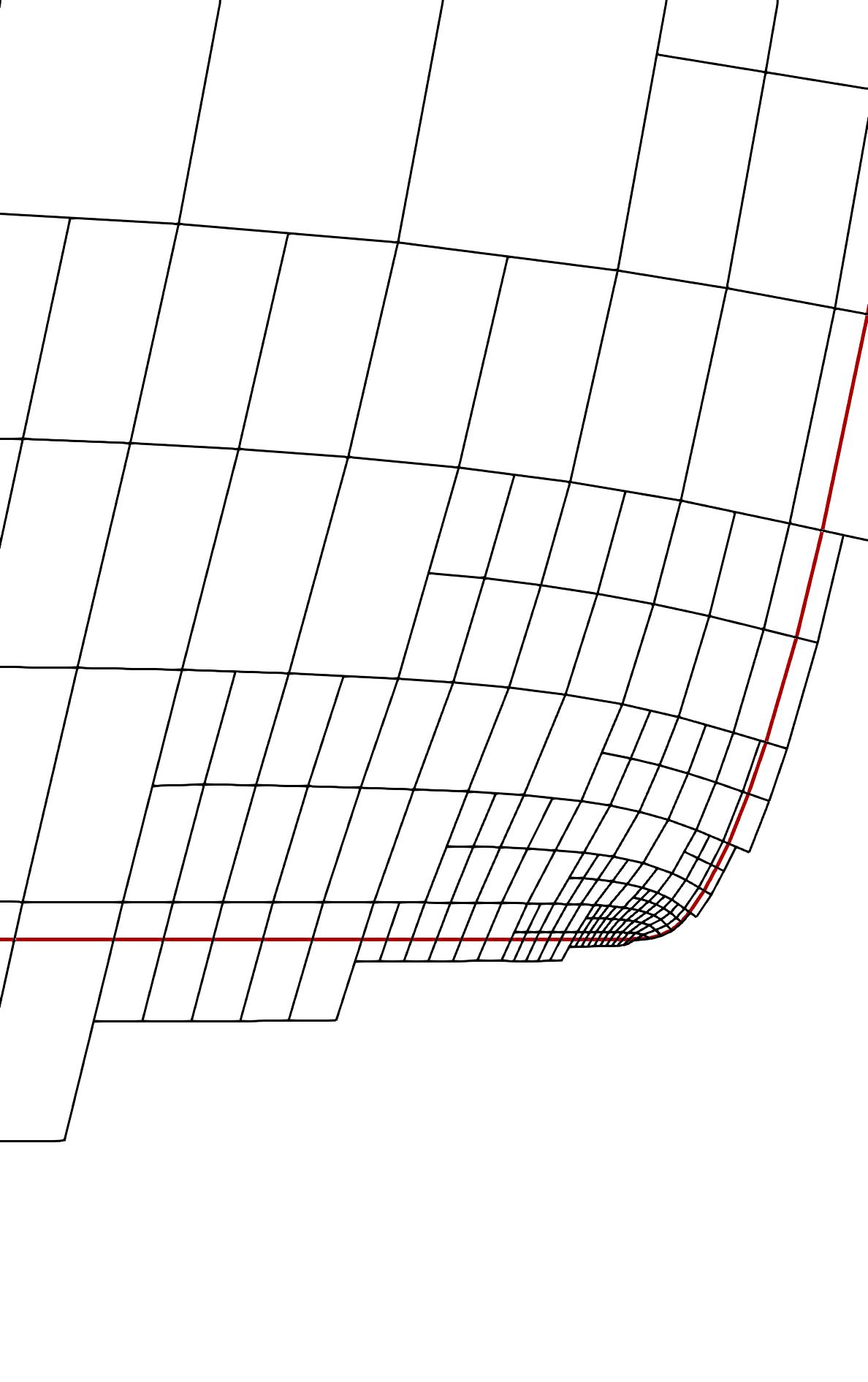}
        \\[2pt] {\footnotesize (b)}
    \end{minipage}
    \caption{Radial grading \eqref{eq:grading} of the pole at $\nu=0.45$.
        \emph{(a)} CutFEM self-convergence $\|\bm u_h^{\mathrm c}-\bm u_{h_{\mathrm{ref}}}^{\mathrm c}\|_{L^2}$
        against a finer graded CutFEM reference mesh for linear (Lin) and neo-Hookean (NH) models, with
        $p=1,2$; grey lines show the optimal slopes $N^{-1}$ and $N^{-3/2}$. \emph{(b)} Grading near the two
        Dirichlet--Neumann junctions; red marks the pole boundary.}
    \label{fig:graded_error}
\end{figure}

\subsection{Related remedies and the three-dimensional outlook}
\label{sec:gradingvsenrich}

Instead of grading, one may enrich $\mathbb V_h$ by the singular mode $r^\lambda\bm\Phi$ \cite{strang1973analysis,
    blum1982boundary, fix1973finite}. In two dimensions this adds one global degree of freedom per corner and retained mode
without changing the cut cells, but the enrichment becomes nearly linearly dependent on $\mathbb V_h$ as $h\to0$ and
represents only the leading \emph{linearised} mode for the neo-Hookean problem. Grading avoids these conditioning and
modelling issues and requires no change to the discrete space, stabilisation, or solver. It prescribes the exponent
through \eqref{eq:gradingcond} and therefore supports the analysis, whereas residual-driven adaptivity can recover the
rate without knowing $\lambda_1$. Since Theorem~\ref{thm:cutinherit} depends only on the resulting local sizes, it also
transfers an adaptive matching-mesh rate, provided the local-size stability of Assumption~\ref{ass:gradedstab} holds. A
robust a posteriori estimator for the nonlinear unfitted problem remains an open requirement.

The formulation, AD-generated constitutive code, and analysis in Sections~\ref{sec:formulation}--\ref{sec:theory} apply
in general dimension; two ingredients are specific to 2D. In three dimensions, the Kolosov--Muskhelishvili exponents
are replaced by edge and vertex eigenproblems \cite{kozlov2001spectral, mazya2010elliptic, dauge1988elliptic}, and the
junction becomes a curve crossing cut cells. Splitting the boundary quadrature then requires integration over implicit
surface patches bounded by that feature curve, beyond smooth-geometry rules such as \cite{Saye2015-xr}. The
regularity-induced cap and the need for local refinement are nevertheless expected to persist, with dimension-dependent
exponents and constants.

\section{Conclusion}
We presented a fully variational, model-independent CutFEM framework for finite-strain elasticity.
One augmented energy combines the bulk response, weak boundary conditions, and ghost penalty. At each iterate, its
first variation gives the residual with the adaptive Nitsche weight frozen, and a symmetric, coercive approximation of
its Hessian gives the quasi-Newton tangent. AceGen generates the constitutive stress and elasticity tensors from the
scalar energy density, allowing the neo-Hookean model and \modelB{} to use the same implementation.

For the linearised quasi-Newton problem, we proved cut-independent coercivity and continuity, an $O(h^{-2})$ condition
number bound, and descent of each damped correction under uniform ellipticity. The BRR result gives quasi-optimality
for regular solutions under the Jacobian-defect and differentiability hypotheses of Theorem~\ref{thm:brr}, and the
smooth test case attains the expected optimal rates. At Dirichlet--Neumann junctions, however, the exact corner
singularity caps both fitted and unfitted methods on quasi-uniform meshes; the Kolosov--Muskhelishvili equation
predicts the observed dependence on the corner type and material parameters. We stress that this cap is a property of
the exact solution, whose regularity is limited by the corner, and not of the discretisation: it is not an artefact of
the cut, of the Nitsche imposition, of the ghost penalty, or of the cut-cell quadrature, and a body-fitted method on
the same quasi-uniform meshes is subject to the identical bound. What is specific to the unfitted setting is only that
none of its geometric devices (exact splitting of the boundary integral at the junction, penalty tuning, or aligning
the junction with a mesh vertex) can circumvent it.

Local refinement removes this approximation barrier. Theorem~\ref{thm:cutinherit} shows that CutFEM inherits the rate
of the matching refined mesh with cut-independent constants; its geometric and linear-elastic ingredients are proved in
\ref{app:gradedstab}, while the finite-strain extension remains conditional on Theorem~\ref{thm:gradedqo}.
Numerically, optimal $L^2$ recovery is observed for $p=1$ in both material models; at $p=2$, the linear-elastic results
approach the optimal slope but the neo-Hookean case remains preasymptotic. Higher-order finite-strain recovery is thus
neither proved nor confirmed here. Adaptive cut meshes, nonlinear near-tip asymptotics, and junction-exact quadrature
in three dimensions remain subjects for future work, alongside extensions to richer materials and multiphysics.

\paragraph{Declarations}

This manuscript is the result of an extended iterative process supported by language models (Claude, Gemini, and
ChatGPT). The models assisted with drafting and
revision, as well as with exploring, checking, and refining mathematical arguments and proofs. All generated
suggestions were critically assessed by the authors, who independently verified the proofs, computations, and
scientific claims and take full responsibility for the final manuscript.

\bibliographystyle{elsarticle-num}
\bibliography{literature}

\newpage
\appendix

\section{Local-size CutFEM stability on graded meshes}
\label{app:gradedstab}

This appendix proves Assumption~\ref{ass:gradedstab}: the cut-independent stability of Section~\ref{sec:theory},
established there for a globally quasi-uniform background mesh, persists on the locally refined family $\{\mathcal
    T_h\}$ when the global size $h$ is replaced by the local sizes $h_T,h_F$ and the triple norm \eqref{eq:triplenorm} by
its mesh-dependent counterpart $|||\cdot|||_g$. The argument generalises the geometric and coercivity ingredients of
Proposition~\ref{prop:coercivity} to remove the quasi-uniformity assumption and states the corresponding nonlinear
result conditionally under local-size analogues of (J0)--(J2); it requires nothing of the refinement
beyond shape-regularity, a bounded neighbour-size ratio, and a bounded active-layer width, all satisfied by the radial
grading \eqref{eq:grading} and by standard adaptive (e.g. newest-vertex-bisection or one-irregular quadtree)
refinement.

\subsection{Geometric assumptions and the local-size norm}

Throughout, $\{\mathcal T_h\}$ is a family of background meshes, $T^h_\Omega$, $T^h_{\partial\Omega}$, $\Omega_T$,
$\mathbb V_h$ and $A_h(\cdot,\cdot)=\mathcal K_{\text{Total}}(\bar{\bm u};\cdot,\cdot)$ as in Section~\ref{sec:theory}.
We attach to every element $T$ its size $h_T=\operatorname{diam}T$ and to every interior face $F$ the size
$h_F=\operatorname{diam}F$, and we assume:

\begin{enumerate}
    \item[\textnormal{(G1)}] \emph{Shape-regularity.} Each $T\in\mathcal T_h$ contains a ball of radius
          $\rho_T$ with $h_T\le\sigma\,\rho_T$, with $\sigma$ independent of $T$ and of the refinement
          level.
    \item[\textnormal{(G2)}] \emph{Bounded neighbour-size ratio.} For every pair of elements $T,T'$
          sharing a face $F$, $\,c_N^{-1}\le h_T/h_{T'}\le c_N$, with $c_N$ independent of the level;
          consequently $h_F\simeq h_T\simeq h_{T'}$ on $F$.
    \item[\textnormal{(G3)}] \emph{Bounded active-layer width.} There is an integer $L$, independent of the
          refinement level, such that every cut cell $T\in T^h_{\partial\Omega}$ is joined to a cell
          $T_{\text{int}}\subset\Omega$ interior to the physical domain by a chain of at most $L$ successive
          face-neighbours, all lying in $T^h_\Omega$.
\end{enumerate}

All three hold for \eqref{eq:grading} with $\sigma,c_N,L$ depending only on the grading construction (and, for (G3),
on the opening angle $\omega$ and the Lipschitz character of $\partial\Omega$), and for graded-conforming adaptive
refinement with a fixed level-jump bound. The first two are standard; (G3) is the usual bounded-active-layer assumption
of ghost-penalty analysis and is what makes the chain argument of Lemma~\ref{lem:locext} level-independent, so we
record why the radial grading \eqref{eq:grading} satisfies it. A family is locally quasi-uniform as a consequence of
(G1)--(G2) (within any ball whose radius is comparable to the local cell size the cells have comparable size), so the
number of cut-cell layers separating $\Omega$ from the fictitious region is controlled at every scale by the Lipschitz
character of $\partial\Omega$ alone. For \eqref{eq:grading} this is explicit. Away from the junction, $h_T\simeq
    h\,r_T^{1-\mu}\ll r_T$, so at the scale of a cut cell the boundary is locally flat and $T$ is separated from an
interior cell by $O(1)$ layers. Inside the core patch $r_T\le h^{1/\mu}$ the cells share the common size $h^{1/\mu}$,
and a bounded number of them covers the corner neighbourhood; the junction (the only point at which the two boundary
segments meet) is therefore crossed in $O(1)$ layers as well. Hence (G3) holds with $L$ fixed by $\omega$, the
Lipschitz constant of $\partial\Omega$ and $\sigma$, independently of $h$.

Assumption (G2) is not a disguised quasi-uniformity, and in particular it does not rule out the strong
refinement one uses for a corner singularity. It is a purely \emph{local} condition: the global sizes
$h_T$ may vary by orders of magnitude across the mesh, which is precisely the point of grading towards
$\bm x_0$, and (G2) constrains only the \emph{jump} between two elements sharing a face. This is exactly
the standard $2{:}1$ balance (one-irregularity) condition: at most a bounded number of refinement levels
may change across any single face. The radial grading \eqref{eq:grading} is smooth in $r_T$ and so
satisfies it with $c_N$ fixed by the construction; newest-vertex-bisection and quadtree/octree
refinement with the usual mesh balancing satisfy it with $c_N=2$.

\begin{remark}[Why \textnormal{(G2)} does not exclude adaptivity]
    \label{rem:notquasiuniform}
    What (G2) does exclude is only
    \emph{unbalanced} refinement, e.g.\ an element of size $h$ placed directly against one of size $h/2^k$ with
    $k\to\infty$, which produces neither a usable interpolation nor a stable ghost-penalty coupling and is
    not generated by standard adaptive loops. This assumption is therefore mild in the operative sense: it
    admits every graded and balanced-adaptive mesh one would target a singularity with, and forbids only the
    pathological unbalanced ones.
\end{remark}

The local-size energy norm is
\begin{equation}
    \label{eq:gradednorm}
    |||\bm v|||_g^2 \;:=\;
    \|\Grad\bm v\|_\Omega^2
    \;+\; \sum_{T\in T^h_{\partial\Omega}} \|h_T^{-1/2}\bm v\|_{\partial\Omega\cap T}^2
    \;+\; |\bm v|_{s,g}^2 ,
    \qquad
    |\bm v|_{s,g}^2 := \sum_{F\in\mathcal F_h}\sum_{k=1}^{p} \gamma_A\,h_F^{2k-1}\,
    \big\|[\![\partial_{n_F}^k \bm v]\!]\big\|_{F}^2 ,
\end{equation}
i.e.\ \eqref{eq:triplenorm} with the global $h$ in the Nitsche and ghost-penalty scalings replaced by the
local $h_T,h_F$. On a quasi-uniform mesh $h_T\simeq h_F\simeq h$ and \eqref{eq:gradednorm} reduces to
\eqref{eq:triplenorm}, so what follows specialises to Section~\ref{sec:theory}. The penalty weight $\eta$ of
\eqref{eq:penaltyweight} and the ellipticity Assumption~\ref{ass:ellipticity}, being pointwise statements on
$\Omega_T$, are unaffected by the refinement and are kept unchanged.

\subsection{Localised ingredients}

The two unfitted ingredients of Section~\ref{sec:theory} are elementwise and scale-covariant; we record their
local-size forms. The argument is that each is an estimate on a \emph{single} full cell (or a face-neighbour pair),
whose constant is produced by a scaling argument to the reference element and therefore depends only on $p$, on
(G1)--(G2) and, for the cut trace inequality, on the local boundary-geometry constants introduced in
Lemma~\ref{lem:loctrace} below, and not on the global size or on how $\partial\Omega$ meets the cell.

\begin{lemma}[Local-size cut trace inequality]
    \label{lem:loctrace}
    Assume \textnormal{(G1)}, and additionally that within each cut cell the boundary portion
    $\partial\Omega\cap T$ is the union of at most $N_\Gamma$ uniformly Lipschitz graph patches of total surface
    measure $|\partial\Omega\cap T|\le C_\Gamma\,h_T^{d-1}$, with $N_\Gamma$, $C_\Gamma$ and the graph constant
    independent of $T$ and of the refinement level. For the geometries of this paper, this means either a single smooth
    boundary piece, locally flat at the scale of the cell, or the two straight segments meeting at the junction
    $\bm x_0$ inside a core-patch cell. Then for every $T\in T^h_{\partial\Omega}$ and every
    $\bm w_h\in(\mathbb P_p(T))^d$,
    \begin{equation}
        \label{eq:loccuttrace}
        h_T\,\|\bm w_h\|_{\partial\Omega\cap T}^2 \;\le\; C_T\,\|\bm w_h\|_T^2 ,
    \end{equation}
    with $C_T=C_T(p,\sigma,N_\Gamma,C_\Gamma)$ independent of $h_T$ and of the position of $\partial\Omega$ in $T$.
\end{lemma}

\begin{proof}
    This is \eqref{eq:cuttrace} with $h$ replaced by $h_T$. Map $T$ to a reference element $\hat T$ of unit
    size by the affine $F_T$; by (G1) the Jacobian satisfies $\|DF_T\|\simeq h_T$ and
    $\|DF_T^{-1}\|\simeq h_T^{-1}$ with constants depending only on $\sigma$. The cut-independent reference
    estimate $\|\hat{\bm w}\|_{\hat\Gamma}^2\le \hat C\,\|\hat{\bm w}\|_{\hat T}^2$, valid for every planar cut
    $\hat\Gamma$ of $\hat T$ by the polynomial trace argument of \cite{HANSBO20025537, massing2014stabilized}
    (its constant $\hat C=\hat C(p)$ is the maximum over the compact family of admissible cut positions),
    extends to the boundary geometry assumed above: each Lipschitz graph patch of bounded measure is covered by a
    bounded number of planar comparisons (for a boundary that is flat at the scale of the cell, a single one), and
    the estimate is applied to each patch separately and the bounds added, at the cost of a factor depending only on
    $N_\Gamma$ and $C_\Gamma$. In particular, in the corner-containing core cell, where the two straight segments
    meet inside $\hat T$, the planar estimate is applied to each segment and the two bounds added
    ($N_\Gamma = 2$). Pulling back under $F_T$, the surface and volume measures contribute the single power $h_T$,
    which yields \eqref{eq:loccuttrace} with $C_T = C_T(p,\sigma,N_\Gamma,C_\Gamma)$.
\end{proof}

\begin{lemma}[Local-size ghost-penalty extension]
    \label{lem:locext}
    Under \textnormal{(G1)}--\textnormal{(G3)}, there is $C_G=C_G(p,\sigma,c_N,L)$, independent of the level
    and of the cut, such that
    \begin{equation}
        \label{eq:locextension}
        \|\Grad\bm v_h\|_{\Omega_T}^2 \;\le\; C_G\Big( \|\Grad\bm v_h\|_\Omega^2 + \gamma_A^{-1}\,|\bm v_h|_{s,g}^2\Big)
        \qquad \forall\,\bm v_h\in\mathbb V_h .
    \end{equation}
\end{lemma}

\begin{proof}
    The extension property \eqref{eq:extension} is proved by a discrete patchwise argument
    \cite{BURMAN20101217, Burman_Claus_Hansbo_Larson_Massing_2014, Burman_Hansbo_Larson_Zahedi_2025}: each cut
    cell $T$ is connected through a chain of at most $L$ face-neighbours to an element $T_{\text{int}}$
    that is interior to $\Omega$ (such a chain exists by (G3)), and the gradient on the fictitious part
    $T\setminus\Omega$ is estimated from $T_{\text{int}}$ by repeatedly applying, across each shared face $F$
    of the chain, the one-face control
    \[
        \big\|\Grad\bm v_h\big\|_{T_+}^2 \;\le\; C\Big(\big\|\Grad\bm v_h\big\|_{T_-}^2 + \sum_{k=1}^p
        h_F^{2k-1}\big\|[\![\partial_{n_F}^k\bm v_h]\!]\big\|_F^2\Big),
    \]
    which is the elementwise content of the ghost penalty. Each step is a scaling estimate on the \emph{full} pair
    $(T_-,T_+)$, hence cut-independent, and by (G2) the two cells across $F$ have comparable size, so $h_F$ is
    interchangeable with either $h_{T_\pm}$ and the per-face constant depends only on $p,\sigma,c_N$. Iterating the
    one-face control along the chain composes this constant at most $L$ times; since by (G2) the sizes along the chain vary
    by a factor at most $c_N^L$, the accumulated factor and the local-size seminorm $|\cdot|_{s,g}$ of
    \eqref{eq:gradednorm} that it produces on the right are bounded by a quantity depending only on $p,\sigma,c_N,L$.
    Summing over the active layer, with the overlap of chains bounded by (G1)--(G3), yields \eqref{eq:locextension} with
    $C_G=C_G(p,\sigma,c_N,L)$.
\end{proof}

We also use the local interpolation estimate: for the conforming interpolant $I_h$ on $\{\mathcal T_h\}$ and $\bm w\in
    (H^{p+1}(T))^d$,
\begin{equation}
    \label{eq:locinterp}
    \|\bm w-I_h\bm w\|_T + h_T\,\|\Grad(\bm w-I_h\bm w)\|_T \;\le\; C_I\,h_T^{\,p+1}\,|\bm w|_{H^{p+1}(T)},
\end{equation}
the standard Bramble--Hilbert estimate on shape-regular elements \cite{Braess2013-mg}, with $C_I=C_I(p,\sigma)$.

\subsection{Coercivity, continuity and quasi-optimality in \texorpdfstring{$|||\cdot|||_g$}{the local-size norm}}

\begin{proposition}[Cut- and level-independent coercivity and continuity]
    \label{prop:gradedcoercivity}
    Under Assumption~\ref{ass:ellipticity}, \textnormal{(G1)}--\textnormal{(G3)} and the boundary-geometry
    hypothesis of Lemma~\ref{lem:loctrace}, with the Nitsche penalty
    taken as $\gamma_D\eta/h_T$ on each cut cell, there exist $c_s,C_s>0$ depending only on
    $c_L,C_L,\eta_{\min},C_K,h_0,\gamma_A/E_e,p,\sigma,c_N,L$ and the constants $N_\Gamma,C_\Gamma$ of
    Lemma~\ref{lem:loctrace}, in particular independent of the refinement level and of how
    $\partial\Omega$ cuts the mesh, such that
    \begin{equation}
        A_h(\bm v_h,\bm v_h)\ge c_s\,|||\bm v_h|||_g^2,
        \qquad
        A_h(\bm u_h,\bm v_h)\le C_s\,|||\bm u_h|||_g\,|||\bm v_h|||_g
        \qquad \forall\,\bm u_h,\bm v_h\in\mathbb V_h .
    \end{equation}
\end{proposition}

\begin{proof}
    We repeat the proof of Proposition~\ref{prop:coercivity}, performing every estimate \emph{cell-by-cell}
    with the local size and summing, so that no global $h$ enters. As there,
    \[
        A_h(\bm v_h,\bm v_h)=\int_\Omega \Grad\bm v_h:\mathbb L(\bar{\bm u}):\Grad\bm v_h\,dV
        \;-\;2\!\int_{\partial\Omega}(\mathbb L\Grad\bm v_h\,\bm n)\!\cdot\!\bm v_h\,dS
        \;+\;\sum_{T\in T^h_{\partial\Omega}}\frac{\gamma_D\eta}{h_T}\|\bm v_h\|_{\partial\Omega\cap T}^2
        \;+\;|\bm v_h|_{s,g}^2 .
    \]
    The volume term is bounded below by $c_L\|\operatorname{sym}\Grad\bm v_h\|_\Omega^2$ (Assumption~\ref{ass:ellipticity}(a), pointwise, hence
    unchanged), and the penalty term by $\gamma_D\eta_{\min}\sum_T\|h_T^{-1/2}\bm v_h\|_{\partial\Omega\cap T}^2$. For the
    consistency Nitsche term we apply, on each cut cell $T$, Cauchy--Schwarz and Lemma~\ref{lem:loctrace} to $\bm
        w_h=\Grad\bm v_h$, then Young's inequality with parameter $\varepsilon>0$:
    \[
        2\!\int_{\partial\Omega\cap T}\!\!(\mathbb L\Grad\bm v_h\,\bm n)\!\cdot\!\bm v_h\,dS
        \;\le\;
        \varepsilon\,C_L^2 C_T\,\|\Grad\bm v_h\|_T^2
        \;+\;\frac{1}{\varepsilon}\,\|h_T^{-1/2}\bm v_h\|_{\partial\Omega\cap T}^2 .
    \]
    Summing over $T\in T^h_{\partial\Omega}$ gives $\varepsilon C_L^2 C_T\|\Grad\bm v_h\|_{\Omega_T}^2$ for the first
    group, which Lemma~\ref{lem:locext} absorbs into $\|\Grad\bm v_h\|_\Omega^2$ and $\gamma_A^{-1}|\bm v_h|_{s,g}^2$;
    Korn's inequality \eqref{eq:korn}, posed on the fixed physical domain $\Omega$ and hence unaffected by the
    refinement, with its boundary term dominated by the local-size penalty term since $h_T\le h_0$, then converts
    $\|\Grad\bm v_h\|_\Omega^2$ into $\|\operatorname{sym}\Grad\bm v_h\|_\Omega^2$ and the penalty term, exactly as in
    the proof of Proposition~\ref{prop:coercivity}.
    Choosing $\varepsilon$ small enough that $\varepsilon C_L^2 C_T C_G\max\{C_K, C_K h_0, \gamma_A^{-1}\}<\min\{c_L,\,1\}$ (the $\gamma_A^{-1}$ entry
    ensuring the ghost-penalty portion of the absorbed term leaves a positive fraction of
    $|\bm v_h|_{s,g}^2$) and tracking the
    $1/\varepsilon$ factor against $\gamma_D\eta_{\min}$ leaves a positive fraction of all three terms of
    \eqref{eq:gradednorm}, after a final application of \eqref{eq:korn} as in Proposition~\ref{prop:coercivity}; this
    is coercivity with $c_s=c_s(c_L,C_L,\eta_{\min},C_K,h_0,p,\sigma,c_N,L)$.

    Continuity is the same four-term split: the volume term by $C_L$ and Cauchy--Schwarz; each Nitsche term by
    Cauchy--Schwarz, Lemma~\ref{lem:loctrace} cell-wise and Lemma~\ref{lem:locext}, giving $C_L\sqrt{C_T}\,\|\Grad\bm
        u_h\|_{\Omega_T}\big(\textstyle\sum_T\|h_T^{-1/2}\bm v_h\|^2_{\partial\Omega\cap T}\big)^{1/2} \le C\,|||\bm u_h|||_g\,|||\bm v_h|||_g$; and the ghost-penalty
    term by Cauchy--Schwarz directly, $|\bm u_h|_{s,g}|\bm v_h|_{s,g}\le|||\bm u_h|||_g\,|||\bm v_h|||_g$. Summing gives
    $C_s$ with the stated dependence. No estimate used quasi-uniformity: every inequality lived on a full cell, a
    face-neighbour pair, or a chain of at most $L$ such pairs, with a constant fixed by (G1)--(G3).
\end{proof}

\begin{theorem}[Quasi-optimality on the graded family]
    \label{thm:gradedqo}
    Under the assumptions of Proposition~\ref{prop:gradedcoercivity}:
    \begin{enumerate}
        \item[\textnormal{(i)}] each linearised tangent problem $A_h(\Delta\bm u,\bm v_h)=\mathcal
                  F_{\text{Total}}(\bar{\bm u},\bm v_h)$ has a unique solution in $\mathbb V_h$, and the C\'ea
              estimate $|||\bm u-\bm u_h|||_g\le (C_s/c_s)\inf_{\bm v_h\in\mathbb V_h}|||\bm u-\bm v_h|||_g$
              holds for the linear elasticity model;
        \item[\textnormal{(ii)}] for a regular branch of the finite-strain problem, let $r_h\to0$ and set
              $\mathcal B_{g,h}=\{\bm v_h:|||I_h\bm u-\bm v_h|||_g\le r_h\}$. Assume on these neighbourhoods the local-size
              analogues of Theorem~\ref{thm:brr}: the assembled tangent is uniformly coercive with constant $c_s>0$;
              $F_h$ is Fr\'echet differentiable and $DF_h$ has Lipschitz modulus $L_{g,h}$ with
              $L_{g,h}r_h\le\rho$, where $\rho$ is sufficiently small independently of the cut and the refinement
              level; and the Jacobian defect is at most $\vartheta c_s$ for a fixed $\vartheta\in[0,1)$. Assume also
              the local-size consistency and quasi-optimal interpolation estimates, with
              $\inf_{\bm v_h\in\mathbb V_h}|||\bm u-\bm v_h|||_g=o(r_h)$. Then, for all sufficiently fine levels, the
              discrete problem has a unique solution in $\mathcal B_{g,h}$, with
              \begin{equation}
                  \label{eq:gradedqo}
                  |||\bm u-\bm u_h|||_g \;\le\; C\,\inf_{\bm v_h\in\mathbb V_h}|||\bm u-\bm v_h|||_g .
              \end{equation}
    \end{enumerate}
    The constants depend only on the uniform constants in these hypotheses and on the regular solution, never on the
    cut configuration or the refinement level.
\end{theorem}

\begin{proof}
    Coercivity and continuity in $|||\cdot|||_g$ (Proposition~\ref{prop:gradedcoercivity}) make $A_h$ a
    uniformly bounded, uniformly coercive bilinear form on $(\mathbb V_h,|||\cdot|||_g)$. For (i),
    Lax--Milgram gives existence, uniqueness and the stability bound, and C\'ea's lemma the quasi-optimal
    estimate; the abstract argument uses only the coercivity/continuity constants and Galerkin orthogonality
    of the consistent form, none of which refers to the global mesh size. For (ii), uniform coercivity and the
    Jacobian-defect bound give
    $\|DF_h(\bm v_h)^{-1}\|\le[(1-\vartheta)c_s]^{-1}$ on $\mathcal B_{g,h}$. The derivative Lipschitz estimate then
    makes the frozen-derivative map
    $G_h(\bm v_h)=\bm v_h-DF_h(I_h\bm u)^{-1}F_h(\bm v_h)$ a contraction whenever $\rho$ is sufficiently small.
    The consistency estimate and
    $\inf_{\bm v_h}|||\bm u-\bm v_h|||_g=o(r_h)$ ensure that $G_h$ maps $\mathcal B_{g,h}$ into itself for all
    sufficiently fine levels. Banach's fixed-point theorem gives the unique discrete solution in that ball, and the
    Taylor-remainder argument of Theorem~\ref{thm:brr}, with
    $L_{g,h}|||\bm u_h-I_h\bm u|||_g\le L_{g,h}r_h\le\rho$, gives \eqref{eq:gradedqo}.
\end{proof}

The local-size cut trace inequality (Lemma~\ref{lem:loctrace}), the ghost-penalty extension property
(Lemma~\ref{lem:locext}), the Nitsche coercivity of $A_h$ in $|||\cdot|||_g$
(Proposition~\ref{prop:gradedcoercivity}) and the local interpolation estimate \eqref{eq:locinterp} establish
Assumption~\ref{ass:gradedstab} with constants independent of the cut and of the refinement level, and therefore give
the linear quasi-optimality conclusion of Theorem~\ref{thm:gradedqo}\textnormal{(i)}. For the nonlinear result,
Proposition~\ref{prop:gradedcoercivity} supplies the local counterpart of (J0) if its material hypotheses hold uniformly
on $\mathcal B_{g,h}$; differentiability, the complete derivative Lipschitz bound, and the Jacobian-defect estimate are
the additional explicit hypotheses in part~(ii), not consequences of mesh geometry. The only properties of the refined family used to establish the geometric ingredients were shape-regularity (G1), the bounded
neighbour-size ratio (G2), and the bounded active-layer width (G3); the analysis is therefore agnostic to whether the
local sizes are produced by the a priori grading \eqref{eq:grading} or by adaptive refinement, as claimed in
Section~\ref{sec:gradinginherit}. Two ingredients of Section~\ref{sec:theory} are \emph{not} purely local. Korn's
inequality \eqref{eq:korn} is posed on the fixed physical domain $\Omega$, so it is independent of the mesh altogether
and enters the graded proof unchanged. The Friedrichs/Poincar\'e inequality of Lemma~\ref{lem:discrete}(ii), used for
the conditioning estimate, is
not needed here: quasi-optimality requires only coercivity and continuity in the energy norm, both of which we proved
cell-by-cell (up to the mesh-independent Korn constant). Consequently Assumption~\ref{ass:gradedstab} may be read as a theorem under (G1)--(G3); the nonlinear
application in Theorem~\ref{thm:cutinherit} additionally remains conditional on the local-size hypotheses of
Theorem~\ref{thm:gradedqo}\textnormal{(ii)}.

\end{document}